%% file: MomRhoAgg.tex
\documentclass[11pt, a4paper]{article}

\input{packages_notes}

\input{commandes_guillaume}

\begin{document}

\title{Robust machine learning by median-of-means : theory and practice}
\author{G. Lecué and M. Lerasle}
\maketitle
\begin{abstract}
We introduce new estimators for robust machine learning based on median-of-means (MOM) estimators of the mean of real valued random variables. These estimators achieve optimal rates of convergence under minimal assumptions on the dataset. The dataset may also have been corrupted by outliers on which no assumption is granted. We also analyze these new estimators with standard tools from robust statistics. In particular, we revisit the concept of breakdown point. We modify the original definition by studying the number of outliers that a dataset can contain without deteriorating the estimation properties of a given estimator. This new notion of \emph{breakdown number}, that takes into account the statistical performances of the estimators, is non-asymptotic in nature and adapted  for machine learning purposes. We proved that the breakdown number of our estimator is of the order of \textit{number of observations} * \textit{rate of convergence}. For instance,  the breakdown number of our estimators for the problem of estimation of a $d$-dimensional vector with a noise variance $\sigma^2$ is $\sigma^2d$ and it becomes $\sigma^2 s \log(ed/s)$ when this vector has only $s$ non-zero component. Beyond this breakdown point, we proved that the rate of convergence achieved by our estimator is \textit{number of outliers} divided by \textit{number of observations}.

Besides these theoretical guarantees, the major improvement brought by these new estimators is that they are easily computable in practice. In fact, basically any algorithm used to approximate the standard Empirical Risk Minimizer (or its regularized versions) has a robust version approximating our estimators. On top of being robust to outliers, the ``MOM version" of the algorithms are even faster than the original ones, less demanding in memory resources in some situations and well adapted for distributed datasets which makes it particularly attractive for large dataset analysis. As a proof of concept, we study many algorithms for the classical LASSO estimator. It turns out that the original algorithm can be improved a lot in practice by randomizing the blocks on which ``local means" are computed at each step of the descent algorithm. A byproduct of this modification is that our algorithms come with a measure of \textit{depth} of data that can be used to detect outliers, which is another major issue in Machine learning.

\end{abstract}

\section{Introduction}\label{sec:intro}
Recent technological developments have allowed companies and state organizations to collect and store huge datasets. 
These yield amazing achievements in artificial intelligence such as self driving cars or softwares defeating humans in highly complex games such as chess or Go. 
In fact, most big organizations have realized that data will have a major role in the future economy. 
Most companies in banks, electric or oil companies, etc. have a digital branch developing new data-based services for customers and new companies even build all their business on data collected on Twitter, Facebook or Google.

Big datasets have also challenged scientists in statistics and computer science to develop new methods. 
In particular, Machine Learning has attracted a lot of attention over the past few years. 
As explained in \cite{Shalev-Schwarz}, machine learning can be seen as a branch of statistical learning dealing with large datasets where any procedure besides presenting optimal statistical guarantees should be provided with at least a tractable algorithm for its practical implementation. 
This new constraint raised several interesting problems while revisiting the classical learning theory of Vapnik \cite{MR1641250}. 
In particular, to name just a few, algorithms minimizing convex relaxations of the classical empirical $0-1$-risk counting the number of misclassification have been proposed and studied over the last few years. More generally, optimization algorithms are now routinely used and analyzed by statisticians.

\subsection{Corruption of big datasets}

Our theoretical understanding of many classical procedures of machine learning such as LASSO \cite{MR1379242} heavily rely on two assumptions on the data, that should both be i.i.d. and have subgaussian behaviors. As noted in \cite{MR3466175}, ``data from real-world experiments oftentimes tend to be corrupted with outliers and/or exhibit heavy tails". For example, in finance, heavy-tailed processes are routinely used and, in biology or medical experiments datasets are regularly subject to some corruption by outliers. These outliers are even in some applications the data of actual interests, one can think of fraud detections for example. The need for robust procedures in data science can be appreciated, for instance, by the recently posted challenges on ``kaggle'', the most popular data science competition platform. The 1.5 million dollars problem ``Passenger Screening Algorithm Challenge'' is about to find terrorist activity from 3D images. The ``NIPS 2017: Defense Against Adversarial Attack'' is about constructing algorithms robust to adversarial data.

\subsubsection{Current ML solutions are highly sensitive to dataset's corruption}

It is easy to check that most routinely used and theoretically understood algorithms break down when even a single ``outlier" corrupts the dataset. This is the case of the standard least-squares estimator or its $\ell_1$ penalized version known as LASSO for example as one can see from the simulation in Figure~\ref{fig:robustness}

\begin{figure}[!h]
    \centering
    \includegraphics[width=0.4\textwidth]{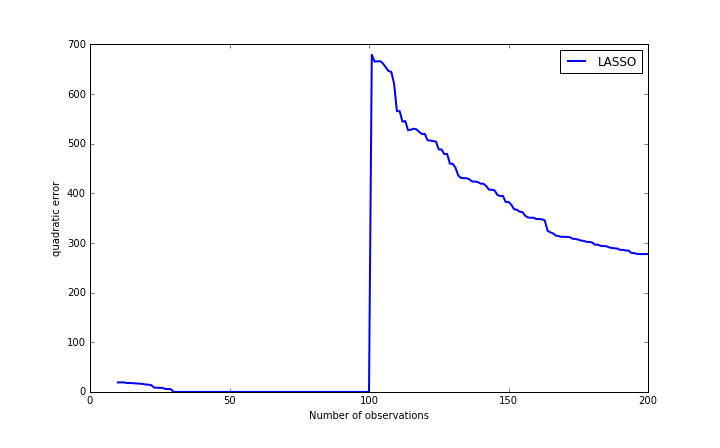}
    \includegraphics[width=0.4\textwidth]{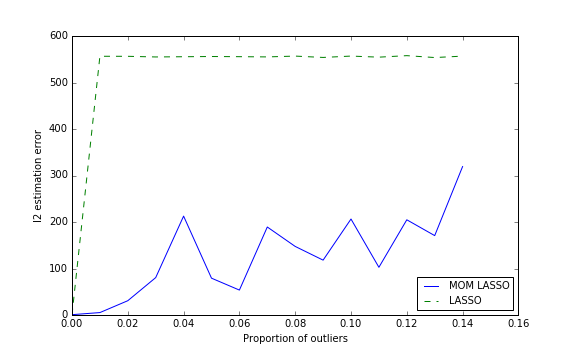}
    \caption{(left): Estimation error of the LASSO after one outliers was added at observation number $100$. (right) Estimation error versus proportion of outliers for LASSO and its "Median of Means (MOM)" version.}
    \label{fig:robustness}
\end{figure}

LASSO is not an isolated example, most algorithms are actually designed to approximate minimizers of penalized empirical loss and empirical means of unbounded functions are highly sensitive to corrupted data. 
In statistics, building robust estimators has been an issue for a long time at least since the work of John Tukey, Frank Hampel and Peter Huber. It has yield interesting ``robust" alternatives to the standard maximum likelihood estimator, one can refer to \cite{MR2488795} for an overview on robust statistics and to Baraud, Birgé and Sart \cite{MR3595933} or Baraud and Birgé \cite{MR3565484} see also \cite{MR3052407,MR2906886} for deep new developments in robust statistics. 
%Their new $\rho$-estimators actually matches the classical MLE in regular nice environments where it is optimal while being insensitive to model misspecifications or the presence of a few ``outliers" that can be as agressive as possible.

\subsubsection{Resist or detect outliers : a raising challenge in ML}

There are mainly two types of outliers in practice : those corrupting a dataset which are not interesting (outliers can appear in datasets due to storage issues, they can also be adversarial data as fake news, false declarative data, etc.) and those that are rare but important observations like frauds, terrorist activities, tumors in medical images,... 
A famous example of the latter type of ``outlier" discovered unexpectedly was the ozone hole \cite[p.2]{maronna2006robust}. 
In the former case, the challenge is to construct predictions as sharp as if the dataset was clean and in the latter, the main question is to detect outliers.

Of course, any dataset is preprocessed to be ``cleaned" from outliers. 
This step, called ``data jujitsu'', ``data massage'', ``data wrangling'' or ``data munging'' usually takes $80\%$ of data scientists time. 
It requires an important quantity of expertise to guide data scientists \cite{barnett1994outliers}. 
Some platforms like ``Amazon mechanical turk'' allow to preprocess collectively datasets with millions of lines. 
Even when done with maximal care, some contamination always affect modern datasets, because outliers are particularly hard to detect in high dimension and because the distinction between outliers and informative data is often delicate in practice.

For these reasons, robustness has recently received a lot of attention in machine learning. Theoretical results have been obtained when ``outliers" appear because the underlying distribution exhibits heavier tails than subgaussian distributions, producing anomalies. In this case, many attempts have been made to design estimators presenting subgaussian behaviors under minimal assumptions for this to make sense. This program was initiated by Catoni \cite{MR3052407} for the estimation of the mean on the real line and Audibert and Catoni \cite{MR2906886} in the least-squares regression framework. To the best of our knowledge, besides for the estimation of the mean of a real valued random variable where median-of-means estimators \cite{MR1688610, MR855970, MR702836} provide a computationally tractable estimator satisfying optimal subgaussian concentration bounds \cite{MR3576558}, the other estimators are either suboptimal see for example \cite{MR3378468} for the estimation of the mean of a vector in $\R^d$ or \cite{MR3405602} for more general learning frameworks, or completely untractable see for example \cite{LugosiMendelson2017-2, LugosiMendelson2017, MOM1} for some examples in regression frameworks.

\subsection{Robustness : state of the art}

%\subsubsection{The multiple faces of robustness}

The book \cite{maronna2006robust} traces back the origins of robust statistics to the fundamental works of John Tukey \cite{tukey1960survey,tukey1962future}, Peter Huber \cite{MR0161415,huber1967behavior} and Frank Hampel \cite{hampel1971general,hampel1974influence}.  Tukey asked the following  question for the location problem of a Gaussian distribution ``what happens if the distribution slightly deviates from the normal distribution?''. Tukey's 1960 example shows the dramatic lack of distributional robustness of some classical procedures like Maximum Likelihood Estimator (MLE) \cite[Example~1.1]{huber2009robust}. A related question was asked by Hodges \cite{MR0214251} about ``tolerance to extreme values''.

In fact, one doesn't define one robustness property in general. Estimation, prediction or more generally any statistical properties like consistency, asymptotic normality, minimax optimality, etc. require assumptions (cf. the ``no free lunch theorem'' \cite{DGL:96}). Such assumptions are naturally questionable on real databases. What properties of an estimator pertain when one or several assumptions are not satisfied? As there are several types of assumptions, one can define several types of robustness. For example, Tukey's question is about robustness to a misspecification of the Gaussian model while Hodges question can be understood as robustness to contamination of the dataset by large outliers or robustness to heavy-tailed distributions in the model that may have caused the appearance of large data.

\subsubsection{Natural robustness properties of learning procedures}

Robustness issues are particularly sensitive when data and risk are unbounded. To understand that, suppose first that the loss $\ell$ is the classical (bounded) $0-1$ loss in classification: $\ell_f(x,y)={\bf 1}_{y\ne f(x)}$. Assume that the dataset $\cD$ is made of $N-|\cO|$ i.i.d. data $(X_i,Y_i)_{i\in\cI}$ and $|\cO|$ ``outliers" $(X_i,Y_i)_{i\in\cO}$ that can be anything. The empirical risk
\[
P_N\ell_f=\frac1N\sum_{i=1}^N\ell_f(X_i,Y_i)
\]
satisfies for all $f:\cX\to \left\{-1, 1\right\}$ (where $\cX$ is  the space where the $X_i$ take their values and the $Y_i\in\{-1, 1\}$)
\begin{equation}\label{eq:BLNopb}
 |(P_N-P_{\cI})\ell_f|\leqslant \frac{2|\cO|}N,\qquad \text{where}\qquad P_{\cI}\ell_f=\frac1{|\cI|}\sum_{i\in \cI}\ell_f(X_i,Y_i)\enspace.
\end{equation}
If the fraction $|\cO|/N$ is small enough, the ERM over a class $\cF$ of functions from $\cX$ to $\{-1, 1\}$
\[
\widehat{f}_N=\argmin_{f\in\cF}P_N\ell_f
\]
performs theoretically as well as the ERM based only on the ``good" data $(X_i,Y_i)_{i\in\cI}$
\[
\widehat{f}_{\cI}=\argmin_{f\in\cF}P_{\cI}\ell_f\enspace,
\]
which is statistically optimal in some problems \cite{DGL:96}. In that case, the number of outliers $|\cO|$ has to be less than $N$ times the \textit{rate of convergence of $\widehat{f}_{\cI}$} -- this is a condition we will meet later for our estimators.   

The problem with $\widehat{f}_N$ is that minimizing the empirical loss is NP-hard. To overcome this issue practitioners have introduced several convex relaxation of the $0-1$ loss. As a result, what is actually computed is a minimizer of empirical risk associated with an unbounded loss function $\ell$ (like the square, logistic, hinge or exponential losses) over an unbounded set of functions (as the set of linear functions $\{\inr{\cdot, t},\ {t\in\R^d}\}$). This actual minimization problem is much more sensitive to outliers (no bounds can be provided between the empirical processes such as in \eqref{eq:BLNopb}) and \emph{practical} ERM (actually its proxies) are completely wrong even if the \emph{theoretical} ERM performs well. In the following, we assume that $Y,\cF$ are unbounded and $\ell$ is the (unbounded) square loss : $\ell_f(x,y)=(y-f(x))^2$ that is known to be sensitive to outliers.

%Another typical example of robustness that is currently tackled by learning procedures is what in statistics, is referred to as robustness to \textit{model misspecification}.  
Classical estimators, such as the maximum likelihood estimator (MLE) are designed for some very specific choice of statistical model and are therefore extremely model dependent : in the set of all Gaussian distributions with mean $\mu\in\R$ and variance $\sigma>0$, the MLE of $\mu$ is the empirical mean while it is the empirical median in the set of all Laplace distributions with position parameters $\mu\in\R$. 
%The typical example of such estimator is the MLE which can be very different from one model to another (it can be the empirical mean, the empirical median, the maximum or the minimum of the data, etc.). 
%Given an estimator which is known to have ``good'' statistical properties when a specific statistical model is granted, how does this estimator behave when data don't follow exactly this model? 
%This misspecification problem is the problem of robustness to the ``statistical model assumption''. 
Tukey's example showed that MLE may perform very poorly under model misspecification. ML and learning theory tackle this issue by not assuming a statistical model on data. Instead of imposing strong assumptions allowing to fit all data distribution, learning procedures typically focus on a simpler task that can be achieved under much weaker restrictions on the randomness of data. 

The typical example is as follows. A dataset $\cD:=\{(X_i,Y_i)_{i=1}^N\}$ of $N$ independent and identically distributed (i.i.d.) variables with values in $\R^d\times \R$ is given. In statistics, the classical (parametric) approach is to start with a statistical model, i.e. a family $(P_{\theta})_{\theta\in\Theta}$ indexed by some subset $\Theta\subset\R^d$ of distributions that should contain that of $(X_1,Y_1)$, or at least that of $Y_1|X_1$. In the Gaussian linear regression model \cite{stigler1986history}, one assumes that $Y_i=\inr{X_i, t^*}+\sigma\zeta_i$, where $t^*\in\R^{d-1}$ and $\sigma>0$ are unknown parameters and $(\zeta_i)_{i=1}^N$ are i.i.d. standard Gaussian random variables  independent of $(X_i)_{i=1}^N$. The problem is to estimate the unknown parameters $\theta=(t^*,\sigma)\in\R^d$ given the dataset $\cD$ \emph{assuming that the distribution of $(Y_1|X_1)$ belongs to the set of Gaussian distributions $(P_{\theta})_{\theta\in\Theta}$}. 

In Learning Theory, no assumption on the conditional expectation $\E[Y_1|X_1]$ nor any particular form for the conditional distribution of $Y_1|X_1$
%, nor independence between the noise $Y_1-f^*(X_1)$ and the design $X_1$ 
is assumed. 
For instance, $Y_i$ can be any function of $X_i$, it may also be independent of $X_i$. Given a new input $X$, one still wants to guess the associated output $Y$. To proceed, a class of functions $F$ is proposed containing functions $f\in F$ used to guess the value $f(X)$ of $Y$. It is not assumed that the regression function $x\mapsto \E[Y|X=x]\in F$, we just hope that the ideal choice $f^*$ in $F$ allows to build a good guess $f^*(X)$ of $Y$. The learning task is to provide from the dataset a good approximation of $f^*$. If $f^*(X)$ is actually a good guess of $Y$, then so will be any good approximation of $f^*$ applied in $X$, if not then we simply end up with a bad guess of $Y$. 
%The main concern in Machine Learning is then to estimate the best approximation of $\E[Y|X=\cdot]$ in $F$. 
The function $f^*$ providing the best guess $f^*(X)$ of $Y$ is called an oracle and the difference $Y-f^*(X)$ is called the noise.
% or to predict $Y$ as well as $f^*(X)$. Note that f
Assume that $F=\{\inr{\cdot, t}: t\in\R^d\}$ is the class of linear functionals. If the regression function $x\mapsto \E[Y|X=x]$ is assumed to be in $F$ then 
%the best $L^2$ approximation of $Y$ is given by $f^*(\cdot)=\inr{\cdot, t^*}$. In this case, 
the tasks in Machine Learning and Statistics are essentially the same because the oracle equal the regression function. The difference is that learning procedures 
%can be appreciated when the statistical model doesn't hold, a model-based estimator doesn't offer any guarantee in this case while a learning procedure still 
consider the estimation of $f^*$ that may not be the regression function without assuming a parametric family of distributions for the noise. As such, the misspecification problem is naturally tackled in Machine Learning.
% and no independence between the noise and the design is ever required. 
% Actually, we even go further in the misspecification problem in this paper by going beyond the i.i.d. assumption.
% More precisely, our procedure only needs a large proportion of the data to induce essentially the same $L_2$ metric to get optimal performance. In particular, it means that the learning task can be optimally achieved if (most of) data only have equivalent $L_2$ moments over the class $F$, even in the presence of other, possibly adversarial data.

%\subsubsection{Robustness is an even bigger issue in ML}
%discuter brievement l'exemple de l'ERM avec perte 0-1 : en théorie, c'est cool car la perte est bornée mais en pratique, ça va nettement moins bien car on doit minimiser des relaxations convexes non bornées, donc complètement perturbées par des outliers...

\subsubsection{Robustness to heavy-tailed data}

%The other typical assumption in statistics is on the distribution of the noise. The prominent one since Gauss's analysis of the regression model  is to assume a Gaussian distribution with mean zero, independent of the design $X_i$. 
The Gaussian assumption on the noise $Y_1-f^*(X_1)$ in statistics is usually weakened in learning theory to a subgaussian assumption (the $L_p$-moments of the noise are bounded by those of a Gaussian variable). This subgaussian assumption is already much more flexible and essentially allows to extend the use of least-squares estimators and their penalized versions to a much broader class of possible distributions for the noise. Going beyond the subgaussian assumption is in general way more technical. One way to proceed is to assume a subexponential assumption -- random variables whose $L_p$ moments are smaller than those of an exponential variable-- as in the Bernstein's condition in \cite{MR1739079}. More recently, the subexponential assumption has also been relaxed into $L_p$-moment assumptions. Least-squares estimators have been studied as well as some modifications less sensitive to large data for instance in \cite{MR3052407,MR2906886,LM_reg1,LM_lin_agg}. It appears that the statistical properties of least-squares estimators are deteriorated when the noise satisfies weaker assumptions \cite{LM_lin_agg}. A modification of the ERM based on Le Cam's estimation by tests principle \cite{MR0334381,MR856411} has also been considered in \cite{LugosiMendelson2016,MOM1}. 
%(in fact, the rate of convergence remains that of the subgaussian case but the set of data on which such rate is achieved become smaller) . 

%As we work with the square loss, we will always grant an $L_2$ assumption on the noise, to give sense to the $L_2$-risk. 
%Besides this, we look for an estimator of $f^*$ that is not based on a statistical model, whose performance are reasonable or even statistically optimal (statistical optimality is discussed later) without independence between noise and design, with nothing but bounded $L_2$ moments granted on the noise. 
%One aim of this work is to show such ideal learning procedure can actually be designed.

\subsection{Robustness in regression}

%%%%%%Regression

Most of the robustness literature of the 1980s was on the regression problem \cite{davies1993aspects},  \cite[Section~7.12]{huber2009robust}. 
%The number of robust methods in regression ``has become so extensive that it is worse than bewildering, namely counterproductive'' \cite[p.195]{huber2009robust}. Nevertheless, all 
These methods make a distinction between the problem of robustness with respect to the outputs $(Y_i)_{i=1}^N$ on one side and with respect to the inputs $(X_i)_{i=1}^N$ on the other side \cite[Chapter~7]{huber2009robust}. Outliers in the inputs remain extremely difficult to handle while many solutions have been proposed to deal with outliers in the outputs.

\subsubsection{Outliers in the inputs in regression}

%``\textit{Regression poses some peculiar and difficult robustness problems}'' \cite{huber2009robust}. \
Even if some solutions have been proposed in particular cases, see for example \cite{MOFaser2000}, the problem of robustness with respect to outliers in the inputs $X_1, \ldots, X_N$ (also called \emph{leverage points}), has not been solved in general neither in theory nor in practice.
In the book \cite{huber2009robust}, a three steps approach is suggested to solve the robustness problem in regression
\begin{enumerate}
\item build redundancy into the design matrix $\bX$ if possible,
\item find routine methods when there are only few leverage points,
\item find analytical methods for identifying leverage points and, if possible, leverage groups.
\end{enumerate}
In other words, \cite{huber2009robust} suggest to make an important ``datacleaning'' step before any analysis for solving the leverage point problem. 
As already discussed, this may not be possible for nowadays large datasets. Therefore, designing procedures that can handle outliers such as leverage points is an important task that requires new ideas.
%, it seems that there is no solution to the robustness issues in this case.

% to quote \cite[p.195]{huber2009robust}
%``\textit{I must stress that any theory of robustness with regard to the carrier $\bX$ requires the specification (i) of a model for the carrier and (ii) of the small deviations from that model one is considering}''. 
%As in the analysis of big datasets, the size $p$ of the inputs often exceeds the number $N$ of outputs,  the presence of outliers seems unavoidable, even after careful preprocessing of data. 
In statistics, an elegant solution to this issue (and many others) called $\rho$-estimators \cite{MR3595933,BarBir2017RhoAgg2} has recently been proposed. However, $\rho$-estimators do not provide a satisfying solution for learning issues since they require a statistical model (even if almost no assumption besides a controlled massiveness are required on this model), they are specifically built for the Hellinger loss (although they seem to work extremely well for other risk function in some special cases) and they remain far from being computable in general which is dead-end for Machine Learning.

\subsubsection{Outliers in the outputs}

Robust procedures can be analyzed under the assumption that ``we can act as if the $x_{ij}$ are free of gross errors'' \cite[paragraph~3 in Section~7.3]{huber2009robust} ($x_{ij}$ being the $j$-th coordinate of $X_i$ for $i\in[N]$ and $j\in[d]$). 
%In other words, robust procedures are constructed under the assumption that only outliers in the $Y_i$'s are allowed. 
%One aim of our work is to show that there is no need to distinguish outliers in the outputs to the one in the inputs: we consider an entire line of data $(Y_i, X_i)$ as outliers or as informative. In particular, there is no need to identify \textit{leverage points} via ``data cleaning'', which is precisely the step we want to avoid for big data. 
%This result is true in ``regression analysis''  based on the self-influence parameters $(h_i)_{1\leqslant i\leqslant N}$ from the hat matrix $H$. 
%
In this approach, the square loss is replaced by a convex differentiable function $\rho$ and the estimator minimizes
\begin{equation}\label{eq:M-estimator}
 \sum_{i=1}^N \rho\left(Y_i-\inr{X_i, t}\right)\enspace.
 \end{equation}
If $\psi=\rho^\prime$, the estimator is also solution of the equation
 \begin{equation}\label{eq:equation_huber}
 \sum_{i=1}^N X_i\psi\left(Y_i-\inr{X_i, t}\right)=0\enspace.
 \end{equation} 
 This estimator is robust when $\psi$ is bounded since even gross errors in a few $Y_i$'s don't affect the right hand side of \eqref{eq:equation_huber}. Classical examples of $\psi$ functions include the Huber function $\psi(t) = t \min(1, c/|t|)$  or Tukey's bisquare function $\psi(t) = t(1-(t/c)^2)^2 I(|t|\leqslant c)$ for all $t\in\R$. In both cases, a tuning parameter $c$ needs to be specified in advance to ultimately balance between resistance to outliers and bias of the estimation. Earlier regression estimate go back to Laplace and its Least Absolute Deviation (LAD) or $L_1$-estimate that minimizes $\sum_{i}|Y_i-\inr{X_i, t}|$ yielding the median in the location problem \cite{dodge1987introduction}. These M-estimation strategies work well both in theory and practice \cite{huber2009robust,maronna2006robust}, but fail totally when even one $X_i$ has a gross error. In the regression problem $Y_i=\inr{X_i, t^*}+\xi_i$ where $\bX^\top\bX = I_{d\times d}$ and $\xi_1, \ldots, \xi_N$ are i.i.d. random errors such that $\E \psi(\xi_i)=0$, asymptotic results for M-estimators can be found in \cite{huber2009robust} when $d$ is fixed and $N$ goes to infinity. When $d$ is allowed to grow with $N$, some results have also been obtained but, as noted in \cite[p168]{huber2009robust}, under conditions that are, for all ``\textit{practical purposes worthless}'' since ``\textit{already for a moderately large number of parameters, we would need an impossibly large number of observations}''.  When the inputs are clean of outliers, subgaussian random variables, or when there is no design, non asymptotic results on minimizers of the Huber loss have also been obtained recently in the small and large dimension settings \cite{Fan1,Fan2}.
% However, it is conjectured that the asymptotic of these M-estimator should still be valid for $h\to 0$.
 
 Many alternative robust procedures can be found in the regression literature, among other, one can mention the least median of squared residual estimator \cite{rousseeuw1984robust}  and the least trimmed sum of squares estimator \cite{rousseeuw1983regression} which minimizes the sum of the $N/2+1$ smallest squared residuals. 
% They both have a breakdown point of almost $1/2$. There exists an adaptation of the Donoho-Huber's definition of breakdown point in non-linear regression in \cite{stromberg1992breakdown}. 
 Finally, \cite{hampel1975beyond} introduced Least Median of Squares in regression which naturally extends the median estimator in the location problem : 
\begin{equation*}
\hat t \in\argmin_{\theta\in\R^d} {\rm Median}\left(Y_1 - \inr{X_1, \theta},\cdots, Y_N- \inr{X_N, \theta} \right).
\end{equation*}
Its computational complexity rises exponentially with $d$ \cite[p.330]{hampel1986linear}, \cite{massart1986least}. 

%Moreover ``[...] a grossly aberrant $X_i$ exerts an overwhelming influence (``leverage'') also on a least absolute deviation fit''  \cite[p.285]{huber2009robust}. 

\subsubsection{Other challenges}

We discussed the problems of detecting outliers and resisting corruption of the inputs earlier. Besides these, robust statistic has left some interesting open questions that have to be addressed. 

Most results in robust statistics are asymptotic but the asymptotic approach may be quite misleading : ``\textit{
%The conclusion is that an analysis tailored to the requirements of the particular estimation problem is needed, taking the regression design into consideration. In any case, I must repeat that high leverage points constitute small sample problems; therefore 
[...] approaches based on asymptotics are treacherous, and it could be quite misleading to transfer insights gained from asymptotic variance theory or from infinitesimal approaches (gross error sensitivity and the like) by heuristics to leverage points with $h_i>0.2$}'' \cite[p.189]{huber2009robust}. In other words, robustness issues ultimately have to be tackled from a non-asymptotic point of view. Again, $\rho$-estimators \cite{MR3595933,BarBir2017RhoAgg2} treat this issue from a statistical perspective but cannot be applied in our learning setting.

The actual solution to get rid of some data a priori and work on what is supposed to be clean data seems quite dangerous. First, if ``outliers" are due to heavy-tailed data, this can induce some bias in the estimation. Second, data can be typical for the estimation of some means $P\ell_f$ but not for other choices of loss function $\ell'$ and/or function $f'$. An estimator $\widehat{P\ell_f}$ may even get better if one uses some data while another $\widehat{P\ell_f}'$ would get worse. The definition of outliers is thus not clear, it may depend on the estimators, the risk function and the parameter of interest. Therefore, if some subsampling is performed to get more robustness, it should be done with respect to the loss $\ell$ and  point $f$ and not on systematic preprocessing algorithms blind to the specific learning problem we want to solve.
% and the estimation procedure $\widehat{P\cdot}$.
%We should instead first think of the estimation problem and then design robust statistics solving this particular statistical issue. 

 When designing robust strategies for ML, one should keep in mind that these should be computationally tractable. To the best of our knowledge, no estimator robust to outlier in the outputs nor any general learning procedure robust to heavy-tailed data can currently be used in a high dimensional learning problems. \emph{There is currently no available algorithm for high dimensional learning whose good statistical performance don't break down in the presence of one outlier in the dataset}.

\subsection{Contributions}

%Mettre ce qu'on a fait à partir d'ici. Notre setup pour faire de l'apprentissage robuste/notre notion de breackdown point/ résumer notre contribution : nouvelle procédure d'apprentissage stat robuste par minimaximization (à la Audibert-Catoni) de tests MOM (à la Gab-Shasha). 
%\begin{itemize}
%\item Guaranties théoriques rock-solides :
%\begin{enumerate}
% \item performance minimax dans notre framework
% \item première étude d'un breackdown point en ML
%\end{enumerate}
%\item Passage à l'\emph{apprentissage machine} \textcolor{red}{insister fortement ici, c'est vraiment la première fois que c'est fait}
%\begin{enumerate}
% \item algos 
% \item simulations extensives
% \item pour nuancer : Pas d'étude théoriques des algos pour l'instant.
%\end{enumerate}
%\end{itemize}
%

In this paper, we propose a new estimator of the learning function 
\[
f^*=\argmin_{f\in F}P[(Y-f(X))^2]\enspace.
\]
This new estimator has minimax-optimal non-asymptotic risk bounds even when the dataset has been corrupted by outliers (that may have corrupted indifferently the inputs, the outputs or both of them) and when informative data (those that are not outliers) only satisfy somehow minimal assumptions relating their $L_2$ geometry on $F$ to that of $P$ (see Sections~\ref{sec:OUIDec} and \ref{sec:Ass} for details). This estimator solves therefore all challenges related to outliers in the learning theory literature including ``adversarial outliers" -- which are data satisfying no assumptions -- and ``stochastic outliers" -- which are informative data that may look like outliers because of some weak concentration properties due to heavy tail distributions.

%Furthermore, we revisit the concept of breakdown point, limiting the class of estimators to those with statistically optimal risk bounds (see Definition~\ref{def:BP_learning_theory}), making it more suitable for ML goals. We compute the breakdown point of our new estimator. This new definition naturally gives rise to many interesting new conjectures on the optimality of robust procedure from this learning point of view. In particular, the optimality of the breakdown point we have is an open question (see however our first investigations in Section~\ref{sec:minimax_optimality_of_}).

We also show that our estimator can be used for ML purposes by providing several algorithms to compute proxies of our estimators and illustrate their performance in an extensive simulation study (see Section~\ref{sec:simulation_study}). This opens several interesting conjectures for example, to study and compare theoretically convergence of these algorithms, 
%in computer science to optimize our codes that can be found at [REF] and for applications 
and to automatically calibrate our estimators.
% that seems to be calibrated differently depending on the objective one seeks (resistance to outliers or detection of these). 

\subsubsection{Going beyond the i.i.d. setup: the $\cO\cup \cI$ framework}\label{sec:OUIDec}

We propose the following setup that allows the presence of outliers and relax the standard i.i.d. setup.
%The last usual assumption is that data are i.i.d.. The independence assumption is sometimes relaxed into other relationships between data like weak dependence or other temporal/graphical loss of memory properties. 
We assume that data are partitioned in two groups (unknown from the statistician), more precisely $\{1,\ldots,N\}=\cO \cup \cI$, with $\cO\cap\cI=\emptyset$ 
\begin{itemize}
 \item data $(X_i, Y_i)_{i\in\cO}$ ($\cO$ stands for ``outliers'') are not assumed to be independent nor independent to the other ones $(X_i, Y_i)_{i\in\cI}$, they may not be identically distributed, in fact \emph{nothing is required on these data}; they can even be adversarial in the sense that may have been designed to make the learning task harder,
  \item data $(X_i, Y_i)_{i\in\cI}$ (called \textit{informative}) are the only data on which we can rely on to solve our learning task. In what follows, we only need very weak assumptions on those data to make MOM estimators achieving optimal statistical properties. In particular, we will always assume that the informative data are independent.
\end{itemize}

%Outliers could be data with a gross error coming from any technical or adversarial reasons and the resulting line of data is at best meaningless and at worse adversarial, trying to disturb the statistical problem we want to solve. They might also be valuable, for instance, an ancestral very valuable but unique data or in frauds detection and terrorist activity. Instead of looking for all types of outliers, we simply assume nothing on these data, except that there is not too many of these.
%As recalled previously, robustness to this type of outliers (and, sometimes, detection of these outliers) is paramount in nowadays datascience and this is the one that we are now discussing more precisely but for the moment let us precise that 

As we allow the presence of outliers, our procedure will be in a sense robust to the i.d. (identically distributed) assumption. This robustness goes actually further, since even informative data won't be assumed identically distributed. Instead, our procedure will be shown to work under the much weaker 
%and way more natural 
requirement that the distributions of $(X_i, Y_i)$ for all $i\in\cI$ induce an equivalent $L_2$ metric over the class $F$ and equivalent covariance between functions in $F$ and the output. If various distributions have equivalent first and second moments for all functions $f\in F$, our estimators will exhibit the same optimal performance as if they were i.i.d. Gaussian with similar first and second moments on all $f\in F$. Our new assumption feels more natural than the usual i.i.d one as distributions inducing the same ``risk geometry" on $F$ seem interesting to solve our learning task and any other higher moment assumption above $2$ seems unnatural. However, the standard ERM once again would fail to achieve this goal as the behavior of the supremum over subclasses of $F$ of the empirical process used to bound the risk of the ERM depends in general on higher moments of these distributions.

\subsubsection{Quantifying robustness by breackdown point}

%Many tools have been proposed to quantify robustness (cf. \cite{huber2009robust,maronna2006robust}). 
%In this paper, we propose a new notion of \textit{breakdown point}.
%, it has emerged in several problems including the regression problem. 
%
%
The breakdown point of an estimator \cite{hampel1968contribution,donoho1983notion,huber2009robust} is the smallest proportion of corrupted observations necessary to push an estimator to infinity \cite[p.1809]{donoho1992breakdown}, \cite[p.279]{huber2009robust}. The notion has been introduced by Hampel \cite{hampel1968contribution} in his 1968 Ph. D. thesis. In \cite{donoho1983notion,donoho1982breakdown}, the notion was revisited with the following non-asymptotic definition, see also \cite{siegel1982robust, hampel1973robust}. Let $\cD_{\cI}$ denote a dataset
%, at which the breakdown point is to be evaluated. L
and let $T$ be an estimator. 
%Consider adjoining to $\cD_{\cI}$ . 
If there exists a strategic / malicious / adversarial choice of another dataset $\cD_{\cO}$ such that $T(\cD_{\cI}\cup \cD_{\cO})-T(\cD_{\cI})$ is arbitrarily large, the estimator is said to break down under the contamination fraction $|\cD_\cO|/(|\cD_\cI|+|\cD_\cO|)$. The \textbf{Donoho-Hampel-Huber breakdown point $\epsilon^*(T, \cD_{\cI})$} is then defined as the smallest contamination fraction under which estimator $T$ breaks down:
\begin{equation}\label{eq:breakdown_point}
 \epsilon^*(T, \cD_{\cI}) = \min_{m\in\bN}\left\{\frac{m}{|\cD_\cI|+m}: \sup_{\cD_{\cO}:|\cD_\cO|=m}|T(\cD_{\cI}\cup \cD_{\cO})-T(\cD_{\cI})| = \infty\right\}.
 \end{equation} 

 For the estimation of the mean $\mu$ of a random variable from $N$ i.i.d. observations, the empirical mean can be made arbitrarily large by adding a single ``bad'' observation. The breakdown point of the empirical mean estimator is therefore $1/(N+1)$ and this is the worst possible value for an estimator. The empirical mean is not robust to outliers, it performs very poorly in corrupted environments.  On the other hand, the empirical median has a breakdown value of $1/2$ for the location problem. 
 
% This is the best possible value among all translation invariant estimators \cite{donoho1982breakdown,davies2005breakdown}. This restriction to translation invariant estimators cannot be completely relaxed to the set of all estimators. 
 
Constant estimators have breakdown points equal to $1$. To avoid these degenerate examples, one restricts the set of estimators to consider optimality from the breakdown point of view. For instance, in the regression problem, one can consider \textbf{equivariant} estimators \cite{davies1993aspects,davies2005breakdown} and \cite[p.92]{maronna2006robust}, that is estimators $T$ invariant by affine transformation of the data: for all $a\in\R^d$ and $\lambda\in\R$,
\begin{equation}\label{eq:equivariant}
T((X_i, \lambda(Y_i+\inr{X_i, a}))_{i=1}^N) = \lambda \left(T((X_i, Y_i)_{i=1}^N) + a\right).
\end{equation}
Among equivariant estimators, the common sense heuristic applies : ``when more than half data are bad, one cannot learn from good ones" \cite{davies2005breakdown}. In particular, among equivariant estimators, $1/2$ is the best possible breakdown point for an estimator and the empirical median is optimal from this perspective.
%
%Estimators which are not equivariant  are not considered, which avoids dealing with uninteresting degenerate cases such as constant estimators that have breackdown points equal to $1$. 
When the parameter of interest belongs to a bounded set, one can adapt the definition of breakdown point as in the discussion papers of \cite{davies2005breakdown}.   
%
%For the scale problem, which is the problem of estimation of the variance $\sigma$ of a Gaussian distribution, it is proved that the empirical variance estimator has a breakdown point of order $1/(N+1)$ whereas the median of all absolute deviations from the median (MAD) has a breakdown value equal to $1/2$. 
Other generalization to quantile estimation can be found in \cite{rousseeuw1993alternatives}.

Similar results are much harder to obtain in higher dimensions, the main reason being the absence  of a natural notion of median or quantiles in dimension $d\geqslant2$. This problem gave birth to the construction of \textit{depths}, the most famous one being Tukey's depth \cite{tukey74a,tukey74b}. The notions of median, trimmed mean and covariance resulting from Tukey's depth yield robust estimators in high dimensions. For instance, in \cite{donoho1982breakdown,donoho1992breakdown}, the depth-trimmed mean and the deepest point are affine invariant estimators 
%(since Tukey's depth is affine invariant) 
with breakdown points $1/3$. It is even possible to achieve a breakdown point for the location problem in $\R^d$ close to $1/2$ by using a weight function based on Tukey's depth \cite{stahel1981breakdown,donoho1982breakdown}.

\subsubsection{A modified notion of breakdown point for ML}

We introduce a new concept of breakdown point for Learning. 
%This new definition measures together robustness to complete outliers and statistical optimality, which particularly fits nowadays datasets. 
The standard breakdown point \eqref{eq:breakdown_point} does not come with any estimation or prediction property. 
As already mentioned, this leads to consider restricted classes of estimators to study optimality from the breakdown point of view to avoid degenerate cases.
% (constant estimators never break down with the Donoho-Huber definition).  
%
Instead of imposing algebraic restrictions, our new definition focus on the risk of the procedures. 
%This is particularly relevant here as we seek for estimators: 1) statistically optimal 
%%(that is converging at minimax rate of convergence when data satisfy a Gaussian model) 
%2) robust to (possibly several) outliers in the dataset and 3) efficiently computable by some algorithm.
% -- but we leave apart that aspect of the learning problem for the time being). 
% We need to slightly adapt the definition of breakdown point if we want to take care of this notion of statistical optimality. 
% We suggest to extend the notion of breakdown point in the context of learning theory with the following definition. 
 
\begin{Definition}\label{def:BP_learning_theory}
Let $\delta\in(0,1)$, $\cR>0$, $N\geqslant 1$, $F$ be a class of functions from $\cX$ to $\R$ and $\cP$ be a set of distributions on $\cX\times\R$. Let $T:\cup_{n\geqslant 1}(\cX\times\R)^n\to F$ denote an estimator and let $\cD=\{(X_i,Y_i)_{i=1}^N\}$ be a dataset made of $N$ i.i.d. random variables with a common distribution in $\cP$. For any $P\in\cP$, let $f^*_{P}\in\argmin_{f\in F}\E_{(X, Y)\sim P}[(Y-f(X))^2]$. The breakdown number of the estimator $T$ on the class $\cP$ at rate $\cR$ with confidence $\delta$ is
\[
K^*_{\text{ML}}(T,N,\cR,\delta,\cP)=\min\left\{k\in \bZ_+ : \inf_{P\in\cP}\bP_{\cD\sim P^{\otimes N}}\pa{\sup_{|\cO|=k}\norm{T(\cD\cup\cO)-f^*_{P}}_{L^2(P_X)}\leqslant\cR}\geqslant 1-\delta\right\}
\]where $P_X$ denotes the marginal on $\cX$ of $P$.
\end{Definition}
 In words, the breakdown number is the minimal number of points one has to add to a dataset to break statistical performance of an estimator $T$. In this definition, the rate $\cR$ one can achieve is of particular importance. Ultimately, we would like to take it as small as possible. To that end, recall the definition of a minimax rate of convergence.

\begin{Definition}\label{def:minimax}
Let $\delta>0$, $N\geqslant 1$, $\sigma>0$, $F$ be a class of functions from $\cX$ to $\R$ and, for any $f\in F$, let $P_f$ denote the distribution of $(X,Y)$ when $X$ is a standard Gaussian process on $\cX$, $\xi$ is a standard Gaussian random variable on $\R$  independent of $X$ and $Y=f(X)+\sigma\xi$. For any $f\in F$, denote by $\bP_f=P_f^{\otimes N}$ the distribution of a sample $(X_i,Y_i)_{i=1,\ldots,N}$ of $N$ i.i.d. random variables distributed according to $P_f$. Any estimator $\ERM{N}=T((X_i,Y_i)_{i=1,\ldots,N})$ is said to perform with accuracy $\cR$ and confidence $\delta$ uniformly on $F$ if 
\[
\forall f\in F,\qquad \bP_f\pa{\norm{\ERM{N}-f}_{L^2(P_X)}^2\leqslant \cR}\geqslant 1-\delta\enspace.
\]
The minimax accuracy $\cR(\delta,F)$ is the smallest accuracy that can be achieved uniformly over $F$ with confidence $\delta$ by some estimator.
\end{Definition}

The minimax rate of convergence defined in Definition~\ref{def:minimax} is obtained in the ``Gaussian model" (that is when the dataset is made of $N$ i.i.d.  observations in the regression model with Gaussian design and Gaussian noise independent of the design) which is somehow the benchmark model in statistics in which one can derive minimax rates of convergence over $F$. We will therefore use the minimax rates obtained in this model as benchmark rates of convergence.

For any class $\cP$ containing the Gaussian model indexed by $F$, that is  $(P_f)_{f\in F}$ (we shall always work under this assumption in the following) and any $\cR<\cR(\delta,F)$, it is clear that  $K^*_{\text{ML}}(T,N,\cR,\delta,\cP)=0$ as no estimator can even achieve the rate $\cR$ even on  datasets containing no outliers. The notion is therefore interesting only for rates $\cR\geqslant \cR(\delta,F)$ and we shall be particularly interested by rates $\cR=C\cR(\delta,F)$ for some absolute constant $C\geqslant 1$ since estimators with $K^*=K^*_{\text{ML}}(T,N,\cR,\delta,\cP)>0$ are minimax (therefore statistically optimal) even when the dataset has been corrupted by up to $K^*$ outliers. Furthermore, we will be interested in classes $\cP$ much larger than $(P_f)_{f\in F}$, typically, one would like $\cP$ to be the class of all distributions $P$ of $(X,Y)$ such that $Y$ and all $f(X)$ with $f\in F$ have finite second order moments since these are the minimal properties giving sense to the square risk $\E[(Y-f(X))^2]$. Finally, the following link between the classical breakdown point and the new breakdown number always holds : for any $\delta>0$, $\cR>0$, any class $\cP$, any dataset $\cD$ and any estimator $T$, one has
\[
\epsilon^*(T,\cD)\geqslant \frac{1+K^*_{\text{ML}}(T,|\cD|,\cR,\delta,\cP)}{1+K^*_{\text{ML}}(T,|\cD|,\cR,\delta,\cP)+N}\enspace.
\]

Using this new definition, one can be interested in various existence or optimality properties such as:  
\begin{enumerate}
	\item Given $0<\delta<1$, $C>0$, what is the largest set of sample distributions  $\cP_{\delta,C}$ where one can build estimators with accuracy $C\cR(\delta,F)$ of the same order as in the Gaussian model ?
	\item  given $0<\delta<1$, $\cR>\cR(\delta,F)$ and $\cP\subset \cP_{\delta,C}$, is there some procedure $T$ such that $K^*_{\text{ML}}(T,N,\cR,\delta,\cP)>0$?
 \item given $0<\delta<1$, $\cR>\cR(\delta,F)$ and $\cP\subset \cP_{\delta,C}$,  how large can be $\sup_T K^*_{\text{ML}}(T,N,\cR,\delta,\cP)$? and is this maximum achievable by a computationally tractable estimator $T$?
\end{enumerate}

%
%Intuitively, a learning procedure is robust when its statistical performance are not deteriorated by contamination of the dataset. The largest proportion of contaminated observations within a dataset satisfying this property is the breakdown point of the learning procedure. Of course, we expect learning procedure to have a large breakdown point given that learning procedure are  likely to be feed with corrupted datasets. 
%
%The main difference with the Donoho-Huber definition of breakdown point is that we are interested only in minimax procedures (even though Definition~\ref{def:BP_learning_theory} can be easily extended to any procedure) and that the effect of outliers on this procedure should not deteriorate its estimation properties: we still expect the learning procedure to achieve the minimax rate even if it is given a corrupted dataset whereas in Donoho-Huber's definition of breakdown point, the corrupted estimator can be made as large as we want by a strategic choice of the outliers. In particular, in the Donoho-Huber's definition of breakdown point, the statistical properties of the estimator are not required to be safeguarded. Here, in Definition~\ref{def:BP_learning_theory}, we expect a stronger robustness property of a learning procedure than in Donoho-Huber's one, since, even for a strategic choice of the outliers, we still expect the estimators to be minimax (if we started with a minimax estimator). 

Most recent theoretical works on robust learning approaches focus on the first question, showing that statistical optimality can actually be achieved over much broader classes of distributions than the Gaussian model, see \cite{LugosiMendelson2016,LugosiMendelson2017} for example. In this paper, we build on the approach introduced in \cite{MOM1} where we proved that, for any $\delta$ smaller  than $\exp(-CN r_N^2)$, there exists estimators such that $K^*_{\text{ML}}(T,N, r_N^2, \delta,\cP)\geqslant CN r_N^2$ where $r_N^2$ is the minimax rate of convergence in the Gaussian model for exponential deviation. It is for instance given by $r_N^2 = \sigma^2 d/N$ for the problem of estimation of a $d$-dimensional vector and $r_N^2 = \sigma^2 s \log(ed/s)/N$ when this vector has only $s$ non-zero coordinates. 
%Furthermore, good data don't need to be identically distributed provided that their first and second moments are close to those of $P$. 
But we go further in this paper, showing that these performance can be achieved by a procedure that is computationally tractable, hence by a Machine Learning algorithm. This ultimate breakthrough is due to a new estimator that we shall now introduce.

\subsubsection{Our estimators}

The key equation to understand our procedure is that the target 
\[
\bayes=\argmin_{f\in F}\E[(Y-f(X))^2]
\]
can be obtained as the solution of the following minimaximization problem
\[
\bayes=\argmin_{f\in F}\sup_{g\in F}\E[(Y-f(X))^2-(Y-g(X))^2]\enspace.
\]
This remark is obvious since the expectation operator is linear, it is fundamental though in our construction as we use non-linear estimators of the expectation. 
More precisely, unknown expectations are estimated by median-of-means (MOM) estimators \cite{MR1688610, MR855970, MR702836} : given a partition of the dataset into blocks of equal sizes, MOM is the median of the empirical means over each block, see Section~\ref{sec:MOMTests} for details. 
MOM estimators are naturally more robust than the empirical mean thanks to the median step : if the number of blocks is at least twice that of the outliers, outliers can only affect less than half blocks so MOM remains an acceptable estimator of the mean. 
MOM's estimators are non linear, therefore plugging a MOM estimator into the minimization problem doesn't yield the same estimator as the plugging estimator on the minimaximization problem. More precisely, it is not so hard to see that MOM minimizers cannot achieve (fast) minimax rates $\cR(\delta,F)$. This is why we focus in the following on the minimaximization problem

MOM estimators of the increments of criteria have already been used in \cite{LugosiMendelson2016,LugosiMendelson2017,MOM1}.  What is new here is that, instead of combining these estimators to build a ``tournament" as in \cite{LugosiMendelson2016,LugosiMendelson2017} or using the approach of Le Cam as in \cite{MOM1}, we simply plug these estimators in the minimaximization problem. Our construction is therefore extremely simplified compared to these previous papers, we'll show that our new estimator achieves the same theoretical properties as these alternatives. 

The main difference is that this new estimator is easy to implement. Given $(f_t,g_t)$ and $\ell_f(x,y)=(y-f(x))^2$, one can find a median block $B_{\text{med}}$ such that
\[
P_{B_{\text{med}}}[\ell_{f_t}-\ell_{g_t}]=\text{median}(P_{B_{k}}[\ell_{f_t}-\ell_{g_t}], \ k\in \{1,\ldots,K\})\enspace.
\]
Then one can perform a descent/ascent algorithm  over the block $B_{\text{med}}$ to get the next iteration $(f_{t+1}, g_{t+1})$. In practice though, this basic idea can be substantially improved by shuffling the blocks at each time step, cf. Section~\ref{sub:maximinimization_saddle_point_random_blocks_and_outliers_detection}. The first advantage of this shuffling step is that the algorithm doesn't converge to local minimaxima. 

Furthermore, our algorithm defines a natural notion of depth of data : deep data are typically regularly chosen as member of the median block $B_{\text{med}}$ while ``outliers" on the other hand are typically left aside. This notion of depth, based on the risk function, is natural in our learning framework and should probably be investigated more carefully in future works. It also suggests an empirical definition of outliers and therefore an outliers detection algorithm as a by-product.
%\textcolor{red}{Introduire un peu le MOM en mettant un peu de biblio.}
%
%One aim of this paper is to construct learning procedures with a large breakdown point according to Definition~\ref{def:BP_learning_theory} satisfying the equivariant property in \eqref{eq:equivariant}. We actually prove that the breakdown point of learning procedure is at least as large as the minimax rate of convergence. Note that if we were keeping the Donoho-Huber definition of breakdown point, we prove that our MOM learning procedure has  a constant breakdown point because our analysis also provide statistical results when the dataset is corrupted up to a constant proportion of the dataset. In fact, if $|\cO|$ is the number of outliers, we proved that the MOM procedure achieves the rate $\max(|\cO|/N, \cR_N(\cF^*, \delta_N))$ which is finite even when $|\cO|\sim N$. 
%

%\textcolor{red}{high breakdown point estimators are also effective for outliers detection (see \cite{atkinson1986influential, rousseeuw1984robust,rousseeuw1990unmasking}).  Actually, }
%
%
%PLAN de l'aricle
 
The paper is organised as follow. Section~\ref{sec:setting} introduces formally the classical least-squares learning framework and presents our new estimator, Section~\ref{sec:AssMR} details our main theoretical results, that are proved in Section~\ref{sec:proofs}. An extended simulation study is provided in Section~\ref{sec:simulation_study}, including the presentation of many robust versions of standard optimization algorithms. Theoretical abstract rates derived in the main results are particularized in Section~\ref{sec:further_examples_of_applications} into important problems in Machine Learning to show the extent of our results. These main results focus on penalized version of the basic tests presented in this introduction that are well suited for high dimensional learning frameworks, we complete these results in Section~\ref{sec:learning_without_regularization_and_minimax_optimality}, providing results for the basic estimators without penalizations in small dimension. Finally, Section~\ref{sec:minimax_optimality_of_} provides some optimality results for our procedures.

\section{Setting}\label{sec:setting}
Let $\cX$ denote a measurable space and let $(X,Y),(X_i,Y_i)_{i\in[N]}$ denote random variables taking values in $\cX\times \R$.
% (note that we do not assume that they are i.i.d.). 
 Let $P$ denote the distribution of $(X,Y)$ and, for $i\in[N]$, let $P_i$ denote the distribution of $(X_i,Y_i)$. Let $F$ denote a convex class of functions $f:\cX\to\R$ and suppose that $F\subset L^2_P$, $\E[Y^2]<\infty$. For any $(x, y)\in\cX\times\R$, let $\ell_f(x,y)=(y-f(x))^2$ denote the square loss function and let $f^*$ denote an oracle
 \[
 f^*\in\argmin_{f\in F}P\ell_f\qquad \text{where} \qquad  \forall g\in L^1_P,\;Pg=\E[g(X,Y)]\enspace. 
 \]
% Consider the  and an associated oracle defined for all $f\in F$ and all $(x, y)\in\cX\times\R$ by
%\[
%\ell_f(x,y)=(y-f(x))^2,\mbox{ and } \bayes\in \argmin_{f\in F}P\ell_f\enspace.
%\]
For any $Q\in \{P,(P_i)_{i\in [N]}\}$ and any $p\geqslant 1$, let $\norm{f}_{L^p_{Q}}=(Q|f|^p)^{1/p}$ the $L_Q^p$-norm of $f$ whenever it's defined. Finally, let $\norm{\cdot}$ be a norm defined on the span of $F$; $\norm{\cdot}$  will be used as a regularization norm.

\subsection{Minimaximization}\label{sec:minimaxMOM}
Since $f^*$ minimizes $f\to P\ell_f$ and the distribution $P$ is unknown, it is estimated in empirical risk minimization \cite{MR1719582,MR1739079,MR2829871} by $\ERM{\text{ERM}}$ minimizing $f\mapsto P_N \ell_f$, where for any function $g$, $P_Ng =N^{-1}\sum_{i=1}^Ng(X_i,Y_i)$. This approach works well when $P_N \ell_f$ is close to $P\ell_f$ for all functions in $F$. This uniform proximity requires strong concentration properties. 
%
%The key idea behind ERM is that the performance of $f$ are revealed by the empirical risk $P_N \ell_f$ when this one is a good proxy of $P\ell_f$. 
Moreover, a single outlier can break down all estimators $P_N \ell_f$.
%  of $P\ell_f$. 

%Therefore, the quantity $P_N \ell_f$ can not be used anymore if one has only weak stochastic assumption or if the dataset has been corrupted by outliers. 

%Another look at the definition of the oracle $f^*$ leads to other type of estimator. 
A key observation in our analysis is that $f^*$ is solution of a minimaximization problem:
\begin{equation}\label{eq:LearningFromTests}
\bayes=\argmin_{f\in F}P\ell_f=\argmin_{f\in F}\sup_{g\in F}P(\ell_f-\ell_g)\enspace. 
\end{equation}
To estimate the unknown expectations $T_{\text{id}}(g,f)=P(\ell_f-\ell_g)$, we use a Median-Of-Means (MOM) estimator of the expectation. As this estimator is non linear, 
%more precisely, since $\text{MOM}_K[\ell_f-\ell_g]\ne \text{MOM}_K[\ell_f]-\text{MOM}_K[\ell_g]$ in general, 
plugging a MOM estimator in the minimaximization problem or in the minimization problem in \eqref{eq:LearningFromTests} does not yield the same estimator.

The minimaximization point of view has been considered in \cite{MR2906886} and \cite{MR3595933,MR3565484}. These papers used bounded modification of the quadratic loss. As for classification, the impact of outliers on the empirical mean is negligible if the loss is bounded and the number of outliers controlled, which ensures robustness. The cost to pay is that the function to be minimized is not convex and is therefore hard to implement. 
The idea of estimation by aggregation of tests is due to Le Cam \cite{MR0334381, MR856411} and was further developed by Birg\'e \cite{MR2219712}, Baraud \cite{MR2834722}, Baraud, Birg\'e and Sart \cite{MR3595933}. 
%The principle relies on the following remark. 
%
 $\text{MOM}$ estimators have been recently considered see \cite{LugosiMendelson2016, MOM1, LugosiMendelson2017, LugosiMendelson2017-2} and Section~\ref{sec:MOMTests} for details. Nevertheless, to the best of our knowledge, this paper is the first to consider the estimators obtained by plugging MOM's estimators in the minimaximization problem \eqref{eq:LearningFromTests}. This approach will proved to be efficient for designing algorithms.
% , and aggregate these tests using the $\rho$-aggregation procedure of \cite{MR3595933}, see Section~\ref{sec:RhoAgg}.

 % such that $T_N(g,f)+T_N(f,g)=0$, for all $f,g$ in $F$ and the second step is to \emph{aggregate} these tests, that is, to estimate the ideal criterion $C_{\text{id}}(f)=\sup_{g\in F}T_{\text{id}}(g,f)$. In this paper, we use

\subsection{$\text{MOM}$ estimator}
\label{sec:MOMTests}
%The idea behind MOM estimators is to split the dataset into $K$ blocks, to use each block to compute an empirical mean and then to take the median of all those means. 

Let $K$ denote an integer smaller than $N$ and let $B_1,\ldots,B_K$ denote a partition of $[N]$ into blocks of equal size $N/K$ (w.l.o.g. we assume that $K$ divides $N$ and that $B_1$ is made of the first $N/K$ data, $B_2$ of the next $N/K$ data, etc.). For all function $\cL:\cX\times\R\to\R$ and  $k\in[K]$, let $P_{B_k}\cL=|B_k|^{-1}\sum_{i\in B_k}\cL(X_i,Y_i)$. Then  $\MOM{K}{\cL}$ is a median of the set of $K$ real numbers $\{P_{B_1}\cL, \cdots, P_{B_K}\cL\}$. 

We will make an extensive use of empirical medians and quantiles in the following. We now precise some conventions used repeatedly hereafter. For all $\alpha\in (0,1)$ and real numbers $x_1,\ldots,x_K$, we denote by 
\[
\cQ_{\alpha}(x_1, \ldots, x_K)=\left\{u\in \R: \quad |\{k \in [K] : x_k\geqslant u\}|\geqslant (1-\alpha)\ell,\quad |\{k\in [K] : x_k\leqslant u\}|\geqslant \alpha\ell\right\}\enspace.
\]
Any element in $\cQ_{\alpha}(x)$ is a $(1-\alpha)$-empirical quantile of the vector $x_1,\ldots,x_K$. Hereafter, $Q_{\alpha}(x)$ denotes an element in $\cQ_{\alpha}(x)$. For all $x=(x_1,\ldots,x_K)$, $y=(y_1,\ldots,y_K)$ and $t\in\R$,
\begin{align*}
 Q_{\alpha}(x)\geqslant t& \qquad\text{iff}\qquad  \sup\cQ_{\alpha}(x)\geqslant t \enspace,\\
 Q_{\alpha}(x)\leqslant t& \qquad\text{iff}\qquad  \inf \cQ_{\alpha}(x)\leqslant t\enspace,\\
 z=Q_{\alpha}(x)+Q_{\alpha}(y)&\qquad\text{iff}\qquad z\in \cQ_{\alpha}(x)+\cQ_{\alpha}(y)
\end{align*}
where in the last inequality we use the Minkowsky sum of two sets. More generally, inequalities involving empirical quantiles are always understood in the worst possible case. For example, one can check that, for all vectors $x$ and $y$
\[
\sup \cQ_{1/4}(x) + \sup\cQ_{1/4}(y)\leqslant \inf\cQ_{1/2}(x+y)\enspace.
\]
In the following, we will simply write with some abuse of notation that for all vectors $x$ and $y$: 
\[
Q_{1/4}(x) + Q_{1/4}(y)\leqslant Q_{1/2}(x+y)\enspace. 
\]
Likewise, one can check that $Q_{1/2}(x-y)\leqslant Q_{3/4}(x) - Q_{1/4}(y)$ and $Q_{1/2}(x) \geqslant  - Q_{1/2}(-x)$. Moreover, to prove that $Q_{\alpha}(x)\geqslant t$ it is enough to prove that $\exists J \subset[K], |J|\geqslant (1-\alpha)K, \forall k\in J, x_k\geqslant t$ and to prove that $Q_{\alpha}(x)\leqslant t$ it is enough to prove that $\exists J \subset[K], |J|\geqslant \alpha K, \forall k\in J, x_k\leqslant t$. 

We conclude this section with the definition of MOM tests and their regularized versions.
\begin{Definition}\label{def:MOMTests}
Let $\alpha\in(0,1)$ and $K\in[N]$. For all functions $\cL :\cX\times\R\to\R$ the \textbf{$\alpha$-quantile on $K$ blocks of $\cL$} is $Q_{\alpha,K}(\cL)=Q_\alpha((P_{B_k}\cL)_{k\in[K]})$. In particular, the Median-of-Means (MOM) of $\cL$ on $K$ blocks is defined as $\MOM{K}{\cL} = Q_{1/2, K}(\cL)$. For all $f, g\in F$, the \textbf{MOM test on $K$ blocks of $g$ against $f$} is defined by
\begin{equation*}
T_{K}(g,f)=\MOM{K}{\ell_f-\ell_g}
\end{equation*}and, for a given regularization parameter $\lambda\geqslant0$, its regularized version is 
\begin{equation*}
T_{K,\lambda}(g,f)= \MOM{K}{\ell_f-\ell_g} + \lambda(\norm{f}-\norm{g})\enspace.
\end{equation*}
\end{Definition}

%Intuitively, $T_{K}(g,f)$ is an estimator of the ideal test $T_{\text{id}}(g,f)$. 
%The advantage of $T_{K}(g,f)$ over the ``more classical'' empirical mean estimator $P_N(\ell_f-\ell_g)$ is that it will be proved to be robust to outliers.

\subsection{Minimaximization of MOM tests}
\label{sec:RhoAgg}
We are now in position to define our basic estimators. These are simply obtained by replacing the unknown expectations $P(\ell_{f}-\ell_g)$ in \eqref{eq:LearningFromTests} by the (regularized version) of the MOM tests on $K$ blocks of $g$ against $f$,
\begin{Definition}\label{def:estimators}
For any $K\in[N/2]$, let
% we define the estimators $\hat f_K$ and its regularized version $\hat f_{K,\lambda}$ as the estimators obtained via a minimaximization of the MOM estimator $T_K:(g, f)\to \MOM{K}{\ell_f-\ell_g}$ and of its regularized version:
\begin{equation}\label{eq:minmax_esti}
\hat f_K \in\argmin_{f\in F}\max_{g\in F} T_K(g, f) \mbox{ and }
\hat f_{K, \lambda} \in\argmin_{f\in F}\max_{g\in F} T_{K, \lambda} (g, f).
\end{equation}
\end{Definition}
\noindent
One can rewrite our estimators as cost minimization estimators
\begin{equation*}
\hat f_K \in\argmin_{f\in F} \Crit{K}{f} \mbox{ and } \hat f_{K,\lambda}\in\argmin_{f\in F}\Crit{K,\lambda}{f}
\end{equation*}where for all $f\in F,$ $\Crit{K}{f}=\sup_{g\in F}T_{K}(g,f)$ and $\Crit{K,\lambda}{f}=\sup_{g\in F}T_{K,\lambda}(g,f)$. These cost functions play a central role in the theoretical analysis of $\hat f_K$ and $\hat f_{K, \lambda}$ (cf. Section~\ref{sec:proofs}).

Compared to the aggregation of tests by ``tournaments" \cite{LugosiMendelson2016, LugosiMendelson2017}, or ``Le Cam type" aggregation of tests \cite{MOM1}, minimaximization estimators have several avantages. First they are very natural from Eq~\eqref{eq:LearningFromTests} and they do not require a proxy for the excess risk or for the $L^2(P)$-norm, which simplifies  the presentation of this estimator. Finally, while building the estimators of \cite{LugosiMendelson2016, LugosiMendelson2017} or \cite{MOM1} is totally impossible, these new estimators are minimaximization procedures of locally convex-concave functions that could be approximated by a numerical scheme. It is actually easy to ``convert'' most of the classical algorithms used for ``cost minimization'' into an algorithm for minimaximization solving \eqref{eq:minmax_esti} (see Section~\ref{sec:simulation_study} below).

\begin{Remark}[$K=1$ and ERM] If one chooses $K=1$ then for all $f, g\in F$, $T_K(g,f) = P_N (\ell_f-\ell_g)$ and it is straightforward to check that $\hat f_K$ and $\hat f_{K, \lambda}$ are respectively the Empirical risk Minimization (ERM) and its regularized version (RERM). 
%In particular, when the dataset is not polluted by outliers, one has no reason to split the dataset in order to remove corrupted blocks of data (since there are no such blocks). That is, one should take $K=1$ in the non-corrupted setup and therefore we consider ERM and RERM procedures. 
It turns out that, when we choose $K$ on a data-driven basis as we do in Section~\ref{sub:adaptive_choice_of_the_number_k_of_blocks_via_mom_cv}, when the proportion of outliers is small and good data have light tails, then the selected number of blocks $\hat K$ is equal or close to $1$ (cf. Figure~\ref{fig:MOM_CV_K} below). Therefore, for clean datasets where ERM performs optimally, our procedure matches this optimal estimator. 
%The difference is that it is still optimal when the dataset is corrupted.
\end{Remark}

\section{Assumptions and main results}\label{sec:AssMR}
%One of our motivations in this paper is to show the robustness properties of median-of-means estimators in statistical learning. We shall therefore 
Denote by $\cO\cup\cI$ a partition of $[N]$, where $\cO$ has cardinality $|\cO|$. Data $(X_i,Y_i)_{i\in \cO}$ are considered as \emph{outliers}, no assumptions on the joint distribution of these data or on their distribution conditionally on data $(X_i,Y_i)_{i\in \cI}$ is made. These data may not be independent, nor independent from the remaining data. The remaining set $(X_i,Y_i)_{i\in \cI}$ is the set of \emph{informative} data, that is the ones one can use for estimation. These are hereafter assumed \emph{independent}. Given the data $(X_i, Y_i)_{i\in[N]}$ no one knows in advance which data is informative or not. 

%Even for the informative data, we do not assume that $(X_i, Y_i)_{i\in \cI}$  are i.i.d. $\sim P$, but that the moments of order $1$ and $2$ of all functions $f\in F$ are close to those of $P$ as well as the correlations between $f$ and the noise $\zeta$. 

\subsection{Assumptions}\label{sec:Ass}
All the assumptions we need to get the results involve only first and second moment of the distributions $P,(P_i)_{i\in\cI}$ on functions in $F$ and $Y$. 
%Such assumption should hold only for the informative data. 
In particular, this setting strongly relaxes usual strong concentration assumptions made on the informative data to study ERM estimators. 
%Despite those very weak assumption, we can establish statistical results for the MOM estimators which are optimal in a much stronger setup (independent Gaussian noise, Gaussian family $F$ and no outliers). 

\begin{Assumption}\label{ass:Mom2F}
There exists $\param{r0}>0$ such that for all $f\in F$ and all $i\in \cI$,
\begin{equation*}
 \sqrt{P_i(f-\bayes)^2}\leqslant \param{r0}\sqrt{P(f-\bayes)^2}.
\end{equation*}
\end{Assumption}
Of course, Assumption~\ref{ass:Mom2F} holds in the i.i.d. framework, with $\param{r0}=1$ and $\cI=[N]$.
The second assumption bounds the correlation between the ``noise'' $\zeta_i= Y_i-\bayes(X_i)$ and  the shifted class $F-\bayes$.

\begin{Assumption}\label{ass:margin}
There exists $\param{m}>0$ such that for all   $i\in \cI$ and all $f\in F$,
\[
 \text{var}(\zeta_i(f-\bayes)(X_i))\leqslant \param{m}^2\norm{f-\bayes}^2_{L^2_P}\enspace.
\]
\end{Assumption}
Assumption~\ref{ass:margin} typically holds in the i.i.d. setup when the noise $\zeta=Y-f^*(X)$ has uniformly bounded $L^2$-moments conditionally to $X$, which holds in the classical framework when $\zeta$ is independent of $X$ and $\zeta$ has a finite $L^2$-moment bounded by $\param{m}$. In non-i.i.d. setups, assumption~\ref{ass:margin} also holds if for all $i\in \cI$, $\norm{\zeta}_{L^4_{P_i}}\leqslant \param{2}<\infty$ -- where $\zeta(x, y) =y-f^*(x)$ for all $x\in\cX$ and $y\in\R$ -- and, for every $f\in F$, $\norm{f-\bayes}_{L^4_{P_i}}\leqslant \theta_1 \norm{f-\bayes}_{L^2_P}$, because, in that case, 
\begin{equation*}
\sqrt{{\rm var}_{P_i}(\zeta(f-\bayes))} \leqslant \norm{\zeta(f-\bayes)}_{L^2_{P_i}}\leqslant \norm{\zeta}_{L^4_{P_i}}\norm{f-\bayes}_{L^4_{P_i}}\leqslant \theta_1\param{2}\norm{f-\bayes}_{L^2_P},
\end{equation*}and so Assumption~\ref{ass:margin} holds for $\param{m} = \theta_1 \param{2}$.

Now, let us introduce a norm equivalence assumption over $F-f^*$: we call it a $L^2/L^1$ assumption.

\begin{Assumption}\label{ass:small-ball}
There exists $\theta_0\geqslant 1$ such that for all $f\in F$ and all $i\in \cI$
\begin{equation*}
\norm{f-f^*}_{L^2_P}\leqslant \theta_0\norm{f-f^*}_{L^1_{P_i}}\enspace.
\end{equation*}
\end{Assumption}
 Note that $\norm{f-f^*}_{L^1_{P_i}}\leqslant \norm{f-f^*}_{L^2_{P_i}}$ for all $f\in F$ and $i\in\cI$. Therefore, Assumption~\ref{ass:Mom2F} and Assumption~\ref{ass:small-ball} are together equivalent to assume that all the norms $L^2_P, L_{P_i}^2, L_{P_i}^1, i\in\cI$ are equivalent over $F-f^*$. Note also that Assumption~\ref{ass:small-ball} is equivalent to the small ball property (cf. \cite{MR3431642,Shahar-COLT}) which has been recently used in Learning theory and signal processing. We refer to \cite{Shahar-Vladimir,LM_compressed,shahar_general_loss,MR3364699,Shahar-ACM,RV_small_ball} for examples of distributions satisfying the small ball assumption.  

\begin{Proposition}\label{prop:SBP=L2/L1} Let $Z$ be a real-valued random variable. The following holds:
\begin{enumerate}
	\item If there exists $\kappa_0$ and $u_0$ such that $\bP(|Z|>\kappa_0\norm{Z}_2)\geqslant u_0$ then $\norm{Z}_2\leqslant (u_0\kappa_0)^{-1}\norm{Z}_1$;
	\item if there exists $\param{0}$ such that $\norm{Z}_2\leqslant \param{0}\norm{Z}_1$, then for any $\kappa_0<\param{0}^{-1}$, $\bP(|Z|>\kappa_0\norm{Z}_2)\geqslant u_0$ where $u_0=(\param{0}^{-1}-\kappa_0)^2$.
\end{enumerate} 
\end{Proposition}
\begin{proof}
Suppose that there exists $\kappa_0$ and $u_0$ such that $\bP(|Z|>\kappa_0\norm{Z}_2)\geqslant u_0$ then 
 \begin{equation*}
  \norm{Z}_{1}\geqslant \int_{|z|\geqslant \kappa_0\norm{Z}_{2}}|z|dP_Z(z)\geqslant u_0\kappa_0\norm{Z}_{2}\enspace,
 \end{equation*}where $P_Z$ denotes the probability distribution associated with $Z$. Conversely, assume that $\norm{Z}_2\leqslant \theta_0 \norm{Z}_2$. Let $p=\bP\left(|Z| \geqslant \kappa_0 \norm{Z}_{2}\right)$. It follows from Paley-Zigmund's argument \cite[Proposition~3.3.1]{MR1666908} that 
\begin{align*}
\norm{Z}_{2}&\leqslant \param{0}\norm{Z}_{1}\leqslant \param{0}\pa{\E[|Z|I(|Z| \leqslant \kappa_0 \norm{Z}_{2})] + \E[|Z|I(|Z| \geqslant \kappa_0 \norm{Z}_{2}}\\
&\leqslant \param{0}\norm{Z}_{2}\pa{\kappa_0 + \sqrt{p}}\enspace,
\end{align*}
therefore, $p\geqslant (\param{0}^{-1}-\kappa_0)^2$.
 \end{proof}

%The main results involve the following fixed points. The first definition will be used to bound the risk of the estimator $\ERM{K}$.
%
%\begin{Definition}\label{def:the-three-parameters}
%Let $(\eps_i)_{i\in [N]}$ be independent Rademacher random variables, independent from $(X_i,Y_i)_{i=1}^N$. Let 
%\[
%B(f,r)=\left\{g\in F : \norm{g-f}_{L^2_P}\leqslant r\right\},\qquad\forall f\in F,\forall r>0\enspace.
%\]
%For any $\gamma_Q,\gamma_M>0$, let
%\begin{align*}
%r_Q(\gamma_Q)& = \sup_{f^\star\in F}\inf\left\{r>0 : \forall J\subset \cI, |J|\ge \frac{|\cI|}2,\;\E \sup_{f \in F \cap B\left(f^\star,r\right)} \left|\sum_{i\in J} \eps_i (f-f^\star)(X_i)\right| \leqslant  \gamma_Q |J| r \right\}\enspace,\\
%r_M(\gamma_M)&=  \sup_{f^\star \in F}\inf\left\{r>0 : \forall J\subset \cI, |J|\ge \frac{|\cI|}2,\; \E \sup_{f \in F \cap B(f^\star,r)} \left|\sum_{i\in J} \eps_{i} \zeta_i (f-f^\star)(X_i)\right| \leqslant \gamma_M r^2|J| \right\}\enspace.
%\end{align*}
%%and let $\rho\to r(\rho,\gamma_Q,\gamma_M)$ be a continuous and increasing function such that for every $\rho>0$,
%%\begin{equation*}
%%r(\rho,\gamma_Q,\gamma_M)\geqslant \max\{r_Q(\rho,\gamma_Q),r_M(\rho,\gamma_M)\}.
%%\end{equation*}
%\end{Definition}
%
%The second definition will be used to bound the risk of the estimator $\ERM{K,\lambda}$.

\subsection{Complexity measures and minimax rates of convergence} % (fold)
\label{sub:complexity_measures_and_minimax_rates_of_convergence}
Balls associated with the regularization norm $\norm{\cdot}$ and the $L^2_P$ norm play a prominent role in learning theory \cite{LM13,LM_reg1}. In particular, for all $\rho\geqslant0$,  the ``sub-models''
\begin{equation*}
B(f^*, \rho)=\{f\in F: \norm{f-f^*}\leqslant \rho\} = f^*+\rho B
\end{equation*}where $B = \{f\in{\rm span}(F), \norm{f}\leqslant \rho\}$ and their ``localizations'' at various level $r\geqslant0$,  i.e. intersection of $B(f^*, \rho)$ with $L^2_P$-balls
\begin{equation*}
B_2(f^*, r)=\{f\in F: \norm{f-f^*}_{L^2_P}\leqslant r\}
\end{equation*}
are key sets because their Rademacher complexities, drives the minimax rates of convergence.
%, together with properties of the subdifferential of the regularization (introduced in the following section).  
%More details are provided below but for the moment l
Let us introduce these complexity measures.

\begin{Definition}\label{def:the-three-parameters-Reg}
Let $(\eps_i)_{i\in [N]}$ be independent Rademacher random variables (i.e. uniformly distributed in $\{-1, 1\}$), independent from $(X_i,Y_i)_{i=1}^N$. For all $f\in F$,  $r>0$ and $\rho\in(0,+\infty]$, we denote the intersection of the $\norm{\cdot}$-ball of radius $r$ and the $L^2_P$-norm of radius $\rho$ centered at $f$ by
\begin{equation*}
B_{\text{reg}}(f,\rho,r)= B(f, \rho)\cap B_2(f, r) = \left\{g\in F : \norm{g-f}_{L^2_P}\leqslant r,\;\norm{g-f}\leqslant \rho\right\} \enspace.
\end{equation*}
Let $\zeta_i=Y_i-f^*(X_i)$ for all $i\in\cI$ and for $\gamma_Q,\gamma_M>0$ define
\begin{align*}
r_Q(\rho,\gamma_Q)& = \inf\left\{r>0 : \forall J\subset \cI, |J|\geqslant \frac{N}2,\;\E \sup_{f \in B_{\text{reg}}(f^*,\rho,r)} \left|\sum_{i\in J} \eps_i (f-\bayes)(X_i)\right| \leqslant  \gamma_Q|J| r \right\}\enspace,\\
r_M(\rho,\gamma_M)&=  \inf\left\{r>0 : \forall J\subset \cI, |J|\geqslant \frac{N}2,\;\E \sup_{f \in B_{\text{reg}}(f^*,\rho,r)} \left|\sum_{i\in J} \eps_{i} \zeta_i (f-\bayes)(X_i)\right| \leqslant \gamma_M |J| r^2 \right\}\enspace,
\end{align*}
and let $\rho\to r(\rho,\gamma_Q,\gamma_M)$ be a continuous and non decreasing function such that for every $\rho>0$,
\begin{equation*}
r(\rho) = r(\rho,\gamma_Q,\gamma_M)\geqslant \max\{r_Q(\rho,\gamma_Q),r_M(\rho,\gamma_M)\}.
\end{equation*}
\end{Definition}It follows from Lemma~2.3 in \cite{LM13} that $r_M$ and $r_Q$ are continuous and non decreasing functions. Note that $r_M(\cdot), r_Q(\cdot)$ depend on $f^*$. According to \cite{LM13}, if one can choose $r(\rho)$  equal to the maximum of $r_M(\rho)$ and $r_Q(\rho)$ then $r(\rho)$ is the minimax rate of convergence over $B(f^*, \rho)$. 
%For regularized estimators, one needs to known the function $r(\cdot)$ to fit the regularization parameter. In that case, one should choose $r(\cdot)$ independent of $f^*$. 
Note also that $r_Q$ and $r_M$ are well defined when $\cI \geqslant N/2$, which implies that at least half data are informative. 
%This seems to be a reasonable assumption given that we consider a scenario where outliers do not have to satisfy any assumption and therefore can be anything.

\subsection{The sparsity equation} % (fold)
\label{sub:the_sparsity_equation}
To control the risk of our estimator, we bound from above $T_{K,\lambda}(f, \bayes)$ for all functions $f$ far from $\bayes$ either in $L^2_P$-norm or for the regularization norm $\norm{\cdot}$. Recall that
\[
T_{K,\lambda}(f,\bayes)=\text{MOM}_K[2\zeta(f-\bayes)-(f-\bayes)^2]+\lambda(\norm{f^*}-\norm{f})\enspace.
\]
The multiplier term ``$2\zeta(f-\bayes)$'' is the one containing the noise and is therefore the term we will try to control from above using either the quadratic process ``$(f-\bayes)^2$'' or the regularization term ``$\lambda(\norm{f^*}-\norm{f})$''. To that end we will need both to control from below the quadratic process 
%(this is done in Lemma~\ref{lem:UBQPReg}) 
and the regularization term. 
%The aim of the section is to introduce a radius $\rho^*$ which a lower bound on the regularization term is possible and useful to control the multiplier process. 

Let $f\in F$ and denote $\rho=\norm{f-\bayes}$. When $\|f-\bayes\|_{L^2_P}$ is small, the quadratic term $(f-\bayes)^2$ will not help to bound from above $T_{K,\lambda}(f,\bayes)$, one shall only rely on the penalization term $\lambda(\norm{f^*}-\norm{f})$. One can bound from below $\norm{f^*}-\norm{f}\gtrsim \rho$ for all $f$ close to $\bayes$ in $L_P^2$ using the \emph{sparsity equation} of \cite{LM_reg_comp}. 
First, introduce the subdifferentials of $\norm{\cdot}$ : for all $f\in F$,
\begin{equation*}
(\partial\norm{\cdot})_f = \{z^*\in E^*: \norm{f+h}\geqslant \norm{f}+z^*(h) \mbox{ for every } h\in E  \}
\end{equation*} where $(E^*, \norm{\cdot}^*)$ is the dual normed space of $(E,\norm{\cdot})$.

For any $\rho>0$, let $H_\rho$ denote the set of functions ``close" to $\bayes$ in $L_P^2$ and at distance $\rho$ from $\bayes$ in regularization norm and let $\Gamma_{\bayes}(\rho)$ denote the set of subdifferentials of all vectors close to $\bayes$:
$$
H_{\rho} = \{f \in F : \norm{f-f^*}=  \rho \ {\rm and} \ \|f-f^*\|_{L^2_P} \leqslant r(\rho)\}
\an
\Gamma_{f^*}(\rho)=\bigcup_{f\in F : \norm{f-f^*} \leqslant \rho/20}(\partial \norm{\cdot})_f\enspace.
$$ 
If there exists a ``sparse'' $f^{**}$ in $\{f\in F:\norm{f^*-f}\leqslant \rho/20\}$, that is $(\partial\norm{\cdot})_{f^{**}}$ is almost all the unit dual sphere, then $\norm{f}-\norm{f^{**}}$ is large for any $f\in H_{\rho}$ so  $\norm{f}-\norm{\bayes}\geqslant \norm{f}-\norm{f^{**}}-\norm{\bayes-f^{**}}$ is large as well. More precisely, let us introduce, for all $\rho>0$,
\begin{equation*}
\Delta(\rho) = \inf_{f \in H_{\rho}} \sup_{z^* \in \Gamma_{\bayes}(\rho)} z^*(f-\bayes)\enspace.
\end{equation*}
%
%$\Delta(\rho)$ is the uniform lower bound over all $f\in H_{\rho}$ of what can be interpreted as the maximal lower bound on $\norm{f}-\norm{f^{**}}$ for $f^{**}\in \Gamma_{\bayes}(\rho)$.
%The maximal value of $\norm{f}-\norm{\bayes}$ when $f\in H_\rho$ is $\rho$ by the triangular inequality. According to our previous analysis, we will have conversely $\norm{f}-\norm{\bayes}\gtrsim \rho$ if there exists $c_0>1/20$ such that, for all $f\in H_\rho$, $\sup_{f^{**}\in \Gamma_{\bayes}(\rho)}(\norm{f}-\norm{f^{**}})\ge c_0\rho$. As explained, $\Delta(\rho)$ will be used to bound from bellow $\sup_{f^{**}\in \Gamma_{\bayes}(\rho)}(\norm{f}-\norm{f^{**}})$ and our goal will follow from the following inequality
The sparsity equation, introduced in \cite{LM_reg_comp}, quantifies these notions of ``large".

\begin{Definition}\label{def:sparsity_equation}
A radius $\rho>0$ is said to satisfy the \textbf{sparsity equation} when $\Delta(\rho) \geqslant 4\rho/5$.
\end{Definition}

One can check that, if $\rho^*$ satisfies the sparsity equation, so do all $ \rho\geqslant \rho^*$. Therefore, one can define
\begin{equation*}
\rho^* = \inf\left(\rho>0: \Delta(\rho) \geqslant \frac{4\rho}{5}\right).
\end{equation*}Note that if $\rho\geqslant 20 \norm{f^*}$ then $0\in \Gamma_{f^*}(\rho)$. Moreover, $(\partial\norm{\cdot})_0$ equals to the dual ball (i.e. the unit ball of $(E^*, \norm{\cdot}^*)$) and so $\Delta(\rho) = \rho$. This implies that any $\rho\geqslant 20 \norm{f^*}$ satisfies the sparsity equation. This simple observation will be used to get ``complexity-dependent rates of convergence'' as in \cite{LM_reg2}.

\subsection{Main results}\label{sec:main_results}
Our first results study the performance of the estimators $\ERM{K,\lambda}$ for a fixed value of $K$. The other ones will provide an adaptive way to select $K$.

%\subsubsection{Results for fixed $K$}
%\begin{Theorem}\label{thm:RBRhoEst}
%Grant Assumptions~\ref{ass:Mom2F},~\ref{ass:margin},~\ref{ass:small-ball} and let $r_Q$, $r_M$ denote the functions introduced in Definition~\ref{def:the-three-parameters-Reg}. There exist absolute constants $(\cabs{i})_{1\leqslant i\leqslant 7}$ such that the following result holds. Let $K_1=\cabs{1}[|\cO|\vee Nr_M(\infty,\cabs{2})^2/\param{m}^2]$, $K_2=\cabs{3}N/(\theta_0\param{r})^2$ and assume that $K\in [K_1,K_2]$. Assume that
%\begin{equation}\label{eq:robust_theo_basic}
% (P_i-P) \zeta (f-\bayes)\leqslant \param{m}\norm{f-\bayes}_{L^2_P}\sqrt{\frac{K}N},\qquad \forall i\in \cI,\forall f\in F\enspace.
% \end{equation}
% Then with probability larger than $1-2e^{-\cabs{4}K}$,
% \[
% \|\hat f_K-f^*\|_{L^2_P}\leqslant \cabs{5}\left[ r_Q(\infty,\cabs{6}\param{0}^{-1})\vee \param{0}^2\param{m}\sqrt{\frac KN}\right] ,\qquad P\cL(\hat f)\leqslant \cabs{7}\left[ r^2_Q(\infty,\cabs{6}\param{0}^{-1})\vee \param{0}^4\param{m}^2\frac KN\right]\enspace.
% \]
%\end{Theorem}
%
%\begin{Remark}
%Comments on the results...
%\end{Remark}

%The second result provides the performance of the regularized version.
\begin{Theorem}\label{thm:RBRhoEstPen}
Grant Assumptions \ref{ass:Mom2F},~\ref{ass:margin} and \ref{ass:small-ball} and let $r_Q$, $r_M$ denote the functions introduced in Definition~\ref{def:the-three-parameters-Reg}. Assume that $N\geqslant 384 (\theta_0\theta_{r0})^2$ and $|\cO|\leqslant N/(768\theta_0^2)$. Let $\rho^*$ be solution to the sparsity equation from Definition~\ref{def:sparsity_equation}. Let $K^*$ denote the smallest integer such that 
\[
K^*\geqslant   \frac{N\eps^2}{384 \theta_m^2} r^2(\rho^*)\enspace,
\] where $\eps = 1/(833\theta_0^2)$ and $r^2(\cdot)$ is defined in Definition~\ref{def:the-three-parameters-Reg} for $\gamma_Q=(384\theta_0)^{-1}$ and $\gamma_M = \eps/192$. For any $K\geqslant K^*$, define the radius $\rho_K$ and the regularization parameter as
\begin{equation*}
r^2(\rho_K) = \frac{384 \theta_m^2}{\eps^2}\frac{K}{N} \mbox{ and }\lambda=\frac{16 \eps r^2(\rho_K)}{\rho_K}.
\end{equation*}

Assume that for every $i\in\cI$, $K\in[\max(K^*, |\cO|), N]$ and $f\in F$ such that $\norm{f-f^*}\leqslant \rho$ for $\rho\in[\rho_K, 2 \rho_K]$, one has 
\begin{equation}\label{eq:robust_theo_basic}
\left|P_i\zeta(f-f^*)- P \zeta(f-f^*)\right|\leqslant  \eps\max\left(r_M^2(\rho, \gamma_M), \frac{384 \theta_m^2}{\eps^2}\frac{K}{N}, \norm{f-f^*}_{L_p^2}^2\right) \enspace.
\end{equation}
Then, for all $K\in[\max(K^*, 8 |\cO|), N/(96 (\theta_0\theta_{r0})^2)]$,  with probability larger than  $1-4\exp(-7K/9216)$, 
the estimator $\hf_{K,\lambda}$ defined in Section~\ref{sec:RhoAgg} satisfies
\begin{gather*}
\norm{\ERM{K,\lambda}-\bayes}\leqslant 2\rho_K, \quad \norm{\ERM{K,\lambda}-\bayes}_{L^2_P}\leqslant r(2\rho_K)\\
 R(\ERM{K,\lambda})\leqslant R(\bayes)+(1+52\eps)r^2(2\rho_K)\enspace. 
\end{gather*}
\end{Theorem}

\begin{Remark}[connexion between the $P_i, i\in\cI$ and $P$]\label{rem:equiv_L2_norm}
Assumption~\eqref{eq:robust_theo_basic} holds for instance when for every $i\in\cI$ and  $f\in F$ one has 
\begin{equation*}
\norm{Y_i - f^*(X_i)}_{L^2}\leqslant \norm{Y-f^*(X)}_{L^2}, \norm{Y_i - f(X_i)}_{L^2}\geqslant \norm{Y-f(X)}_{L^2}  
\end{equation*}and
\begin{equation*}
\norm{f^*(X_i) - f(X_i)}_{L^2}\leqslant \norm{f^*(X)-f(X)}_{L^2}.
\end{equation*}
These assumptions involve second moments associated with $P$ and $(P_i)_{i\in\cI}$. As a consequence, if the metrics $L^2_{P_i}$ for $i\in\cI$ and $L^2_P$ coincide on the functions $x\to (f-f^*)(x)$ and $(x, y)\to y-f(x)$ and the $P_i$ satisfies Assumption~\ref{ass:margin} and \ref{ass:small-ball} then Theorem~\ref{thm:RBRhoEstPen} holds. This means that we only need informative data to induce the same $L^2$ metric as $P$ to estimate the oracle $f^*$ even if we do not have any observation coming from $P$ itself. This setup relaxes the classical i.i.d. where all data are generated from $P$. In this setting, our estimators achieve the same results as the ERM would if data were \emph{all} i.i.d. with a noise $\zeta$ independent of $X$ and both $X$ and $\zeta$ had a Gaussian distribution (see Section~\ref{sec:minimax_optimality_of_}). 
%Moreover, a large proportion of the data may be corrupted in our setup. 
%The main message of theorem~\ref{thm:RBRhoEstPen} is therefore to say that in learning theory we don't need much information in the data on the distribution $P$ associated with the oracle. We just need a large proportion of the data to induce the same $L^2$ metric as the one induced by $P$ on the class $F$ and its shifted version by the output.
\end{Remark}
The function $r(\cdot)$ is used to define the regularization parameter, so it cannot depend on $f^*$. When $r_M(\cdot), r_Q(\cdot)$ depend on $f^*$, $r$ should be a computable upper bound independent from $f^*$.

\subsubsection{Adaptive choice of $K$}
\label{sec:adap_lep_reg}
In Theorem~\ref{thm:RBRhoEstPen}, all rates depend on the choice of the tuning parameter $K$. The following construction inspired from Lepski's method provides an adaptive choice of this parameter. Let us first recall the definition of empirical criterion introduced in Section~\ref{sec:RhoAgg} and the associated confidence regions: for all $J\in[K]$, $\lambda>0$, $f\in F$ and absolute constant $\cabs{ad}>0$, 
\begin{equation*}
\cC_{J, \lambda}(f) = \sup_{g\in F} T_{J, \lambda}(g, f) \mbox{ and } \hat R_{J,\cabs{ad}}=\left\{f\in F : \Crit{J,\lambda}{f}\leqslant \frac{\cabs{ad}}{\param{0}^2}r^2(\rho_J)\right\}\enspace,
\end{equation*}where $T_{J, \lambda}(g, f) = \MOM{J}{\ell_f-\ell_g} + \lambda\left(\norm{f} - \norm{g}\right)$.
%
%Now, we construct a data-driven way to choose the number of blocks $K$ and an associated estimator. 
For all $J\in [\max(K^*, 8 |\cO|), N/(96 (\theta_0\theta_{r0})^2)]$ , let 
\begin{gather*}
\hat K_{\cabs{ad}}=\inf\left\{K\in [\max(K^*, 8 |\cO|), N/(96 (\theta_0\theta_{r0})^2)] : \cap_{J=K}^{N/(96 (\theta_0\theta_{r0})^2)}\hat R_{J,\cabs{ad}}\ne \emptyset\right\} \\
 \mbox{ and choose } \ERM{\cabs{ad}}\in \cap_{J=\hat K_{\cabs{ad}}}^{N/(96 (\theta_0\theta_{r0})^2)}\hat R_{J,\cabs{ad}}\enspace.
\end{gather*}
The following theorem shows the performance of these estimators.
\begin{Theorem}\label{thm:LepskiReg}
 Grant the assumptions of Theorem~\ref{thm:RBRhoEstPen} and assume moreover that and $|\cO|\leqslant N/(768\theta_0^2\theta_{r0}^2)$. For any $K\in [\max(K^*, 8 |\cO|), N/(96 (\theta_0\theta_{r0})^2)]$, with probability larger than 
 \[
 1-4\exp(-K/2304) = 1 - 4\exp\left(-\eps^2 N r^2(\rho_K)/884736\right)\enspace,
 \]
 one has
 
\begin{gather*}
 \norm{\ERM{\cabs{ad}}-\bayes}\leqslant 2\rho_K,\qquad  \norm{\ERM{\cabs{ad}}-\bayes}_{L^2_P}\leqslant r(2\rho_K)\enspace,\\
  R(\ERM{\cabs{ad}})\leqslant R(\bayes)+ (1+52\eps) r^2(2 \rho_K)\enspace,
\end{gather*}
 where $c_{ad}=18/833$ and $\eps=(833\theta_0^2)^{-1}$.  In particular, for $K=K^*$, we have $r(2\rho_{K^*}) = \max\left(r(2\rho^*), \sqrt{|\cO|/N}\right)$. Therefore,  if $r(2\rho^*)\leqslant c_1 r(\rho^*)$ holds for some absolute constant $c_1$, then the breakdown number of $\ERM{\cabs{ad}}$ is larger than $N r(\rho^*)^2$. 
\end{Theorem}

\begin{Remark}[deviation parameter]
Note that $r(\cdot)$ is any continuous, non decreasing function such that for all $\rho>0$, $r(\rho)\geqslant \max\left(r_Q(\rho, \gamma_Q), r_M(\rho, \gamma_M)\right)$. In particular, if $r_*:\rho\to \max\left(r_Q(\rho, \gamma_Q), r_M(\rho, \gamma_M)\right)$ is itself a continuous function (it is clearly non decreasing) then for every $x>0$, $r(\rho) = \max\left(r_Q(\rho, \gamma_Q), r_M(\rho, \gamma_M)\right) + x/N$ is another non decreasing upper bound. Therefore, one can derive results similar to those in Theorem~\ref{thm:LepskiReg} but with an extra confidence parameter : for all $x>0$, with probability at least $1-4\exp(-c_0 Nr_*^2(\rho_{K^*}) + c_0x)$, 
\begin{gather*}
\norm{\ERM{\cabs{ad}}-\bayes}\leqslant 2\rho_K,\qquad  \norm{\ERM{\cabs{ad}}-\bayes}_{L^2_P}\leqslant r_*(2\rho_K)+\frac{x}{N}\enspace,\\
 R(\ERM{\cabs{ad}})\leqslant R(\bayes)+ (1+52\eps) \left(r^2(2 \rho_K)+\frac{x}{N}\right).
\end{gather*}
Note however that $\ERM{\cabs{ad}}$ depends on $x$ through the regularization parameter $\lambda=16\eps(r_*(\rho_K)+x/N)/\rho_K$. 
%of the ``aggregated''estimators $\ERM{K,\lambda}$.
\end{Remark}

\section{Proofs}
\label{sec:proofs}
Recall the quadratic / multiplier decomposition of the difference of losses: for all $f, g\in F$, $x\in\cX$ and $y\in\R$, 
\begin{align}\label{eq:quad_multi_decomp}
\ell_f(x, y) - \ell_g(x, y) &= (y-f(x))^2 - (y-g(x))^2\notag\\
& = (f(x) - g(x))^2 + 2 (y-g(x))(g(x) - f(x)).
\end{align}
Upper and lower bounds on $T_K(\cdot, \cdot)$ follow from a study of ``quadratic'' and ``multiplier'' quantiles of means processes. 
%The dataset may have been corrupted by $|\cO|$ outliers that can be absolutely anything taking values in $\cX\times \R$. Therefore, a
As no assumption is granted on the outliers, any block of data containing  one or more of these outliers is ``lost'' from our perspective meaning that empirical means over these blocks cannot be controlled. Let $\cK$ denote the set of blocks which have not been corrupted by outliers: 
\begin{equation}
 \label{eq:def_cK}
 \cK = \left\{k\in[K]: B_k\subset \cI\right\}.
 \end{equation} 
 If $k\in\cK$, all data indexed by $B_k$ are informative. We will show that controls on the blocks indexed by $\cK$ are sufficient to insure statistical performance of MOM estimators.

\subsection{Bounding quadratic and multiplier processes}
The first result is a lower bound on the quantiles of means quadratic processes.
\begin{Lemma}\label{lem:UBQPReg}Grant Assumptions~\ref{ass:Mom2F} and~\ref{ass:small-ball}. Fix $\eta\in (0,1)$, $\rho\in(0,+\infty]$ and let $\alpha,\gamma,\gamma_Q, x$ be positive numbers such that $\gamma\left(1-\alpha-x-16 \gamma_Q\theta_0\right) \geqslant 1-\eta$.
Assume that $K\in[ |\cO|/(1-\gamma),N\alpha/(2\theta_0\param{r0})^2]$.
Then there exists an event $\Omega_Q(K)$ such that $\bP\left(\Omega_Q(K)\right)\geqslant 1-\exp\left(-K\gamma x^2/2\right)$ and, on $\Omega_Q(K)$:  for all  $f\in F$ such that $ \norm{f-\bayes}\leqslant \rho$, if $\norm{f-\bayes}_{L^2_P}\geqslant r_Q(\rho,\gamma_Q)$ then
\begin{equation*}
 \left|\left\{k\in [K] : P_{B_k}(f-f^*)^2 \geqslant (4\theta_0)^{-2}\norm{f-\bayes}^2_{L^2_P}\right\}  \right|\geqslant(1-\eta)K\enspace.
\end{equation*}
In particular, $Q_{\eta,K}((f-\bayes)^2)\geqslant (4\theta_0)^{-2}\norm{f-\bayes}^2_{L^2_P}$.
\end{Lemma}
\begin{proof}Define $F^*_\rho = B(f^*, \rho) = \{f\in F: \norm{f-f^*}\leqslant\rho\}$. For all $f\in F^*_\rho$, let $n_f=(f-\bayes)/\norm{f-\bayes}_{L^2_P}$ and note that for all $i\in\cI$,  $P_i|n_f|\geqslant \theta_0^{-1}$ by Assumption~\ref{ass:small-ball} and $P_i n_f^2\leqslant \param{r0}^2$ by Assumption~\ref{ass:Mom2F}. It follows from Markov's inequality that, for all $k\in \cK$ ($\cK$ is defined in \eqref{eq:def_cK}),
\[
\bP\left(|P_{B_k}|n_f|-\overline P_{B_k}|n_f||>\frac{\param{r0}}{\sqrt{\alpha|B_k|}}\right)\leqslant \alpha\enspace,
\]
where $\overline P_{B_k}|n_f| = |B_k|^{-1}\sum_{i\in B_k} \E |n_f(X_i)|\geqslant \theta_0^{-1}$ and therefore, 
\[
\bP\left(P_{B_k}|n_f|\geqslant \frac1{\theta_0}-\frac{\param{r0}}{\sqrt{\alpha|B_k|}}\right)\geqslant 1-\alpha\enspace.
\]
Since $K\leqslant N\alpha /(2\theta_0\param{r0})^2$, $|B_k| = N/K\geqslant (2\theta_0\param{r0})^2/\alpha$ and so
\begin{equation}\label{eq:quad-1}
\bP\left(2\theta_0 P_{B_k}|n_f|\geqslant 1\right)\geqslant 1-\alpha\enspace.
\end{equation}
Let $\phi$ be the function defined on $\R_+$ by $\phi(t)=(t-1)I(1\leqslant t\leqslant 2)+I(t\geqslant 2)$, and, for all $f\in F^*_\rho$ let $Z(f)=\sum_{k\in[K]}I(4\theta_0P_{B_k}|n_f|\geqslant 1)$. Since for all $x\in \R$, $I(x\geqslant 1)\geqslant \phi(x)$,
\begin{equation*}
Z(f)\geqslant \sum_{k\in\cK}\phi\left(4\theta_0P_{B_k}|n_f|\right)\enspace.
\end{equation*}
Now, for any $x\in \R_+$, $\phi(x)\geqslant I(x\geqslant 2)$, thus, according to \eqref{eq:quad-1},
\begin{align*}
\bE\left[\sum_{k\in \cK}\phi\left(4\theta_0P_{B_k}|n_f|\right)\right]\geqslant \sum_{k\in \cK}\bP\left(4\theta_0P_{B_k}|n_f|\geqslant 2\right)\geqslant |\cK|(1-\alpha)\enspace.
\end{align*}
Therefore,
\begin{align*}
 Z(f)\geqslant |\cK|(1-\alpha)+\sum_{k\in \cK}\left(\phi\left(4\theta_0P_{B_k}|n_f|\right)-\bE\left[\phi\left(4\theta_0P_{B_k}|n_f|\right)\right]\right)\enspace.
\end{align*}
Denote $\cF=\{f\in F : \norm{f-\bayes}\leqslant \rho,\; \norm{f-\bayes}_{L^2_P}\geqslant r_Q(\rho,\gamma_Q)\}$. By the bounded difference inequality (see, for instance \cite[Theorem~6.2]{BouLugMass13}), there exists an event $\Omega_Q(K)$ with probability larger than $1-\exp(- x^2|\cK|/2)$, on which, for all $f\in\cF$, 
\begin{equation*}
\sup_{f\in \cF}\left|\sum_{k\in \cK}\left(\phi\left(4\theta_0P_{B_k}|n_f|\right)-\bE\left[\phi\left(4\theta_0P_{B_k}|n_f|\right)\right]\right)\right|
\leqslant \E\sup_{f\in \cF}\left| \sum_{k\in \cK}\left(\phi\left(4\theta_0P_{B_k}|n_f|\right)-\bE\left[\phi\left(4\theta_0P_{B_k}|n_f|\right)\right]\right)\right|+|\cK|x\enspace.
\end{equation*}
By the symmetrization argument,
\begin{equation*}
 \E\sup_{f\in \cF  } \left|\sum_{k\in \cK}\left(\phi\left(4\theta_0P_{B_k}|n_f|\right)-\bE\left[\phi\left(4\theta_0P_{B_k}|n_f|\right)\right]\right)\right|
 \le2\E\sup_{f\in \cF} \left|\sum_{k\in \cK}\epsilon_k\phi\left(4\theta_0P_{B_k}|n_f|\right)\right| \enspace.
\end{equation*}
Since the function $\phi$ is 1-Lipschitz and $\phi(0)=0$, by the contraction principle (see, for example \cite[Chapter~4]{MR1849347} or \cite[Theorem 11.6]{BouLugMass13}), we have
\begin{equation*}
\E\sup_{f\in \cF
 } \left|\sum_{k\in \cK}\epsilon_k\phi\left(4\theta_0P_{B_k}|n_f|\right)\right| \leqslant 4\theta_0\E\sup_{f\in \cF
  } \left|\sum_{k\in \cK}\epsilon_kP_{B_k}|n_f|\right|. 
\end{equation*}The family $(\eps_{[i]}|n_f(X_i)|:i\in\cup_{k\in\cK}B_k)$,  where $[i]=\lceil i/K\rceil$ for all $i\in \cI$, is a collection of centered random variables. Therefore, if $(\eps^\prime_{k})_{k\in\cK}$ and $(X_i^\prime)_{i\in\cI}$ denote independent copies of  $(\eps_{k})_{k\in\cK}$ and $(X_i)_{i\in\cI}$ then 
\begin{align*}
\E\sup_{f\in \cF
  } \left|\sum_{k\in \cK}\epsilon_kP_{B_k}|n_f|\right| \leqslant \E\sup_{f\in \cF
  } \left|\sum_{k\in \cK}\frac{1}{|B_k|}\sum_{i\in B_k}\epsilon_k |n_f(X_i)| - \epsilon_k^\prime |n_f(X_i^\prime)|\right|.
\end{align*}Then, as $(X_i)_{i\in\cI}$ and $(X_i^\prime)_{i\in\cI}$ are two independent families of independent variables therefore, if $(\eps_i^{\prime\prime})_{i\in\cI}$ denote a family of i.i.d. Rademacher variables independent of $(\eps_i), (\eps_i^{\prime}), (X_i)_{i\in\cI}, (X_i^\prime)_{i\in\cI}$ then $(\epsilon_k |n_f(X_i)| - \epsilon_k^\prime |n_f(X_i^\prime)|)$ and $\left(\eps_i^{\prime\prime}\left(\epsilon_k |n_f(X_i)| - \epsilon_k^\prime |n_f(X_i^\prime)|\right)\right)$ have the same distribution. Therefore, 
\begin{align*}
 &\E\sup_{f\in \cF
  } \left|\sum_{k\in \cK}\frac{1}{|B_k|}\sum_{i\in B_k}\eps_k |n_f(X_i)| - \eps_k^\prime |n_f(X_i^\prime)|\right|\leqslant \E\sup_{f\in \cF
  } \left|\sum_{k\in \cK}\frac{1}{|B_k|}\sum_{i\in B_k}\eps_i^{\prime\prime}\left(\eps_k |n_f(X_i)| - \eps_k^\prime |n_f(X_i^\prime)|\right)\right|\\
  &= \E\sup_{f\in \cF
  } \left|\sum_{k\in \cK}\frac{1}{|B_k|}\sum_{i\in B_k}\eps_i^{\prime\prime}\left(|n_f(X_i)| -  |n_f(X_i^\prime)|\right)\right|\leqslant \frac{2K}{N}\E\sup_{f\in \cF} \left|\sum_{i\in\cup_{k\in \cK}B_k}\epsilon_in_f(X_i)\right|.
 \end{align*} 
By the contraction principle, on $\Omega_Q(K)$,
\begin{equation}\label{eq:quad-2}
 Z(f)\geqslant |\cK|\left(1-\alpha-x-\frac{16\theta_0 K}{|\cK| N}\E\sup_{f\in \cF
%  : \norm{f-\bayes}_{L^2}\ge \sqrt{\frac DN}
} \left|\sum_{i\in\cup_{k\in \cK}B_k}\epsilon_in_f(X_i)\right| \right)\enspace.
\end{equation}
For any $f\in \cF$, $r_{Q}(\rho,\gamma_Q)n_f + f^*\in F$ because $F$ is convex. Moreover,  $\norm{r_{Q}(\rho,\gamma_Q)n_f}_{L^2_{P}}=r_Q(\rho,\gamma_Q)$ and $\norm{r_{Q}(\rho,\gamma_Q)n_f} = [r_Q(\rho, \gamma_Q)/\norm{f-f^*}_{L^2_P}]\norm{f-f^*}\leqslant \rho$. Therefore, $r_{Q}(\rho,\gamma_Q)n_f + f^* \in\cF$.  Therefore, by definition of $r_Q(\rho,\gamma_Q)$,
\[
\E\sup_{f\in \cF
 } \left|\sum_{i\in\cup_{k\in \cK}B_k}\epsilon_in_f(X_i)\right|=\frac1{r_Q(\rho,\gamma_Q)}\E\sup_{f\in F : \norm{f-\bayes}\leqslant \rho,\;\norm{f-\bayes}_{L^2_P}=r_Q(\rho,\gamma_Q)
 } \left|\sum_{i\in\cup_{k\in \cK}B_k}\epsilon_i(f-\bayes)(X_i)\right| \leqslant \gamma_Q\frac{|\cK|N}{K}\enspace.
\]Using the last inequality together with \eqref{eq:quad-2} and the assumption $K\geqslant |\cO|/(1-\gamma)$ (so that $|\cK|\geqslant K-|\cO|\geqslant \gamma K$), we get that, on the event $\Omega_Q(K)$, for any $f\in \cF$,
\begin{align*}
 Z(f)\geqslant |\cK|\left(1-\alpha-x-16\theta_0\gamma_Q \right)\geqslant  (1-\eta)K\enspace.
\end{align*}
Hence, on $\Omega_Q(K)$, for any $f\in \cF$,
there exists at least $(1-\eta)K$ blocks $B_k$ for which $P_{B_k}|n_f|\geqslant (4\theta_0)^{-1}$. 
%This means that, on $\Omega_Q(K)$, $Q_{\eta,K}[|n_f|]\ge (4\theta_0)^{-1}$. 
On these blocks, $P_{B_k}n_f^2\geqslant (P_{B_k}|n_f|)^2\geqslant (4\theta_0)^{-2}$, therefore, on $\Omega_Q(K)$, $Q_{\eta,K}[n_f^2]\geqslant (4\theta_0)^{-2}$.
\end{proof}

Now, let us turn to a control of the multiplier process.

\begin{Lemma}\label{lem:proc_multiplicatif} Grant Assumption~\ref{ass:margin}. Fix $\eta\in (0,1)$, $\rho\in(0,+\infty]$, and let $\alpha,\gamma_M,\gamma,x$ and $\eps$ be positive  absolute constants such that $\gamma\left(1-\alpha-x-8 \gamma_M/\eps\right) \geqslant 1-\eta$. 
Let $K\in[|\cO|/(1-\gamma),N]$. There exists an event $\Omega_M(K)$ such that $\bP(\Omega_M(K))\geqslant 1-\exp(-\gamma K x^2/2)$ and on the event $\Omega_M(K)$: if $f\in F$ is such that  $\norm{f-\bayes}\leqslant \rho$ then
\begin{align*}
 \left|\left\{k\in \cK : \left|2(P_{B_k}-\overline P_{B_k})(\zeta(f-\bayes))\right| \leqslant \eps\max\left(\frac{16 \theta_m^2}{\eps^2\alpha}\frac{K}{N}, r^2_M(\rho, \gamma_M), \norm{f-\bayes}^2_{L^2_P}\right)\right\}\right|\geqslant(1-\eta)K \enspace.
\end{align*} 
\end{Lemma}

\begin{proof}
For all $k\in [K]$ and $f\in F$, set $W_k=((X_i,Y_i))_{i\in B_k}$ and define
\begin{equation*}
g_f(W_k)=2(P_{B_k}-\overline P_{B_k})\left(\zeta(f-\bayes)\right)\an \gamma_k(f)= \eps\max\left(\frac{16 \theta_m^2}{\eps^2\alpha}\frac{K}{N}, r^2_M(\rho, \gamma_M), \norm{f-\bayes}^2_{L^2_P}\right)\enspace.
\end{equation*}
Let $f\in F$ and $k\in \mathcal K$. It follows from Markov's inequality that
\begin{align}\label{eq:multi_1}
\nonumber \bP&\left[2\Big|g_f(W_k)\Big|\geqslant \gamma_k(f)\right]\leqslant \frac{4\E \left[\Big(2(P_{B_k}-\overline P_{B_k})(\zeta (f-\bayes))\Big)^2\right]}{ \frac{16\param{m}^2}{\alpha}\norm{f-\bayes}_{L^2_P}^2\frac{K}N}\\
& \leqslant \frac{\alpha \sum_{i\in B_k}{\rm var}_{P_i}(\zeta (f-\bayes))}{|B_k|^2 \param{m}^2\norm{f-\bayes}_{L^2_P}^2\frac{K}N} \leqslant \frac{\alpha \param{m}^2 \norm{f-\bayes}_{L^2_P}^2}{|B_k|  \param{m}^2\norm{f-\bayes}_{L^2_P}^2\frac{K}N} =\alpha \enspace.
\end{align}
%The last inequality follows from $K\leqslant c r^2_M(\rho)N$. 
%Hence, $\bP\left[\Big|g_f(W_k)\Big|< \gamma_k(f)/2\right]\ge 1-\alpha$. 

Let $J=\cup_{k\in\mathcal K}B_k$ and let $r_M(\rho) = r_M(\rho,\gamma_M)$. We have
\begin{align*}
 \E&\sup_{f\in B(f^*, \rho)}\sum_{k\in\mathcal K}\epsilon_k\frac{g_f(W_k)}{\gamma_k(f)}
\leqslant 2\E \sup_{f\in B(f^*, \rho)}\left|\sum_{k\in\mathcal K}  \frac{\eps_k(P_{B_k}-\overline P_{B_k})(\zeta (f-\bayes))}{\eps\max(r_M^2(\rho), \norm{f-\bayes}^2_{L^2_P})}\right|\\
&\leqslant \frac{2}{\epsilon r_M^2(\rho)}
\E \left[\sup_{f\in B(f^*, \rho): \norm{f-\bayes}_{L^2_P}\geqslant r_M(\rho)}\left|\sum_{k\in\mathcal K} \eps_k(P_{B_k}-\overline P_{B_k})\left(\zeta r_M(\rho)\frac{f-\bayes}{\norm{f-\bayes}_{L^2_P}}\right)\right|\right.\\
&\qquad\qquad\qquad\;\left.\vee\sup_{f\in B(f^*, \rho): \norm{f-\bayes}_{L^2_P}\leqslant r_M(\rho)}\left|\sum_{k\in\mathcal K} \eps_k(P_{B_k}-\overline P_{B_k})\left(\zeta (f-\bayes)\right)\right|\right]\\
&\leqslant \frac{2}{\epsilon r_M^2(\rho)}
\E \sup_{f\in B(f^*, \rho): \norm{f-\bayes}_{L^2_P}\leqslant r_M(\rho)}\left|\sum_{k\in\mathcal K} \eps_k(P_{B_k}-\overline P_{B_k})\left(\zeta (f-\bayes)\right)\right|\enspace,
\end{align*}where in the last but one inequality, we used that the class $F$ is convex and the same argument as in the proof of Lemma~\ref{lem:UBQPReg}. Since $(\eps_{[i]}(\zeta_i(f-f^*)(X_i) -P_i\zeta_i(f-f^*)):i\in\cI)$ is a family of centered random variables, one can use the symmetrization argument to get
\begin{align}\label{eq:multi-2}
 \nonumber \E\sup_{f\in B(f^*, \rho)}\sum_{k\in\mathcal K}\epsilon_k\frac{g_f(W_k)}{\gamma_k(f)}
&\leqslant \frac{4 K}{\epsilon r_M^2(\rho)N} \E \sup_{f\in B(f^*, \rho) : \norm{f-\bayes}_{L^2_P}\leqslant r_M(\rho)}\left|\sum_{i\in J} \eps_{i}\zeta_i (f-\bayes)(X_i)\right|\\
&\leqslant \frac{4 K}{\epsilon N}\gamma_M|\cK|\frac NK=\frac{4\gamma_M}{\epsilon}|\cK|\enspace,
\end{align}where the definition of $r_M(\rho)$ has been used in the last but one inequality.

Let $\psi(t) = (2t-1)I(1/2\leqslant t \leqslant 1)+I(t\geqslant 1)$. The function $\psi$ is $2$-Lipschitz and satisfies $I(t\geqslant 1)\leqslant \psi(t)\leqslant I(t\geqslant 1/2)$, for all $t\in \R$. Therefore, all $f\in B(f^*, \rho)$ satisfies
\begin{align*}
\sum_{k\in \cK}I&\left(|g_f(W_k)| < \gamma_k(f)\right)=|\cK|-\sum_{k\in \cK} I\left(\frac{|g_f(W_k)|}{\gamma_k(f)}\geqslant 1\right)\geqslant|\cK|-\sum_{k\in \cK}\psi\left(\frac{|g_f(W_k)|}{\gamma_k(f)}\right)\\
&=|\cK|-\sum_{k\in\cK}\bE \psi\left(\frac{|g_f(W_k)|}{\gamma_k(f)}\right)-\sum_{k\in \cK}\left[\psi\left(\frac{|g_f(W_k)|}{\gamma_k(f)}\right)-\bE \psi\left(\frac{|g_f(W_k)|}{\gamma_k(f)}\right)\right]\\
&\geqslant |\cK|-\sum_{k\in\cK}\bE I\left(\frac{|g_f(W_k)|}{\gamma_k(f)}\geqslant \frac12\right)-\sum_{k\in \cK}\left[\psi\left(\frac{|g_f(W_k)|}{\gamma_k(f)}\right)-\bP\psi\left(\frac{|g_f(W_k)|}{\gamma_k(f)}\right)\right]\\
&\geqslant(1-\alpha)|\cK|-\sup_{f\in B(f^*, \rho)}\left|\sum_{k\in \cK}\left[\psi\left(\frac{|g_f(W_k)|}{\gamma_k(f)}\right)-\bE \psi\left(\frac{|g_f(W_k)|}{\gamma_k(f)}\right)\right]\right|
\end{align*}where we used \eqref{eq:multi_1} in the last inequality.

The bounded difference inequality ensures that there exists an event $\Omega_M(K)$ satisfying $\bP(\Omega_M(K))\geqslant 1-\exp(-x^2|\cK|/2)$, where
\begin{multline*}
\sup_{f\in B(f^*, \rho)} \left|\sum_{k\in \cK}\left[\psi\left(\frac{|g_f(W_k)|}{\gamma_k(f)}\right)-\bE\psi\left(\frac{|g_f(W_k)|}{\gamma_k(f)}\right)\right]\right|\\
\leqslant \bE \sup_{f\in B(f^*, \rho)}\left|\sum_{k\in \cK}\left[\psi\left(\frac{|g_f(W_k)|}{\gamma_k(f)}\right)-\bE\psi\left(\frac{|g_f(W_k)|}{\gamma_k(f)}\right)\right]\right|+|\cK|x\enspace.
\end{multline*}
Furthermore, it follows from by the symmetrization argument that
\[
\bE \sup_{f\in B(f^*, \rho)}\left|\sum_{k\in \cK}\left[\psi\left(\frac{|g_f(W_k)|}{\gamma_k(f)}\right)-\bE\psi\left(\frac{|g_f(W_k)|}{\gamma_k(f)}\right)\right]\right|\leqslant 2\bE\sup_{f\in B(f^*, \rho)}\left|\sum_{k\in \cK}\epsilon_k\psi\left(\frac{|g_f(W_k)|}{\gamma_k(f)}\right)\right|
\]
and, from the contraction principle and \eqref{eq:multi-2}, that
\[
\bE\sup_{f\in B(f^*, \rho)}\left|\sum_{k\in \cK}\epsilon_k\psi\left(\frac{|g_f(W_k)|}{\gamma_k(f)}\right)\right|\leqslant 2\bE \sup_{f\in B(f^*, \rho)}\left|\sum_{k\in \cK}\epsilon_k\frac{|g_f(W_k)|}{\gamma_k(f)}\right| \leqslant \frac{8\gamma_M}{\eps}|\cK|\enspace.
\] In conclusion, on $\Omega_M(K)$, for all $f\in B(f^*, \rho)$,
\[
\sum_{k\in \cK}I\left(|g_f(W_k)|< \gamma_k(f)\right)\geqslant \left(1-\alpha-x-8 \gamma_M/\eps\right)|\cK|\geqslant K \gamma \left(1-\alpha-x-8 \gamma_M/\eps\right)\geqslant (1-\eta)K\enspace.
\]
\end{proof}

\subsection{Bounding the empirical criterion $\Crit{K,\lambda}{f^*}$}
Let us first introduce the event on which the statement of Theorem~\ref{thm:RBRhoEstPen} holds. Denote by $\Omega(K)$ the intersection of the events $\Omega_Q(K)$, $\Omega_M(K)$ defined respectively in Lemmas~\ref{lem:UBQPReg} and~\ref{lem:proc_multiplicatif} for $\rho \in\{\kappa\rho_K:\kappa\in\{1,2\}\}$ and
\begin{equation}\label{eq:def-params}
\eta=\frac14, \gamma = \frac78, \alpha=\frac1{24}, x=\frac1{24}, \gamma_Q = \frac{1}{384\param{0}}, \eps = \frac{1}{c\param{0}^{2}} \mbox{ and } \gamma_M = \frac{\eps}{192}
\end{equation}for some absolute constants $c>0$ to be specified later. For these values, conditions in both Lemmas~\ref{lem:UBQPReg} and~\ref{lem:proc_multiplicatif} are satisfied:
\begin{equation*}
\gamma(1-\alpha-x-16\gamma_Q \theta_0)\geqslant 1-\eta=\frac34 \mbox{ and } \gamma(1-\alpha-x-8\gamma_M/\eps)\geqslant 1-\eta=\frac34.
\end{equation*} According to Lemmas~\ref{lem:UBQPReg} and~\ref{lem:proc_multiplicatif}, the event $\Omega(K)$ satisfies $\bP(\Omega(K))\geqslant 1-4\exp\left(- 7K/9216\right)$.
%
%Let us know establish useful properties satisfied by the quadratic and multiplier MOM processes on the event $\Omega(K)$.  
On $\Omega(K)$, the following holds for all $\rho \in\{\kappa\rho_K:\kappa\in\{1, 2\}\}$ and $f\in F$ such that $\norm{f-\bayes}\leqslant \rho$,
\begin{enumerate}
	\item  if $\norm{f-\bayes}_{L^2_P}\geqslant r_Q(\rho,\gamma_Q)$ then
\begin{equation}\label{eq:LBQPReg}
 Q_{1/4,K}((f-\bayes)^2)\geqslant \frac1{(4\theta_0)^2}\norm{f-\bayes}^2_{L^2_P}\enspace,
\end{equation}
\item there exists $3K/4$ block $B_k$ with $k\in \cK$, for which
\begin{equation}\label{eq:ContMultReg}
|(P_{B_k}-\overline{P}_{B_k})[2\zeta(f-\bayes)]|\leqslant \eps\max\left(r_M^2(\rho, \gamma_M), \frac{384\theta_m^2}{\eps^2}\frac{K}{N}, \norm{f-f^*}_{L_p^2}^2\right)\enspace. 
\end{equation}
\end{enumerate}

Moreover, on the blocks $B_k$ where \eqref{eq:ContMultReg} holds, it follows from assumption in \eqref{eq:robust_theo_basic} that all $f\in F$ such that $\norm{f-\bayes}\leqslant \rho$ satisfies
\begin{equation*}
P_{B_k}[2\zeta(f-\bayes)]|\leqslant P[2\zeta(f-\bayes)]+2\eps\max\left(r_M^2(\rho, \gamma_M), \frac{384 \theta_m^2}{\eps^2}\frac{K}{N}, \norm{f-f^*}_{L_p^2}^2\right) \enspace.
\end{equation*}
It follows from the convexity of $F$ and the nearest point theorem that $P[2\zeta(f-\bayes)]\leqslant 0$ for all $f\in F$, therefore, for all $f\in F$ such that $\norm{f-\bayes}\leqslant \rho$,
\begin{align}
\label{eq:UBPMNC} Q_{3/4,K}(2\zeta(f-\bayes))&\leqslant 2\eps\max\left(r_M^2(\rho, \gamma_M), \frac{384 \theta_m^2}{\eps^2}\frac{K}{N}, \norm{f-f^*}_{L_p^2}^2\right)\enspace.
\end{align}
Moreover, still on the blocks $B_k$ where \eqref{eq:ContMultReg} holds, one also has, thanks to assumption \eqref{eq:robust_theo_basic}, that for all $f\in F$ such that $\norm{f-\bayes}\leqslant \rho$,
\begin{equation*}
P[-2\zeta(f-\bayes)]\leqslant P_{B_k}[-2\zeta(f-\bayes)]+2\eps\max\left(r_M^2(\rho, \gamma_M), \frac{384 \theta_m^2}{\eps^2}\frac{K}{N}, \norm{f-f^*}_{L_p^2}^2\right)\enspace. 
\end{equation*}
It follows that, for all $f\in F$ such that $\norm{f-\bayes}\leqslant \rho$,
\begin{align}
\notag &P[-2\zeta(f-\bayes)]\leqslant Q_{1/4,K}[-2\zeta(f-\bayes)]+2\eps\max\left(r_M^2(\rho, \gamma_M), \frac{384 \theta_m^2}{\eps^2}\frac{K}{N}, \norm{f-f^*}_{L_p^2}^2\right)\\
\notag&\leqslant Q_{1/4,K}[(f-\bayes)^2-2\zeta(f-\bayes)]+\lambda(\norm{f}-\norm{\bayes})+2\eps\max\left(r_M^2(\rho, \gamma_M), \frac{384 \theta_m^2}{\eps^2}\frac{K}{N}, \norm{f-f^*}_{L_p^2}^2\right)+\lambda\rho\\
\label{eq:ToutEstDansLeTestReg}&\leqslant T_{K,\lambda}(\bayes,f)+2\eps\max\left(r_M^2(\rho, \gamma_M), \frac{384 \theta_m^2}{\eps^2}\frac{K}{N}, \norm{f-f^*}_{L_p^2}^2\right)+\lambda\rho\enspace.
\end{align}

The main result of this section is Lemma~\ref{lem:BoundSup}. It will be used to bound from above the criterion $\Crit{K,\lambda}{f^*} = \sup_{g\in F} T_{K, \lambda}(g,f^*)$. 
%We recall that $r(\cdot)$ is a continuous and non-decreasing function such that for all $\rho\geqslant0$, $r(\rho)\geqslant \max\left(r_M(\rho, \gamma_M), r_Q(\rho, \gamma_Q)\right)$. 
Recall that $\rho_K$ and $\lambda$ are defined as
\begin{equation}\label{eq:important-quantities}
 r^2(\rho_K) =  \frac{384 \theta_m^2}{\eps^2}\frac{K}{N} \mbox{ and } \lambda = \frac{c^\prime\eps r^2(\rho_K)}{\rho_K}
\end{equation}
where $\eps = (c \theta_0^2)^{-1}$ and $c, c^\prime\geqslant$ are absolute constants. We also need to consider a partition of the space $F$ according to the distance between $g$ and $f^*$ w.r.t. $\norm{\cdot}$ and $\norm{\cdot}_{L_P^2}$ as in Figure~\ref{fig:partition_set_F_positive_excess_loss}: define for all $\kappa\geqslant1$,
\begin{align*}
 F_1^{(\kappa)}&=\left\{g\in F : \norm{g-\bayes}\leqslant \kappa\rho_K \mbox{ and } \norm{g-\bayes}_{L^2_P}\leqslant  r(\kappa\rho_K) \right\}\enspace,\\
 F_2^{(\kappa)}&=\left\{g\in F : \norm{g-\bayes}\leqslant \kappa\rho_K \mbox{ and } \norm{g-\bayes}_{L^2_P}>  r( \kappa\rho_K)\right\}\enspace,\\
 F_3^{(\kappa)}&=\left\{g\in F : \norm{g-\bayes}> \kappa\rho_K\right\}\enspace.
\end{align*}

\input{dessin_1}

\begin{Lemma}\label{lem:BoundSup}
%We recall that $T_{K, \lambda}(g, f^*) = \MOM{K}{\ell_{f^*}-\ell_g} + \lambda\left(\norm{f^*} - \norm{g}\right)$ for all $g\in F$. 
On the event $\Omega(K)$, it holds for all  $\kappa\in\{1, 2\}$, 
\begin{align*}
 \sup_{g\in F_1^{(\kappa)}} T_{K, \lambda}(g, f^*) \leqslant (2 + c^\prime \kappa) \eps r^2(\kappa \rho_K), \quad  \sup_{g\in F_2^{(\kappa)}} T_{K, \lambda}(g, f^*)\leqslant \left((2+c^\prime \kappa)\eps- \frac{1}{16 \theta_0^2}\right) r^2(\kappa\rho_K)
\end{align*}and
\begin{equation*}
\sup_{g\in F_3^{(\kappa)}}T_{K, \lambda}(g, f^*)\leqslant \kappa\max\left(2\eps-\frac{1}{16\theta_0^2} + \frac{11c^\prime\eps }{10}, 2\eps -\frac{7c^\prime\eps}{10} \right)  r^2( \rho_K)
\end{equation*}when $c\geqslant 32$ and $10\eps/4\leqslant c^\prime\eps\leqslant ((4\theta_0)^{-2} - 2\eps)$. 
\end{Lemma}
% \begin{Remark}
%  In the case $\lambda=0$, the previous result holds with the conventions $F_2^{(\kappa)}=F_4^{(\kappa)}=\emptyset$ and $\sup_{\emptyset}=-\infty$. The proof is straightforward from the following proof using the convention $\lambda\rho_K=0$ if $\lambda=0$.
% \end{Remark}

\noindent\textbf{Proof of Lemma~\ref{lem:BoundSup}.} Recall that, for all $g\in F$, $\ell_{f^*}-\ell_{g} = 2 \zeta (g-f^*) - (g-f^*)^2$ where $\zeta(x, y) = y-f^*(x)$. Let us now place ourself on the event $\Omega(K)$ up to the end of proof. 

\paragraph{Bounding $\mathbf{\sup_{g\in F_1^{(\kappa)}}T_{K, \lambda}(g, f^*)}$.} Let $g\in F_{1}^{(\kappa)}$. Since the quadratic process is non negative,
\begin{align*}
T_{K, \lambda}(g, f^*)  &= \MOM{K}{2\zeta(g-\bayes)-(g-\bayes)^2}-\lambda\left(\norm{g}-\norm{\bayes}\right)\leqslant Q_{3/4,K}(2\zeta (g-\bayes))+\lambda\norm{\bayes-g}\enspace.
\end{align*}
Therefore, applying \eqref{eq:UBPMNC} for $\rho=\kappa\rho_K$ and the choice of $\rho_K$ and $\lambda$ as in \eqref{eq:important-quantities}, we get
\begin{align*}
T_{K, \lambda}(g, f^*)  & \le
2\eps\max\left(r_M^2(\kappa\rho_K, \gamma_M), \frac{384 \theta_m^2}{\eps^2}\frac{K}{N},\norm{f-f^*}_{L_p^2}^2\right) + \lambda  \kappa\rho_K\leqslant  2\eps r^2(\kappa \rho_K) + c^\prime \kappa\eps r^2(\rho_K) \\
&\leqslant (2 + c^\prime \kappa) \eps r^2(\kappa \rho_K)\enspace. 
\end{align*}

\paragraph{Bounding $\mathbf{\sup_{g\in F_2^{(\kappa)}} T_{K, \lambda}(g, f^*)}$.} Let $g\in F_2^{(\kappa)}$. Given that $Q_{1/2}(x-y)\leqslant Q_{3/4}(x) - Q_{1/4}(y)$ for any vector $x$ and $y$, we have
\begin{equation*}
\MOM{K}{2\zeta(g-\bayes)-(g-\bayes)^2}+\lambda\left(\norm{\bayes}-\norm{g}\right)\\
\leqslant Q_{3/4,K}(2\zeta (g-\bayes))-Q_{1/4,K}((\bayes-g)^2)+\lambda\kappa\rho_K\enspace.
\end{equation*}
Moreover $2\eps \leqslant (4\theta_0)^{-2}$ when $c\geqslant32$, so it follows from \eqref{eq:LBQPReg} and \eqref{eq:UBPMNC} for $\rho=\kappa\rho_K$ that
\begin{align*}
Q_{3/4,K}(2\zeta (\bayes-g))-Q_{1/4,K}((\bayes-g)^2)&\leqslant 2\eps\max\left(r_M^2(\kappa\rho_K, \gamma_M), \frac{384 \theta_m^2}{\eps^2}\frac{K}{N}, \norm{f-f^*}_{L_p^2}^2\right) - \frac{\norm{f-f^*}_{L_P^2}^2}{(4\theta_0)^2}\\
&\leqslant \left(2\eps - \frac{1}{(4\theta_0)^2}\right)\norm{f-f^*}_{L_P^2}^2\leqslant \left(2\eps-\frac{1}{16\theta_0^2}\right)  r^2(\kappa\rho_K)\enspace. 
\end{align*}Putting both inequalities together and using that $\lambda \kappa \rho_K= c^\prime\kappa\eps r^2(\rho_K)$, we get 
\begin{equation*}
T_{K, \lambda}(g, f^*)\leqslant \left((2+c^\prime \kappa)\eps - \frac{1}{16 \theta_0^2} \right) r^2(\kappa\rho_K)\enspace.
\end{equation*}

\paragraph{Bounding $\mathbf{\sup_{g\in F_3^{(\kappa)}} T_{K, \lambda}(g, f^*)}$ via an homogeneity argument.} Start with two lemmas.
\begin{Lemma}\label{lem:LBSubGrad}Let $\rho\geqslant0$, $\Gamma_{\bayes}(\rho) = \cup_{f\in f^*+(\rho/20)B}(\partial\norm{\cdot})_f$ (cf.) section~\ref{sub:the_sparsity_equation}). For all $g\in F$,
 \begin{equation*}
\norm{g}-\norm{\bayes}\geqslant\sup_{z^*\in\Gamma_{\bayes}(\rho)}z^*(g-\bayes)-\frac{\rho}{10}\enspace.  
\end{equation*} 
\end{Lemma}
\begin{proof}
Let $g\in F$, $f^{**}\in f^* +(\rho/20)B$ and $z^*\in(\partial\norm{\cdot})_{f^{**}}$. We have 
\begin{equation*}
\norm{g}-\norm{\bayes}\geqslant \norm{g} - \norm{f^{**}} -\norm{f^{**}-\bayes}\geqslant z^*(g-f^{**})-\frac{\rho}{20} = z^*(g-\bayes)-z^*(f^{**}-\bayes)-\frac{\rho}{20}\geqslant z^*(g-\bayes)-\frac{\rho}{10}\enspace,
\end{equation*}
where the last inequality follows from $z^*(f^{**}-\bayes)\leqslant \|f^{**}-\bayes\|$. The result follows by taking supremum over $z^*\in\Gamma_{\bayes}(\rho)$.
\end{proof}

\begin{Lemma}\label{lem:Homogeneity}Let $\rho\geqslant0$. Let $g\in F$ be such that $\norm{g-\bayes}\geqslant \rho$. Define $f=\bayes + \rho(g-\bayes)/\norm{g-\bayes}$. Then $f\in F$, $\norm{f-f^*}=\rho$ and,
\begin{multline*}
\MOM{K}{(g-\bayes)^2-2\zeta(g-\bayes)}+\lambda\sup_{z^*\in\Gamma_{\bayes}(\rho)}z^*(g-\bayes)\\
\geqslant \frac{\norm{g-\bayes}_{L^2_P}}{\rho}\pa{\MOM{K}{(f-\bayes)^2-2\zeta(f-\bayes)}+\lambda\sup_{z^*\in\Gamma_{\bayes}(\rho)}z^*(f-\bayes)}\enspace.
\end{multline*}
\end{Lemma}
\begin{proof}
%Assume from now on that $\|f-\bayes\|\ge \rho$.
The first conclusion holds by convexity of $F$, the second statement is obvious. For the last one, let $\Upsilon= \|g-\bayes\|/\rho$ and note that $\Upsilon\geqslant 1$ and $g-f^* = \Upsilon(f-f^*)$, so we have
\begin{align*}
 &\MOM{K}{(g-\bayes)^2-2\zeta(g-\bayes)}+\lambda\sup_{z^*\in\Gamma_{\bayes}(\rho)}z^*(g-\bayes)\\
 &=\MOM{K}{\Upsilon^2(f-\bayes)^2-2\Upsilon\zeta(f-\bayes)}+\lambda\Upsilon\sup_{z^*\in\Gamma_{\bayes}(\rho)}z^*(f-\bayes)\\
 &\geqslant \Upsilon\left(\MOM{K}{(f-\bayes)^2-2\zeta(f-\bayes)}+\lambda\sup_{z^*\in\Gamma_{\bayes}(\rho)}z^*(f-\bayes)\right)\enspace. 
\end{align*}
\end{proof}
Now, let us bound $\sup_{g\in F_3^{(\kappa)}} T_{K, \lambda}(g, f^*)$. Let $g\in F_3^{(\kappa)}$. Apply Lemma~\ref{lem:LBSubGrad} and Lemma~\ref{lem:Homogeneity} to $\rho=\rho_K$: there exists $f\in F$ such that $\norm{f-\bayes}=\rho_K$ and
\begin{align}\label{eq:F3-main}
\nonumber&T_{K, \lambda}(g, f^*) = \MOM{K}{2\zeta(g-\bayes)-(g-\bayes)^2}-\lambda\left(\norm{g}-\norm{\bayes}\right)\\
\nonumber&\leqslant \MOM{K}{2\zeta(g-\bayes)-(g-\bayes)^2}-\lambda\sup_{z^*\in\Gamma_{\bayes}(\rho_K)}z^*(g-\bayes)+\lambda\frac{\kappa\rho_K}{10}\\
&\leqslant \frac{\norm{g-\bayes}}{\rho_K}\pa{\MOM{K}{2\zeta(f-\bayes)-(f-\bayes)^2}-\lambda\sup_{z^*\in\Gamma_{\bayes}(\rho_K)}z^*(f-\bayes)}+\lambda\frac{\kappa\rho_K}{10}\enspace. 
\end{align}

First assume that $\norm{f-f^*}_{L_P^2}\leqslant  r(\rho_K)$. In that case, $\norm{f-f^*} = \rho_K$ and $\norm{f-f^*}_{L_P^2}\leqslant  r(\rho_K)$ therefore, $f\in H_{\rho_K}$. Moreover, by definition of $K^*$ and since $K\geqslant K^*$, we have $\rho_K\geqslant \rho^*$ which implies that $\rho_K$ satisfies the sparsity equation from Definition~\ref{def:sparsity_equation}. Therefore,  $\sup_{z^*\in\Gamma_{\bayes}(\rho_K)}z^*(f-\bayes)\geqslant \Delta(\rho_K)\geqslant 4\rho_K/5$. Now, it follows from the definition of $\lambda$ in \eqref{eq:important-quantities} that
\[
-\lambda\sup_{z^*\in\Gamma_{\bayes}(\rho_K)}z^*(f-\bayes)\leqslant -\frac{4c^\prime\eps r^2(\rho_K)}{5}\enspace.
\]
Moreover, since the quadratic process is non-negative, by \eqref{eq:UBPMNC} applied to $\rho=\rho_K$,
\begin{align*}
&\MOM{K}{2\zeta(f-\bayes)-(f-\bayes)^2}\leqslant Q_{3/4,K}[2\zeta(f-\bayes)]\\
&\leqslant 2\eps\max\left(r_M^2(\rho_K, \gamma_M), \frac{384 \theta_m^2}{\eps^2}\frac{K}{N}, \norm{f-f^*}_{L_p^2}^2\right) \leqslant 2 \eps r^2(\rho_K) \enspace.
\end{align*} Finally, noting that $2\eps - 4c^\prime\eps/5\leqslant0$ when $c^\prime\geqslant10/4$, binding all the pieces together in \eqref{eq:F3-main} yields 
\begin{align*}
T_{K, \lambda}(g, f^*)\leqslant \kappa\eps\left(2 - 4c^\prime/5\right)r^2( \rho_K)+\lambda\frac{\kappa\rho_K}{10} = \kappa\eps\left(2 -\frac{7c^\prime}{10}\right)  r^2( \rho_K)\enspace. 
\end{align*}

Second, assume that $\norm{f-f^*}_{L_P^2}\geqslant  r(\rho_K)$. Since $\norm{f-f^*} =  \rho_K$, it follows from \eqref{eq:LBQPReg} and \eqref{eq:ContMultReg} for $\rho= \rho_K$ that 
\begin{align*}
&\MOM{K}{2\zeta(f-\bayes)-(f-\bayes)^2}
\leqslant Q_{3/4,K}(2\zeta (f-\bayes))-Q_{1/4,K}((\bayes-f)^2)\\
&\leqslant 2\eps\max\left(r_M^2(\rho_K, \gamma_M), \frac{384 \theta_m^2}{\eps^2}\frac{K}{N}, \norm{f-f^*}_{L_p^2}^2\right) - \frac{\norm{f-f^*}_{L_P^2}^2}{(4\theta_0)^2}\\
&\leqslant \left(2\eps - \frac{1}{16\theta_0^2}\right)\norm{f-f^*}_{L_P^2}^2\leqslant \left(2\eps - \frac{1}{16\theta_0^2}\right) r^2(\rho_K)\enspace, 
\end{align*}where we used that $2\eps \leqslant (16\theta_0)^{-2}$ when $c\geqslant32$ in the last inequality. Plugging the last result in \eqref{eq:F3-main} we get
\begin{align*}
T_{K, \lambda}(g, f^*) &\leqslant \frac{\norm{g-f^*}}{ \rho_K}\left(\left(2\eps-\frac{1}{16\theta^2_0}\right)  r^2( \rho_K) + \lambda\rho_K\right)+\lambda\frac{ \kappa\rho_K}{10}\\
&\leqslant \frac{\norm{g-f^*}}{ \rho_K} \left((2+c^\prime)\eps-\frac{1}{16\theta_0^2}\right)  r^2( \rho_K) + \frac{c^\prime\kappa\eps}{10}r^2( \rho_K)\leqslant \kappa\left(\pa{2+\frac{11c^\prime }{10}}\eps-\frac{1}{16\theta_0^2}\right)  r^2( \rho_K)
\end{align*}when $16(2+c^\prime)\eps\leqslant \theta_0^{-2}$.

\subsection{From a control of $\Crit{K,\lambda}{\hat f}$ to statistical performance} 
\label{sec:EndProof1}
The proof follows essentially the one of \cite[Theorem 3.2]{LM_reg_comp} or \cite[Lemma 2]{MOM1}. 
\begin{Lemma}\label{lem:Conclusion}
 Let $\hat f\in F$ be such that, on $\Omega(K)$, $\Crit{K,\lambda}{\hat f}\leqslant (2+c^\prime) \eps r^2( \rho_K)$. Then, on $\Omega(K)$, $\hat f$ satisfies
  \[
 \norm{\hat f-\bayes}\leqslant 2\rho_K,\quad  \norm{\hat f-\bayes}_{L^2_P}\leqslant r(2\rho_K)\quad \text{and}\quad R(\hat f)\leqslant R(\bayes)+ (1+(4+3 c^\prime)\eps) r^2(2\rho_K)\enspace,
 \]when $c^\prime = 16$ and $c>832$.
\end{Lemma}
\begin{proof}
Recall that for any $x\in\R^K$, $Q_{1/2}(x)\geqslant - Q_{1/2}(-x)$. Therefore,  
\[
\Crit{K,\lambda}{\hat f}=\sup_{g\in F}T_{K,\lambda}(g,\hat f)\geqslant T_{K,\lambda}(\bayes,\hat f)\geqslant -T_{K,\lambda}(\hat f,\bayes)\enspace.
\]
Thus, on $\Omega(K)$,
$\hat f\in\left\{g\in F : T_{K,\lambda}(g,\bayes)\geqslant -(2+c^\prime) \eps r^2( \rho_K)\right\}$. When $c^\prime = 16$ and $c>832$,  
\begin{equation*}
-(2+c^\prime)\eps > 2(1+c^\prime)\eps - \frac{1}{16\theta_0^2}  \mbox{ and } -(2+c^\prime)\eps> 2\max\left(2\eps - \frac{1}{16 \theta_0^2} + \frac{11c^\prime\eps}{10}, 2\eps - \frac{7 c^\prime\eps}{10}\right)
\end{equation*}therefore, $\hat f\in F_1^{(2)}$ on $\Omega(K)$. This yields the results for both the regularization and the $L^2_P$-norm. 

Finally, let us turn to the control on the excess risk. It follows from \eqref{eq:ToutEstDansLeTestReg} for $\rho=\kappa\rho_K$ that
\begin{align*}
&R(\hat f)-R(\bayes)= \norm{\hat f-\bayes}_{L^2_P}^2+P[-2\zeta(\hat f-\bayes)]\\
&\leqslant r^2(2\rho_K)+ T_{K,\lambda}(\bayes, \hat f)+2\eps\max\left(r_M^2(2 \rho_K, \gamma_M), \frac{384 \theta_m^2}{\eps^2}\frac{K}{N}, \norm{\hat f-f^*}_{L_p^2}^2\right)+2\lambda\rho_K\\
&\leqslant  r^2(2\rho_K) + \Crit{K,\lambda}{\hat f} + 2 \eps r^2(2\rho_K) +  2 c^\prime \eps r^2(\rho_K) = (1+(4+3 c^\prime)\eps) r^2(2\rho_K) \enspace.
\end{align*}
\end{proof}

\subsection{End of the proof of Theorem~\ref{thm:RBRhoEstPen}}
By definition of $\ERM{K,\lambda}$,
\begin{align*}
 \Crit{K,\lambda}{\ERM{K,\lambda}}&\le\Crit{K,\lambda}{\bayes}=\sup_{g\in F}T_{K, \lambda}(g,f^*)\le\max_{i\in[3]}\sup_{g\in F^{(1)}_i} T_{K, \lambda}(g,f^*),
\end{align*}where $\{F^{(1)}_1, F^{(1)}_2, F^{(1)}_3\}$ is the decomposition of $F$ as in Figure~\ref{fig:partition_set_F_positive_excess_loss}. It follows from Lemma~\ref{lem:BoundSup} (for $\kappa=1$) that on the event $\Omega(K)$,
\[
 \Crit{K,\lambda}{\ERM{K,\lambda}}\leqslant (2+c^\prime)\eps r^2( \rho_K)\enspace.
\]
Therefore, for $c^\prime = 16$ and $c=833$ the conclusion of the proof of Theorem~\ref{thm:RBRhoEstPen} follows from Lemma~\ref{lem:Conclusion}.

\subsection{Proof of Theorem~\ref{thm:LepskiReg}}
Define
\begin{equation*}
K_1 = \frac{|\cO|}{1-\gamma} = 8 |\cO| \mbox{ and } K_2 = \frac{N \alpha}{(2\theta_0\theta_{r0})^2} = \frac{N}{96 (\theta_0\theta_{r0})^2}.
\end{equation*}

Let $K\in [K_1,K_2]$ and let $\Omega_{K,c_{ad}}=\{\bayes\in  \cap_{J=K}^{K_2}\hat R_{J,c_{ad}}\}$ where we recall that $\hat R_{J,c_{ad}} = \{f\in F: \cC_{J, \lambda}(f)\leqslant (c_{ad}/\theta_0^2)r^2(\rho_J)\}$. Lemma~\ref{lem:BoundSup} (for $\kappa=1$) shows that, for $c_{ad}=(2+c^\prime)/c$, $\Omega_{K,c_{ad}}\supset \cap_{J=K}^{K_2}\Omega(J)$. Therefore, on $\cap_{J=K}^{K_2}\Omega(J)$, $\hat K_{c_{ad}}\leqslant K$ which implies that $\ERM{c_{ad}}\in \hat R_{K,c_{ad}}$. By Lemma~\ref{lem:Conclusion} (for $c^\prime = 16$ and $c=833$), this implies that 
\[
\norm{\ERM{c_{ad} }-\bayes}\leqslant 2\rho_K,\qquad \norm{\ERM{c_{ad} }-\bayes}_{L^2_P}\leqslant r(2\rho_K)\quad\text{and}\quad R(\ERM{c_{ad}})\leqslant R(\bayes)+(1+(4+3 c^\prime)\eps) r^2(2\rho_K)\enspace.
\]

% subsection proof_of_theorem_thm:expectation_results (end)

\appendix

%\section{Tools}
%\subsection{The robustification Lemma}
%\begin{Lemma}\label{lem:RobustLemma}
% Let $\Omega_1,\ldots,\Omega_K$ denote independent events, let $\alpha,p\in(0,1)$ and assume that $2p^\alpha<1$ and $\max_{1\leqslant i\leqslant K}P(\Omega_i^c)\leqslant p$. Then
% \[
% P\left(\sum_{i=1}^K I_{ \Omega_i}\ge (1-\alpha)K\right)\ge 1-(2p^\alpha)^K\enspace.
% \]
%\end{Lemma}
%\begin{proof}
%Since, by hypothesis, the events $\Omega_i^c$ are independent with probability upper bounded by $p$, $\sum_{i=1}^K I_{ \Omega_i^c}$ is stochastically dominated by the binomial distribution $\mathcal B(K,p)$ (as a simple coupling argument would show). Therefore,
%\[
% P\left(\sum_{i=1}^K I_{ \Omega_i^c}\ge \alpha K\right)\leqslant \sum_{k=\alpha K}^K\binom{K}{k}p^k(1-p)^k\leqslant (2p^\alpha)^K\enspace,
% \]
% where the last bound is obtained by bounding from above all $(1-p)^{K-k}$ by $1$, all $p^k$ by $p^{\alpha K}$ and the sum of the binomial coefficients by $2^K$. The result follows by taking the complementary event.
%\end{proof}
%

\section{Simulation study} % (fold)
\label{sec:simulation_study}
The aim of this section is to show that the min-max procedure introduced in this work can be computed using alternating descent-ascent algorithms. It appears that all the following algorithms can be recast as \textbf{block gradient descent (BGD)} but the major difference with the classical BGD is that blocks are chosen according to their ``centrality'' via the median operator instead of iterating over all the blocks (resp. at random) in the classical (resp. stochastic) BGD approach. 

We test our algorithms in a high-dimensional framework. In this setup, the $\ell_1$-norm has played a prominent role through the LASSO estimator. There has been a huge number of algorithms designed to implement the LASSO. The aim of this section is to show that there is a natural equivalent formulation to all of these algorithms for the MOM-LASSO that makes them more robust to outliers as can be appreciated in Figure~\ref{fig:robustness}. The choice of hyper-parameters like the number of blocks or the regularization parameter cannot be done via classical Cross-Validation approaches because of the potential presence of outliers in the test sets, CV procedures are adapted using MOM estimators.  We also advocate for using random blocks at every iterations of the algorithms. This bypasses a problem of ``local saddle points'' we have identified. A by product of the latter approach is a definition of depth adapted to the learning task and therefore of an outliers detection algorithm.

\subsection{Data generating process and corruption by outliers} % (fold)
\label{sub:data_generating_process_and_corruption_by_outliers}
We study the performance of all the algorithms on a dataset corrupted by outliers of various forms. The ``basic'' set of ``informative data'' is called $\cD_1$. This set is corrupted by various datasets of outliers, named $\cD_2, \cD_3$, $\cD_4$ and $\cD_5$. Good and bad data have been merged and shuffled in the dataset $\cD= \cD_1 \cup \cD_2\cup \cD_3\cup \cD_4$ given to the statistician. 
Let us now detail the construction of these datasets.
\begin{enumerate}
	\item The set $\cD_1$ of ``informative data'' is a set of $N_{good}$ i.i.d. data $(X_i, Y_i)$ with common distribution
	\begin{equation}\label{eq:gauss_model_simus}
	Y= \inr{X, t^*} + \zeta \enspace,
	\end{equation}where $t^*\in\R^d$,  $X\sim \cN(0, I_{d\times d})$ and $\zeta\sim \cN(0, \sigma^2)$ is independent of $X$.
	\item $\cD_2$ is a dataset of  $N_{bad-1}$ ``bad data'' $(X_i, Y_i)$ such that $Y_i=1$ and $X_i = (1)_{j=1}^d$ 
	\item $\cD_3$ is a dataset of  $N_{bad-2}$ ``bad data'' $(X_i, Y_i)$ such that $Y_i=10000$ and $X_i = (1)_{j=1}^d$
	\item $\cD_4$ is a dataset of $N_{bad-3}$ ``bad data'' $(X_i, Y_i)$ where $Y_i$ is a $0-1$-Bernoulli random variable and $X_i$ is uniformly distributed over $[0,1]^d$,
	\item $\cD_5$ is also a set of ``outliers'' that have been generated according to a linear model as in  \eqref{eq:gauss_model_simus} (i.e. with the same target vector $t^*$) but for a different choice of design $X$ and noise $\zeta$. Here, we take $X\sim \cN(0, \Sigma)$ where $\Sigma = (\rho^{|i-j|})_{1\leqslant i, j \leqslant d}$ and $\zeta$ is a heavy-tailed noise distributed according to a Student distribution with various degrees of freedom.
\end{enumerate}These different types of ``outliers'' $\cD_j, j=2, 3, 4, 5$ have been chosen to illustrate that the theory allows for outliers that may have absolutely nothing to do with the oracle $t^*$ that can be neither independent nor random as illustrated by datasets $\cD_2$ and $\cD_3$.

\subsection{From algorithms for the LASSO to their ``MOM versions''} % (fold)
\label{sub:algorithms_and_their}
There is a huge literature presenting many algorithms to implement the LASSO procedure. We explore several of them to investigate their performance regarding robustness. 

Each algorithm designed for the LASSO can be transformed into an algorithm for the min-max estimator we have been studying in this work. Let us now explain how this can be done. 
Recall that MOM version of the LASSO estimator is 
\begin{equation}\label{eq:esti}
\hat t_{K, \lambda} \in \argmin_{t\in \R^d} \sup_{t^\prime\in \R^d} T_{K, \lambda}(t^\prime, t)
\end{equation}where $
T_{K, \lambda}(t^\prime, t) = \MOM{K}{\ell_t-\ell_{t^\prime}} + \lambda \left(\norm{t}_1 - \norm{t^\prime}_1\right)$, $\MOM{K}{\ell_t-\ell_{t^\prime}}$ is a median of the set of real numbers $\{P_{B_1}(\ell_t-\ell_{t^\prime}), \cdots, P_{B_K}(\ell_t-\ell_{t^\prime})\}$ and for all $k\in[K]$,
\begin{equation*}
 P_{B_k}(\ell_t-\ell_{t^\prime}) = \frac{1}{|B_k|}\sum_{i\in B_k} (Y_i-\inr{X_i, t})^2 - (Y_i-\inr{X_i, t^\prime}))^2.
 \end{equation*}

A natural idea to implement \eqref{eq:esti} is to consider algorithms based on a sequence of alternating descents (in $t$) and ascents (in $t^\prime$) steps with or without a ``proximal/projection'' step and for various choices of ``step sizes''.  A key issue here is that the ``marginal'' functions $t\to T_{K, \lambda}(t^\prime_0, t)$ and $t^\prime\to T_{K, \lambda}(t^\prime, t_0)$, for some given $(t_0,t_0^\prime)$, may not be convex. Nevertheless, one can still locally compute the steepest descent by assuming that the index in $[K]$ of the block achieving the median in $\MOM{K}{\ell_{t_0}-\ell_{t^\prime_0}}$ remains constant on a convex open set containing $(t_0, t^\prime_0)$.  

% if one consider the following sets: for all $k\in[K]$, 
% \begin{equation*}
% \cC_k = \left\{(t,t^\prime)\in\R^d\times \R^d: P_{B_k}(\ell_t-\ell_{t^\prime}) \in \MOM{K}{\ell_t-\ell_{t^\prime}} \right\}.
% \end{equation*}

\begin{Assumption}\label{assum:partition}
Almost surely (with respect to $(X_i,Y_i)_{i=1}^N$) for almost all $(t_0, t^\prime_0)\in\R^d\times \R^d$ (with respect to the Lebesgue measure on $\R^d\times \R^d$), there exists a convex open set $B$ containing $(t_0,t^\prime_0)$ and $k\in[K]$ such that for all $(t,t^\prime)\in B$, $P_{B_k}(\ell_{t}-\ell_{t^\prime})\in \MOM{K}{\ell_{t}-\ell_{t^\prime}}$.
\end{Assumption}

Under Assumption~\ref{assum:partition}, for $\lambda\otimes\lambda$-almost all couples $(t_0,t^\prime_0)\in\R^d\times \R^d$ (where $\lambda$ is the Lebesgue measure on $\R^d$), $t\to T_{K, \lambda}(t_0^\prime, t)$ is ``locally strictly convex'' and $t^\prime\to T_{K, \lambda}(t^\prime, t_0)$ is ``locally strictly concave''. Therefore, for the choice of index $k$ such that $P_{B_k}(\ell_{t_0}-\ell_{t_0^\prime})\in \MOM{K}{\ell_{t_0}-\ell_{t_0^\prime}}$, we have
 \begin{equation}\label{eq:differential_in_t}
  \nabla_t \MOM{K}{\ell_t-\ell_{t^\prime_0}}_{|t=t_0}  = -2 (X^{(k)})^\top(Y^{(k)}-X^{(k)}t_0)
  \end{equation} where $Y^{(k)}= (Y_i)_{i\in B_k}$ and $X^{(k)}$ is the $|B_k|\times d$ matrix with rows given by $X_i^\top$ for $i\in B_k$.  The integer $k\in[K]$ is the index of the median of $K$ real numbers $P_{B_1}(\ell_t-\ell_{t^\prime}), \cdots, P_{B_K}(\ell_t-\ell_{t^\prime})$, which is straightforward to compute. The gradient $-2 (X^{(k)})^\top(Y^{(k)}-X^{(k)}t_0)$ in \eqref{eq:differential_in_t} depends on $t_0^\prime$ only through the index $k$. 

\begin{Remark}[Block Gradient Descent and map-reduce]Algorithms developed for the minmax estimator can be interpreted as block gradient descent. The major difference with the classical Block Gradient descent (which takes sequentially all the blocks one after another), is that the index of the block is chosen here at each round according to a ``centrality measure'': the index $k\in[K]$ of the block $B_k$ chosen to make one more descent / ascent step satisfies $P_{B_k}(\ell_{t_0}-\ell_{t_0^\prime})\in \MOM{K}{\ell_{t_0}-\ell_{t_0^\prime}}$. In particular, we expect blocks corrupted by outliers not to be chosen whereas in the classical BGD all blocks are chosen at each pass. Moreover, by choosing the ``descent / ascent'' block $k$ using this centrality measure we also expect $P_{B_k}(\ell_{t_0} - \ell_{t_0^\prime})$ to be close to the true expectation $P(\ell_{t_0} - \ell_{t_0^\prime})$ which is ultimately the function we would like to know. This makes every descent  (resp. ascent) steps particularly efficient since the right descent (resp. ascent) direction is $ - \nabla_t P(\ell_t-\ell_{t_0^\prime})_{|t=t_0}$ (resp. $\nabla_{t^\prime} P(\ell_{t_0}-\ell_{t^\prime})_{|t^\prime=t_0^\prime}$).

Moreover, our algorithms particularly fits the ``big data'' framework which is our original motivation for the introduction of robust procedures in machine learning. In this framework, the map-reduce paradigm has emerged as a leading idea to design procedures. In this situation, the data are spread out in  a cluster of servers and are therefore naturally split into blocks. Then our procedures simply uses for \textit{mapper} a mean function and for \textit{reducer} a  median function. This makes our algorithms easily scalable into the big data framework even when some servers have crashed down (making outliers data). MOM algorithm could therefore reduced the maintenance of big clusters.     
\end{Remark}

\begin{Remark}[Normalization]In the classical i.i.d. setup, the design matrix $\bX$ (i.e.,  the $N\times d$ matrix with row vectors $X_1, \ldots, X_N$) is usually normalized so that the $\ell_2^N$-norms of the columns equal to one. In our corrupted setup, we cannot normalize the design matrix this way because one row of $\bX$ may be corrupted. In that case, normalizing each column of $\bX$ would corrupt the entire matrix $\bX$. $\bX$ has to be kept as it is in all simulations.
\end{Remark}

In the sequel, we use this strategy to transform several algorithms implemented for the LASSO into algorithms for the min-max estimator \eqref{eq:esti}. We first start with the subgradient descent algorithm.

\subsection{Subgradient descent algorithm} % (fold)
\label{sub:subgradient_descent_algorithm}
The LASSO is solution of the minimization problem $\min_{t\in \R^d} F(t)$ where $F$ is defined for all $t\in\R^d$ by $F(t) = \norm{\bY-\bX t}_{2}^2 + \lambda \norm{t}_1$ with $\bY=(Y_i)_{i=1}^N$ and $\bX$ is the $N\times d$ matrix with row vectors $X_1, \ldots, X_N$. The LASSO can be approximated using a subgradient descent procedure : given a starting point $t_0\in\R^d$ and $(\gamma_p)_p$ a sequence of step sizes (i.e. $\gamma_p>0$ and $(\gamma_p)_p$ decreases), at step $p$ we update
\begin{equation}\label{eq:sub_grad_descent}
 t_{p+1} = t_p -\gamma_p \partial F(t_p)
 \end{equation} where $\partial F(t_p)$ is a subgradient of $F$ at $t_p$, for instance, $\partial F(t_p) = -2\bX^\top(\bY-\bX t_p) + \lambda {\rm sign}(t_p)$ where ${\rm sign}(t_p)$ is the vector of signs of the coordinates of $t_p$ with the convention ${\rm sign}(0)=0$. The sub-gradient descent algorithm \eqref{eq:sub_grad_descent} can be turned into an alternating subgradient ascent/descent algorithm for the min-max estimator
\begin{equation*}
\hat t_{K, \lambda} \in \argmin_{t\in\R^d} \sup_{t^\prime\in\R^d}\left(\MOM{K}{\ell_t-\ell_{t^\prime}} + \lambda \left(\norm{t}_1 - \norm{t^\prime}_1\right)\right)\enspace.
\end{equation*} Let 
\begin{equation}\label{eq:notation_data_blocks}
\bY_k = (Y_i)_{i\in B_k} \mbox{ and } \bX_k=(X_i^\top)_{i\in B_k}\in\R^{|B_k|\times d}\enspace.
\end{equation}

\begin{algorithm}[H]\label{algo:mom_sub_descent}
\SetKwInOut{Input}{input}\SetKwInOut{Output}{output}\SetKw{Or}{or}
\SetKw{Return}{Return}
\Input{$(t_0, t_0^\prime)\in\R^d\times \R^d$ : initial point\\ 
$\eps>0$ : a stopping criteria\\ 
$(\eta_p)_p, (\beta_p)_p$: two step size sequences}
\Output{approximated solution to the min-max problem \eqref{eq:esti}}  
\BlankLine
 \While{$\norm{t_{p+1}-t_p}_2\geqslant \eps$ \Or $\norm{t^\prime_{p+1}-t^\prime_p}_2\geqslant \eps$}{
  find $k\in[K]$ such that $\MOM{K}{\ell_{t_p^\prime} - \ell_{t_p}} = P_{B_k}(\ell_{t_p} - \ell_{t_p^\prime})$\\
   $$t_{p+1} = t_{p} + 2\eta_p\bX_k^\top(\bY_k-\bX_k t_p) - \lambda \eta_p {\rm sign}(t_{p})$$\\
   find $k\in[K]$ such that $\MOM{K}{\ell_{t_p^\prime} - \ell_{t_{p+1}}} = P_{B_k}(\ell_{t_{p+1}} - \ell_{t_p^\prime})$\\
    $$t_{p+1}^\prime = t_{p}^\prime  + 2\beta_p\bX_k^\top(\bY_k-\bX_kt_p^\prime) - \lambda \beta_p {\rm sign}(t_p^\prime)$$}
 \Return $(t_p, t_p^\prime)$
 \caption{An alternating sub-gradient descent algorithm for the minimax MOM estimator \eqref{eq:esti}.}
\end{algorithm} 

\vspace{0.8cm} 
The key insight in Algorithm~\ref{algo:mom_sub_descent} are step~2 and step~4 where the blocks number have been chosen according to their centrality  among the other blocks via the median operator. Those steps are expected 1) to remove outliers from the descent / ascent directions 2) to improve the accuracy of the latter directions. 

A classical choice of step size $\gamma_p$ in \eqref{eq:sub_grad_descent} is $\gamma_p = 1/L$ where $L = \norm{\bX}_{S_\infty}^2$ ($\norm{\bX}_{S_\infty}$ is the operator norm of $\bX$). Another possible choice follows from the Armijo-Goldstein condition with the following backtracking line search: $\gamma$ is decreased geometrically while the Armijo-Goldstein condition is not satisfied
\begin{equation}\label{eq:backtracking_line_search}
 \mathbf{ while }\quad  F(t_p + \gamma_\ell \partial F(t_p)) > F(t_p) + \delta\gamma_\ell \norm{\partial F(t_p)}_2^2 \quad \mathbf{ do } \quad \gamma_{\ell+1} = \rho \gamma_\ell
 \end{equation} for some given $\rho \in(0, 1)$, $\delta=10^{-4}$ and initial point $\gamma_0=1$. 

 Of course, the same choice of step size can be made as well for $(\eta_p)_p$ and $(\beta_p)_p$ in Algorithm~\ref{algo:mom_sub_descent}. In the first case, one can take $\eta_p =  1 / \norm{\bX_k}_{S_\infty}^2$ where $k\in[K]$ is the index defined in line \texttt{2} of Algorithm~\ref{algo:mom_sub_descent} and $\beta_p  = 1 / \norm{\bX_k}_{S_\infty}^2$ where $k\in[K]$ is the index defined in line \texttt{4} of Algorithm~\ref{algo:mom_sub_descent}. In the other backtracking line search case, the Armijo-Goldstein condition adapted for Algorithm~\ref{algo:mom_sub_descent} reads like
\begin{equation}\label{eq:backtracking_line_search}
 \mathbf{ while }\quad  F_k(t_p + \gamma_\ell \partial F_k(t_p)) > F_k(t_p) + \delta\gamma_\ell \norm{\partial F_k(t_p)}_2^2 \quad \mathbf{ do } \quad \eta_{\ell+1} = \rho \eta_\ell
 \end{equation}where $F_k(t) = \norm{\bY_k-\bX_k t}_2^2 + \lambda \norm{t}_1$ where $k\in[K]$ is the index defined in line \texttt{2} of Algorithm~\ref{algo:mom_sub_descent} and a similar update follows for $\beta_p$ with an index  $k\in[K]$ as defined in line \texttt{4} of Algorithm~\ref{algo:mom_sub_descent}.

% subsection subgradient_descent_algorithm (end)

\subsection{Proximal gradient descent algorithms} % (fold)
\label{sub:proximal_gradient_descent_algorithms}
In this section, we provide a MOM version of the classical ISTA and FISTA algorithms. Recall that ISTA (and its accelerated version FISTA) are proximal gradient descent that fall in the general class of splitting algorithms where one usually decomposes the objective function $F(t) = f(t) + g(t)$ with $f(t)= \norm{\bY-\bX t}_2^2$ (convex and differentiable) and $g(t) = \lambda \norm{t}_1$ (convex). 

ISTA stands for Iterative Shrinkage-Thresholding Algorithm. It  is a splitting algorithm that alternates between a descent step in the direction of the gradient and a ``projection step'' through the proximal operator of $g$ which is the soft-thresholding operator in the case of the $\ell_1$-norm : at step $p$
\begin{equation}\label{eq:ista}
   t_{p+1} = {\rm prox}_{\lambda \norm{\cdot}_1} \left(t_p + 2 \gamma_p \bX^\top (\bY - \bX t_p)\right) 
 \end{equation} where ${\rm prox}_{\lambda \norm{\cdot}_1}(t) = ({\rm sign}(t_j)\max(|t_j|-\lambda, 0))_{j=1}^d$ for all $t=(t_j)_{j=1}^d\in\R^d$.
A natural candidate for \eqref{eq:esti} is given by the following alternating method.

  \vspace{0.6cm}

\begin{algorithm}[H]\label{algo:prox_desc_grad}
\SetKwInOut{Input}{input}\SetKwInOut{Output}{output}\SetKw{Or}{or}
\SetKw{Return}{Return}
\Input{$(t_0, t_0^\prime)\in\R^d\times \R^d$ : initial point\\ 
$\eps>0$ : a stopping criteria\\ 
$(\eta_k)_k, (\beta_k)_k$: two step size sequences}
\Output{approximated solution to the min-max problem \eqref{eq:esti}}  
\BlankLine
 \While{$\norm{t_{p+1}-t_p}_2\geqslant \eps$ \Or $\norm{t^\prime_{p+1}-t^\prime_p}_2\geqslant \eps$}{
  find $k\in[K]$ such that $\MOM{K}{\ell_{t_p^\prime} - \ell_{t_p}} = P_{B_k}(\ell_{t_p} - \ell_{t_p^\prime})$
   $$t_{p+1} = {\rm prox}_{\lambda \norm{\cdot}_1}\left(t_{p}  + 2\eta_k\bX_k^\top(\bY_k-\bX_k t_p) \right)$$\\
   find $k\in[K]$ such that $\MOM{K}{\ell_{t_p^\prime} - \ell_{t_{p+1}}} = P_{B_k}(\ell_{t_{p+1}} - \ell_{t_p^\prime})$
    $$t_{p+1}^\prime = {\rm prox}_{\lambda \norm{\cdot}_1}\left(t_{p}^\prime  + 2\beta_k\bX_k^\top(\bY_k-\bX_k t_p^\prime) \right)$$}
 \Return $(t_p, t_p^\prime)$
 \caption{An alternating proximal gradient descent for the minimaximization procedure  \eqref{eq:esti}.}
\end{algorithm} 

\vspace{0.8cm} 
Note that the step sizes sequences $(\eta_p)_p$ and $(\beta_p)_p$ may be chosen according to the remarks below Algorithm~\ref{algo:mom_sub_descent}.

\subsection{Douglas-Racheford / ADMM} % (fold)
\label{sub:douglas_racheford_admm}
In this  section, we consider the ADMM algorithm. ADMM stands for Alternating Direction Method of Multipliers. It is also a splitting algorithm which reads as follows in the LASSO case: at step $p$,
\begin{align}\label{eq:ADMM}
\nonumber t_{p+1} &= (\bX^\top \bX + \rho I_{d\times d})^{-1}(\bX^\top \bY +\rho z_p - u_p)\\
\nonumber z_{p+1} &= {\rm prox}_{\lambda \norm{\cdot}_1} \left( t_{p+1} + u_p/\rho\right)\\
u_{p+1} &= u_p + \rho(t_{p+1}-z_{p+1})
\end{align}where $\rho$ is some parameter to be chosen (for instance $\rho=10$). The ADMM algorithm returns $t_p$ after a stopping criteria is met. In Algorithm~\ref{algo:mom_admm}, we provide a MOM version of this algorithm.

\vspace{0.6cm}
\begin{algorithm}[H]\label{algo:mom_admm}
\SetKwInOut{Input}{input}\SetKwInOut{Output}{output}\SetKw{Or}{or}
\SetKw{Return}{Return}
\Input{$(t_0, t_0^\prime)\in\R^d\times \R^d$ : initial point\\ 
$\eps>0$ : a stopping criteria\\ 
$\rho$: a parameter}
\Output{approximated solution to the min-max problem \eqref{eq:esti}}  
\BlankLine
 \While{$\norm{t_{p+1}-t_p}_2\geqslant \eps$ \Or $\norm{t^\prime_{p+1}-t^\prime_p}_2\geqslant \eps$}{
  find $k\in[K]$ such that $\MOM{K}{\ell_{t_p^\prime} - \ell_{t_p}} = P_{B_k}(\ell_{t_p} - \ell_{t_p^\prime})$
  \begin{align*}
\nonumber t_{p+1} &= (\bX_k^\top \bX_k + \rho I_{d\times d})^{-1}(\bX_k^\top \bY_k +\rho z_p - u_p)\\
\nonumber z_{p+1} &= {\rm prox}_{\lambda \norm{\cdot}_1} \left( t_{p+1} + u_p/\rho\right)\\
u_{p+1} &= u_p + \rho(t_{p+1}-z_{p+1})
\end{align*}
   \\
   find $k\in[K]$ such that $\MOM{K}{\ell_{t_p^\prime} - \ell_{t_{p+1}}} = P_{B_k}(\ell_{t_{p+1}} - \ell_{t_p^\prime})$
    \begin{align*}
\nonumber t_{p+1}^\prime &= (\bX_k^\top \bX_k + \rho I_{d\times d})^{-1}(\bX_k^\top \bY_k +\rho z_p^\prime - u_p^\prime)\\
\nonumber z_{p+1}^\prime &= {\rm prox}_{\lambda \norm{\cdot}_1} \left( t_{p+1}^\prime + u_p^\prime/\rho\right)\\
u_{p+1}^\prime &= u_p^\prime + \rho(t_{p+1}^\prime-z_{p+1}^\prime)
\end{align*}
    }
 \Return $(t_p, t_p^\prime)$
 \caption{An ADMM algorithm for the minimaximization MOM estimator \eqref{eq:esti}.}
\end{algorithm} 

\vspace{0.8cm}

\subsection{Cyclic coordinate descent} % (fold)
\label{sub:cyclic_coordinate_descent}
Repeatedly minimizing the LASSO objective function $t\to \norm{\bY-\bX t}_2^2+\lambda \norm{t}_1$ w.r.t. to each coordinate $t_j$ has proved to be an efficient algorithm. The closed form solution (cf. for instance  \cite[p. 83]{MR3307991}) is given by 
\begin{equation*}
t_j = \frac{R_j}{\norm{\bX_{\cdot j}}_2^2} \left(1-\frac{\lambda}{ 2 |R_j|}\right)_{+} \mbox{ with } R_j = \bX_{\cdot j}^\top \left(\bY - \sum_{k\neq j} t_k \bX_{\cdot k}\right).
\end{equation*}We can adapt this algorithm to approximate the MOM estimator \eqref{eq:esti}. Denote by $\bX_{kj}$ the $j$-th column of $\bX_k$ for all $j\in[d]$.

\vspace{0.6cm}
\begin{algorithm}[H]\label{algo:mom_admm}
\SetKwInOut{Input}{input}\SetKwInOut{Output}{output}\SetKw{Or}{or}\SetKw{In}{in}
\SetKw{Return}{Return}
\Input{$(t_0, t_0^\prime)\in\R^d\times \R^d$ : initial point\\ 
$\eps>0$ : a stopping criteria}
\Output{approximated solution to the min-max problem \eqref{eq:esti}}  
\BlankLine
 \While{$\norm{t_{p+1}-t_p}_2\geqslant \eps$ \Or $\norm{t^\prime_{p+1}-t^\prime_p}_2\geqslant \eps$}{
  find $k\in[K]$ such that $\MOM{K}{\ell_{t_p^\prime} - \ell_{t_p}} = P_{B_k}(\ell_{t_p} - \ell_{t_p^\prime})$\\
\For{j \In $[d]$}{
	$t_{p+1, j} = \frac{R_{kj}}{\norm{\bX_{kj}}_2^2} \left(1-\frac{\lambda}{ 2 |R_{kj}|}\right)_{+} \mbox{ with } R_{kj} = \bX_{kj}^\top \left(\bY_k - \sum_{q< j} t_{p+1,q} \bX_{kq}-\sum_{q> j} t_{p,q} \bX_{kq}\right)$}
   find $k\in[K]$ such that $\MOM{K}{\ell_{t_p^\prime} - \ell_{t_{p+1}}} = P_{B_k}(\ell_{t_{p+1}} - \ell_{t_p^\prime})$\\
\For{j \In $[d]$}{
	$t^\prime_{p+1, j} = \frac{R_{kj}}{\norm{\bX_{kj}}_2^2} \left(1-\frac{\lambda}{ 2 |R_{kj}|}\right)_{+} \mbox{ with } R_{kj} = \bX_{kj}^\top \left(\bY_k - \sum_{q< j} t^\prime_{p+1,q} \bX_{kq}-\sum_{q> j} t^\prime_{p,q} \bX_{kq}\right)$}
    }

 \Return $(t_p, t_p^\prime)$
 \caption{A Cyclic Coordinate Descent algorithm for the minimax MOM estimator \eqref{eq:esti}.}
\end{algorithm} 

\subsection{Adaptive choice of hyper-parameters via MOM V-fold Cross Validation} % (fold)
\label{sub:adaptive_choice_of_the_number_k_of_blocks_via_mom_cv}
The choice of hyperparameters in a ``data corrupted" environment has to be done carefully. Classical Cross-validation methods cannot be used because of the potential presence of outliers in the dataset that are likely to corrupt the classical CV-criterion. We design new (empirical) criteria that trustfully reveal performance of estimators even in situations where ``test datasets" might have been corrupted.

MOM's principles can be combined with the idea of multiple splitting into training / test datasets in cross-validation. Let us now explain the construction of estimators $\hat f_{\hat K, \hat \lambda}$ where the number of blocks $\hat K$ and the regularization parameter $\hat \lambda$ are hyper-parameters learned via such version of the CV principle.

We are given an integer $V\in[N]$ such that $N$ is divided by $V$. We are also given two finite grids $\cG_K\subset[N]$ and $\cG_\lambda\subset (0, 1]$. Our aim is to chose the ``best'' numbers of blocks and best regularization parameter within both grids. The dataset is splitted into $V$ disjoints blocks $\cD_1, \ldots, \cD_V$. For each $v\in[V]$, $\cup_{u\neq v}\cD_u$ is used to train a family of estimators
\begin{equation}\label{eq:family_estimators}
\left(\hat f_{K, \lambda}^{(v)}: K\in \cG_K, \lambda\in \cG_\lambda\right).
\end{equation}The remaining $\cD_v$ of the dataset is used to test the performance of each estimator in the family~\eqref{eq:family_estimators}. Using these notations, we can define a MOM version of the cross-validation procedure.

\begin{Definition}\label{def:MOM-CV}
 The \textbf{Median of Means $V$-fold Cross Validation procedure} associated to the family of estimators~\eqref{eq:family_estimators} is $\hat f_{\hat K, \hat \lambda}$ where $(\hat K, \hat \lambda)$ is minimizing the ${\rm MomCv}_V$ criteria 
 \begin{equation*}
 (K, \lambda)\in\cG_K\times \cG_\lambda \to {\rm MomCv}_V(K, \lambda) = Q_{1/2}\left({\rm MOM}_{K^\prime}^{(v)}\left(\ell_{\hat f_{K, \lambda}^{(v)}}\right)_{v\in[V]}\right), 
  \end{equation*} where, for all $v\in[V]$ and $f\in F$, 
  \begin{equation}\label{eq:MOM-crit}
   {\rm MOM}_{K^\prime}^{(v)}\left(\ell_f\right) = {\rm MOM}_{K^\prime} \left(P_{B_1^{(v)}}\ell_f, \cdots, P_{B_{K^\prime}^{(v)}}\ell_f\right)
  \end{equation}and $B_1^{(v)}\cup \cdots, \cup B_{K^\prime}^{(v)}$ is a partition of the test set $\cD_v$ into $K^\prime$ blocks where $K^\prime\in [N/V]$ such that $K^\prime$ divides $N/V$.
 \end{Definition} 

 The difference between standard V-fold cross validation procedure and its MOM version in Definition~\ref{def:MOM-CV} is that empirical means on test sets $\cD_v$ in the classical V-fold CV procedure have been replaced by MOM estimators in \eqref{eq:MOM-crit}. Moreover, the mean over all  $V$ splits in the classical $V$-fold CV is replaced by a median. 
% The idea is the same as previously: the test dataset $\cD_v$ is likely to be corrupted by outliers therefore the empirical risk constructed over $\cD_v$ cannot be trusted whereas its MOM version is not influenced by outliers if not too many.  

The choice of $V$ raises the same issues for MOM CV as for classical $V$-fold CV \cite{MR2602303, MR3595142}. In the simulations we use $V=5$. The construction of the MOM-CV requires to choose another parameter: $K^\prime$, the number of blocks used to construct the MOM criteria \eqref{eq:MOM-crit} over the test set. A possible solution is to take $K^\prime = K/V$. This has the advantage to make only one split of the dataset $\cD$ into $K$ blocks and then use for each of the $V$ rounds, $(V-1)K/V$ of these blocks to construct the family of estimators \eqref{eq:family_estimators} and then $K/V$ of these blocks to test them. 

In Figures~\ref{fig:MOM_CV_K} and~\ref{fig:MOM_CV_lambda}, hyper-parameters $K$ (i.e. the number of blocks) and $\lambda$ (i.e. the regularization parameter) have been chosen for the MOM LASSO estimator via the MOM V-fold Cross validation procedure introduced above. 
%More details  

\begin{figure}[!h]
    \centering
    \begin{minipage}{.45\textwidth}
    \includegraphics[width=1.1\textwidth]{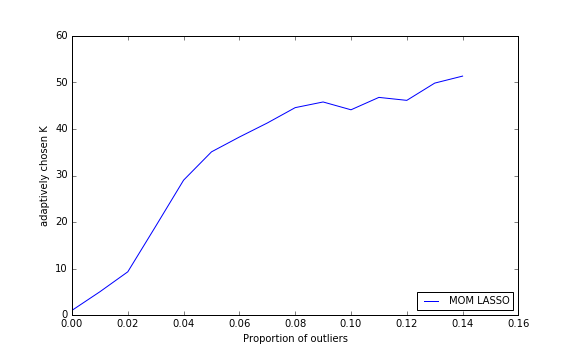}
    \caption{Adaptively chosen number of blocks $K$ for the MOM LASSO.}
    \label{fig:MOM_CV_K}
    \end{minipage}%
\hspace{0.2cm}
\begin{minipage}{.45\textwidth}
  \centering
    \includegraphics[width=1.1\textwidth]{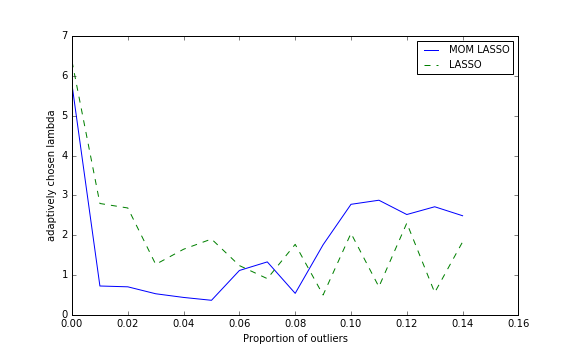}
    \caption{Adaptively chosen $\lambda$ for the LASSO and MOM LASSO}
    \label{fig:MOM_CV_lambda}
  \end{minipage}
\end{figure}
In Figure~\ref{fig:MOM_CV_K}, the adaptively chosen $\hat K$ grows with the number of outliers which is something expected since one needs the number of blocks to be at least twice the number of outliers.

\subsection{Maximinimization, saddle-point, random blocks, outliers detection and depth} % (fold)
\label{sub:maximinimization_saddle_point_random_blocks_and_outliers_detection}
In the previous sections, we considered a setup that particularly fits a dataset distributed on clusters. This distributed source of data yields some specific partition of the dataset in compliance with the clusters structure -- in other words, the blocks of data $B_1, \cdots, B_K$ are imposed and fixed by the specific physical form of the dataset. In a batch setup, there is a priori no structural restriction to choose a fixed partition of the dataset given that it has no predefined organization. It appears that, in this setup, there is a way to improve the performance of the various iterative algorithms constructed in the previous sections by choosing randomly the blocks at every (descent and ascent) steps of the algorithm. The aim of this section is to show some advantages of this choice and how this modified version works on the example of ADMM. Moreover, as a byproduct of this approach, it is possible to construct an outliers detection algorithm. This algorithm outputs a score to each data in the dataset measuring its centrality. In particular, data with a low score should be considered as outliers.

Note that the way we introduced the minimaximization estimator in Section~\ref{sec:intro} was based on the observation that the oracle is solution to the minmaximization problem 
$f^*\in\argmin_{f\in F}\sup_{g\in F}P(\ell_f-\ell_g)$. But it appears that $f^*$ is also solution to a maxminimization problem: $f^*\in\argmax_{g\in F}\inf_{f\in F}P(\ell_f-\ell_g)$. This observation naturally yields another estimator: the maxmin estimator 
\begin{equation}\label{eq:maxmin_esti}
\hat g_{K, \lambda} \in\argmax_{g\in F}\inf_{f\in F} T_{K, \lambda}(g, f).
\end{equation}Following the same strategy as in Section~\ref{sec:proofs}, we can show that $\hat g_{K, \lambda} $ has the very same statistical performance as the estimator  $\hat f_{K, \lambda}$ from Section~\ref{sec:main_results} (cf. Section~\ref{sec:learning_without_regularization_and_minimax_optimality} in the case where there is no regularization for a proof). Nevertheless, there is a priori no reason that $\hat g_{K, \lambda}$ and $\hat f_{K, \lambda}$ are the same estimator. That is we don't know in advance that 
\begin{equation}\label{eq:duality_gap}
\argmin_{f\in F}\sup_{g\in F} T_{K, \lambda}(g, f) = \argmax_{g\in F}\inf_{f\in F} T_{K, \lambda}(g, f). 
\end{equation}In other words there is no evidence that the duality gap is null. Since $T_{K, \lambda}(g,f) = - T_{K, \lambda}(f, g)$, \eqref{eq:duality_gap} holds if and only if
\begin{equation*}
 \inf_{f\in F}\sup_{g\in F} T_{K, \lambda}(f, g) = 0.
 \end{equation*} In that case, $\hat f$ is by definition a \textbf{saddle-point} estimator and therefore the two sets of minmax and maxmin estimators are equal. 

 \begin{Remark}[stopping criteria] Along iterations of the previous algorithms it is possible to track down the values of the objective function $\MOM{K}{\ell_{t_p}-\ell_{t_p^\prime}} + \lambda \left(\norm{t_p}_1 - \norm{t_p^\prime}_1\right)$. When this one is close to zero this means that it has achieved a saddle-point and therefore the output of the algorithm is close to the solution of the minmax problem (as well as the maxmin problem). 
 \end{Remark}

 When  $\hat f$ is a saddle-point, the choice of fixed blocks $B_1, \ldots, B_K$ may result in a problem of ``\textbf{local saddle points}": iterations of our original algorithms remain close to a potentially suboptimal local saddle point. 

 To see this, let us consider the vector case (that is for $f(\cdot)=\inr{\cdot, t}$ for $t\in\R^d$). We introduce the following sets: for all $k\in[K]$,
 \begin{equation}\label{eq:partition}
 \cC_k = \left\{(t, t^\prime)\in \R^d\times \R^d: \MOM{K}{\ell_t-\ell_{t^\prime}}= P_{B_k}(\ell_t-\ell_{t^\prime})  \right\}.
 \end{equation} It is clear that $\cup_{k\in[K]}\cC_k = \R^d\times \R^d$. The key idea behind all the previous algorithms is that at each step the current iteration $(t_p, t^\prime_p)$ lies in one of those cells $\cC_k$ and that, if Assumption~\ref{assum:partition} holds,  for almost all iterations there will be an open ball around $(t_p, t^\prime_p)$ contained in the cell $\cC_k$ so  that the objective function is locally equal to $(t, t^\prime)\to P_{B_k}(\ell_t-\ell_{t^\prime}) + \lambda(\norm{t}_1 - \norm{t^\prime}_1)$ which is a convex-concave function. 

 The problem here is that our algorithms are looking for saddle-points so that if there are several cells $\cC_k$ containing saddle-points the previous algorithms may be stuck in one of them whereas a ``better saddle-point'' (that is a saddle point closer to $t^*$) may be in an other cell.

 To overcome this issue, we choose at every (descent and ascent) steps of the previous algorithms a random partition of the dataset into $K$ blocks. The decomposition of the space $\R^d\times \R^d$ into cells $\cC_1, \cdots, \cC_K$ does not exist anymore since the cells are different (randomly chosen) at every steps. As an example, we develop the ADMM procedure with a random choice of blocks In Algorithm~\ref{algo:mom_admm_RB}.

 \vspace{0.6cm}
\begin{algorithm}[H]\label{algo:mom_admm_RB}
\SetKwInOut{Input}{input}\SetKwInOut{Output}{output}\SetKw{Or}{or}
\SetKw{Return}{Return}
\Input{$(t_0, t_0^\prime)\in\R^d\times \R^d$: initial point\\ 
$\eps>0$: a stopping criteria\\ 
$\rho$: parameter}
\Output{approximated solution to the min-max problem \eqref{eq:esti}}  
\BlankLine
 \While{$\norm{t_{p+1}-t_p}_2\geqslant \eps$ \Or $\norm{t^\prime_{p+1}-t^\prime_p}_2\geqslant \eps$}{
  Partition the datasets into $K$ blocks $B_1, \ldots, B_K$ of equal size at random.

  Find $k\in[K]$ such that $\MOM{K}{\ell_{t_p^\prime} - \ell_{t_p}} = P_{B_k}(\ell_{t_p} - \ell_{t_p^\prime})$
  \begin{align*}
\nonumber t_{p+1} &= (\bX_k^\top \bX_k + \rho I_{d\times d})^{-1}(\bX_k^\top \bY_k +\rho z_p - u_p)\\
\nonumber z_{p+1} &= {\rm prox}_{\lambda \norm{\cdot}_1} \left( t_{p+1} + u_p/\rho\right)\\
u_{p+1} &= u_p + \rho(t_{p+1}-z_{p+1})
\end{align*}
   \\
   Partition the datasets into $K$ blocks $B_1, \ldots, B_K$ of equal size at random.

   Find $k\in[K]$ such that $\MOM{K}{\ell_{t_p^\prime} - \ell_{t_{p+1}}} = P_{B_k}(\ell_{t_{p+1}} - \ell_{t_p^\prime})$
    \begin{align*}
\nonumber t_{p+1}^\prime &= (\bX_k^\top \bX_k + \rho I_{d\times d})^{-1}(\bX_k^\top \bY_k +\rho z_p^\prime - u_p^\prime)\\
\nonumber z_{p+1}^\prime &= {\rm prox}_{\lambda \norm{\cdot}_1} \left( t_{p+1}^\prime + u_p^\prime/\rho\right)\\
u_{p+1}^\prime &= u_p^\prime + \rho(t_{p+1}^\prime-z_{p+1}^\prime)
\end{align*}
    }
 \Return $(t_p, t_p^\prime)$
 \caption{The ADMM algorithm for the minimax MOM estimator \eqref{eq:esti} with a random choice of blocks at each steps.}
\end{algorithm}

\begin{figure}[!h]
    \centering
    \includegraphics[width=1\textwidth]{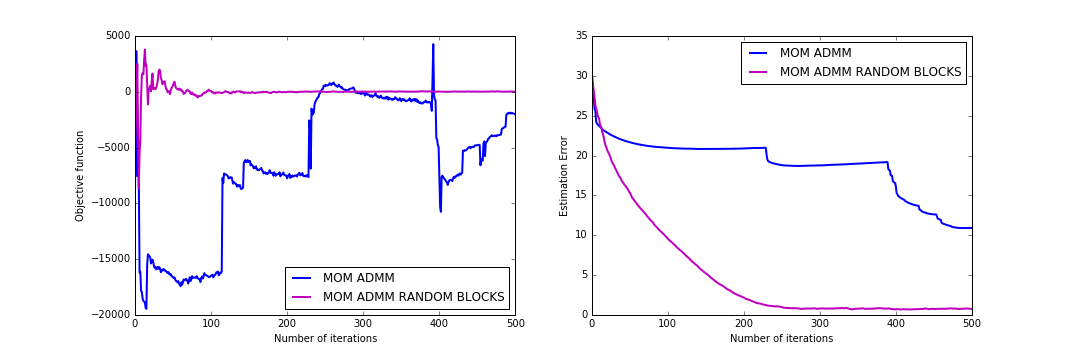}
    \caption{Fixed blocks against random blocks.}
    \label{fig:fixed_random_blocks}
\end{figure}

In Figure~\ref{fig:fixed_random_blocks}, we ran both MOM LASSO procedures via ADMM with fixed and random blocks. Both the objective function and the estimation error of MOM LASSO jump in the case of fixed blocks. These jumps correspond to a change of cell number for these iterations. The algorithm converge to local saddle-points before jumping to other cells. This slows down its convergence. On the other hand, the algorithms with random blocks do not suffer this drawback. Figure~\ref{fig:fixed_random_blocks} shows that the estimation error converges much faster and smoothly for random blocks than for fixed blocks. As a conclusion choosing blocks at random improve both stability and speed of the algorithm, avoiding the issue of ``local saddle points'' in the cells $\cC_k$. Note also that the objective function of the MOM ADMM with random blocks tends to zero so that the duality gap tends to zero (meaning that we do have a natural stopping criterium and that the MOM LASSO is a saddle point).

\vspace{0.8cm} 

A byproduct of this approach is that one can construct an \textbf{outliers detection procedure}. To that end one simply has to count the number of times each data is selected at steps $2$ and $4$ of Algorithm~\ref{algo:mom_admm_RB}. Every data starts with a null score. Then, every time a block $B$ is selected as median block, every data it contains increases its score by one. At the end, every data ends up with a score revealing its centrality for our learning task. 

It is expected that aggressive outliers corrupt their respective blocks. These ``corrupted blocks'' will not be median blocks and will not be chosen. In the case of fixed blocks, informative data cannot be distinguished from outliers lying in the same block, therefore, this outliers detection algorithm only makes sense when blocks are chosen at random. Figure~\ref{fig:outliers_detection} shows performance of this strategy on synthetic data (cf. Section~\ref{sub:simulations_setup_for_the_figures} for more details on the simulations). On this example, outliers (data number $1, 32, 170$ and $194$) end up with a null score meaning that they have never been selected along the descent / ascent iterations of the algorithm. Rearranging the scores, one can observe a jump in the score function between outliers and informative data. The gap between these scores can be enlarged by running more iterations of the algorithm. Finally, the score may be seen as a \textbf{depth} of a data point $(X_i, Y_i)$, the most selected data are the most central. 

%Given a new point $(X, Y)$ one may compute its depth with respect to the dataset by adding this point to the dataset and compute the number of times $(X, Y)$ has been selected. If it is high then the point is in the inner center of the dataset and conversely.

\begin{figure}[!h]
    \centering
    \includegraphics[width=1\textwidth]{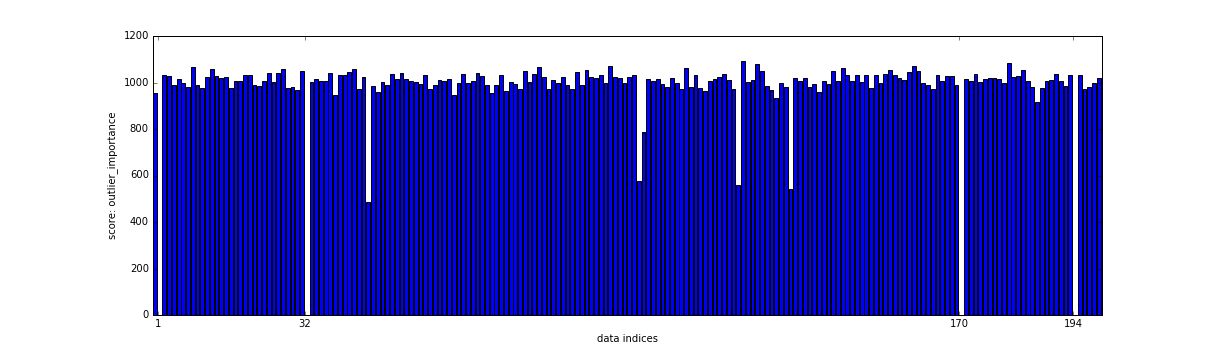}
    \caption{Outliers detection algorithm. The dataset has been corrupted by $4$ outliers at number $1, 32, 170$ and $194$. The score of the outliers is $0$: they haven't been selected even once.}
    \label{fig:outliers_detection}
\end{figure}   

% subsection maximinimization_saddle_point_random_blocks_and_outliers_detection (end)

\subsection{Simulations setup for the figures} % (fold)
\label{sub:simulations_setup_for_the_figures}
All codes used to obtain the figures in this work are available at \cite{notebook_mom_lasso} and can therefore be used to reproduce all the figures. Let us now precise the parameters of the simulations in Figure~\ref{fig:robustness}, Figure~\ref{fig:MOM_CV_K} and Figure~\ref{fig:MOM_CV_lambda}: the number of observations is $N=200$, the number of features is $d=500$, we construct a sparse vector $t^*\in \R^d$ with sparsity $s=10$ and support chosen at random and non-zero coordinates $t^*_j$ being either equal to $10$, $-10$ or decreasing according to $\exp(-j/10)$. We end up with a set of informative data $\cD_1$ as described in Section~\ref{sec:simulation_study} with noise variance $\sigma=1$. Then this dataset is increasingly corrupted with outliers as in $\cD_4$ (cf. Section~\ref{sec:simulation_study}): outliers are data with output $Y = 10000$ and input $X = (1)_{j=1}^d$. 

The proportion of outliers are $0, 1/100, 2/100, \ldots, 15/100$. On each dataset $\cD_1\cup\cD_4$ obtained after corruption of $\cD_1$ for various proportion of outliers, we run the ADMM algorithm with adaptively chosen regularization parameter $\lambda$ via $V$-fold CV with $V=5$ for the LASSO. Then we run the MOM ADMM with adaptively chosen number of blocks $K$ and regularization parameter $\lambda$ via the MOM CV procedure as introduced in Section~\ref{sub:adaptive_choice_of_the_number_k_of_blocks_via_mom_cv} with $V=5$ and $K^\prime = \max({\rm grid}_K )/V$ where  ${\rm grid}_K = \left\{1, 4, \cdots, 115/4\right\} $ and ${\rm grid}_{\lambda} = \left\{0, 10, 20 \cdots, 100\right\}/\sqrt{100}$ are the search grids used to select the best $K$ and $\lambda$ during the CV and MOM CV steps. The number of iterations of ADMM and MOM ADMM was taken equal to $200$. Those simulations have been run $70$ times and the averaged values of the estimation error, adaptively chosen $K$ and $\lambda$ have been reported in Figure~\ref{fig:robustness}, Figure~\ref{fig:MOM_CV_K} and Figure~\ref{fig:MOM_CV_lambda}. The $\ell_2$ estimation error of the LASSO increases roughly from $0$ when there is no outlier and stabilize at $550$ right after a single outlier enter the dataset. The $\ell_2$ estimation error value $550$ can be explained by the fact that the outliers which was added is $Y=10000$ and $X=(1)_{j=1}^{500}$ which satisfies $Y=\inr{X, t^{**}}$ for $t^{**}=(20)_{j=1}^{500}$ which is the solution with minimal $\ell_1^d$ norm among all the solutions $t\in\R^d$ such that $Y=\inr{X, t}$. It appears that $\norm{t^{**} - t^*}_2$ is of the order of $550$. It means that the LASSO is actually solving the solution to the linear program associated with the outlier instead of solving the linear problem associated with the $200$ other informative data. A single outliers is therefore completely misleading the LASSO.

For Figure~\ref{fig:fixed_random_blocks}, we have ran similar experiments with $N=200$, $d=300$, $s=20$, $\sigma=1$, $K=10$, the number of iterations was $500$ and the regularization parameter was $1/\sqrt{N}$.

For Figure~\ref{fig:outliers_detection}, we took $N=200$, $d=500$, $s=20$, $\sigma=1$, the number of outliers is $|\cO| = 4$ and the outliers are of the form $Y=10000$ and $X=(1)_{j=1}^d$, $K=10$, the number of iterations is $5.000$ and $\lambda = 1/\sqrt{200}$.

% subsection simulations_setup_for_the_figures (end)

\section{Examples of applications} % (fold)
\label{sec:further_examples_of_applications}
This section is dedicated to some applications of Theorem~\ref{thm:LepskiReg}. We present two classical examples of regularization in high-dimensional statistics: the $\ell_1$-norm and the SLOPE norm. The associated RERM have been extensively studied in the literature under strong assumptions. The aim of this section is to show that these assumptions can be strongly relaxed when dealing with MOM versions. For instance, as observed in Figure~\ref{fig:robustness}, the performance of the LASSO are drastically deteriorated in the presence of a single outlier while its MOM version remains informative. 

\subsection{The LASSO} % (fold)
\label{sub:the_lasso}
The LASSO is obtained when $F=\{\inr{t,\cdot} : t \in \R^d\}$ and the regularization function is the  $\ell_1$-norm :
\[
\hat{t} \in\argmin_{t \in \R^d} \Big(\frac{1}{N}\sum_{i=1}^N \left(\inr{t,X_i}-Y_i\right)^2 + \lambda \|t\|_1 \Big)
,\quad \text{where}\quad \|t\|_1 = \sum_{i=1}^d |t_i|\enspace.\]

Even if recent advances \cite{MR3153940,MR3224285,MR3161450} have shown some limitations of LASSO, it remains the benchmark estimator in high-dimensional statistics because a high dimensional parameter space does not significantly affect its performance as long as $t^*$ is sparse. This was shown for example, in \cite{MR2533469,MR2386087,deter_lasso,MR2396809,MR2488351,MR3025133,MR3181133} for estimation and sparse oracle inequalities, in \cite{MR2278363,MR2274449,bach_support} for support recovery results; more results and references on LASSO can be found in the books \cite{MR2807761,MR2829871}.

%\vskip0.4cm

\subsection{SLOPE} % (fold)
\label{sub:slope}
SLOPE is an extension of the LASSO that was introduced in \cite{slope1,slope2}. The class $F$ is still $F=\{\inr{t,\cdot} : t \in \R^d\}$ and the regularization function is defined for parameters $\beta_1 \geqslant \beta_2 \geqslant ... \geqslant \beta_d>0$ by
$$
\norm{t}_{SLOPE}=\sum_{i=1}^d \beta_i t_i^\sharp,
$$
where $(t_i^\sharp)_{i=1}^d$ denotes the non-increasing re-arrangement of $(|t_i|)_{i=1}^d$. SLOPE norm is a weighted $\ell_1$-norm that coincide with $\ell_1$-norm when $(\beta_1,...,\beta_d)=(1,...,1)$.

% \subsection{Trace-norm regularization} % (fold)
% \label{sub:trace_norm_regularization}
% Consider the trace inner-product on $\R^{m \times T}$. Let $F=\{\inr{A,\cdot} : A \in \R^{m \times T}\}$ and given a target $Y$ put $A^*$ to be the matrix that minimizes $A \to\E (\inr{A,X}-Y)^2$. The regularization function is the {\it trace norm}.
% \begin{Definition}
% Let $A$ be a matrix and set $(\sigma_i(A))$ to be its singular values, arranged in a non-increasing order.
% For $p \geqslant 1$, $\|A\|_p = (\sum \sigma_i^p(A))^{1/p}$ is the $p$-Schatten norm. The  $1$-Schatten norm is also called the trace-norm, the $2$-Schatten is also called the Hilbert-Schmidt norm and the $\infty$-Schatten norm is the operator norm.
% \end{Definition}

% The trace norm regularization procedure is
% \begin{equation*}
%  \hat A \in\argmin_{A\in\R^{m\times T}}\Big(\frac{1}{N}\sum_{i=1}^N (Y_i-\inr{X_i,A})^2+\lambda \|A\|_1\Big)
%  \end{equation*}
% and it was introduced for the reconstruction of low-rank, high-dimensional matrices  \cite{MR2680543,MR2906869,MR2816342,MR2809094,MR2815834,MR2930649}.

\subsection{Classical results for LASSO and SLOPE} % (fold)
\label{sub:classical_results_for_lasso_slope_and_trace_norm}
Typical results for LASSO and SLOPE have been obtained in the i.i.d. setup under a subgaussian assumption on the design $X$ and, most of the time, on the noise $\zeta$ as well. Let us for the moment provide a definition of this assumption and recall the one of isotropicity.

\begin{Definition} \label{def:subgaussian}
Let $\ell_2^d$ be a $d$-dimensional inner product space and let $X$ be random variable with values in $\ell_2^d$. We say that $X$ is isotropic when for every  $t \in \ell_2^d$, $\|\inr{X,t}\|_{L^2} = \|t\|_{\ell_2^d}^2$ and it is $L$-subgaussian if for every $p \geqslant 2$ and every $t \in \ell_2^d$, $\|\inr{X,t}\|_{L^p} \leqslant L\sqrt{p} \|\inr{X,t}\|_{L^2}$.
\end{Definition}
In other words, the covariance structure of an isotropic random variable coincides with the inner product in $\ell_2^d$, and if $X$ is an $L$-subgaussian random vector then the $L^p$ norm of all linear forms does not grow faster than the $L^p$ norm of the corresponding Gaussian variable. When dealing with the LASSO and SLOPE, the natural Euclidean structure is used in $\R^d$.

In the following assumption, we recall a setup where both estimators have been studied in \cite{LM_reg1}.
\begin{Assumption}\label{ass:examples}
\begin{enumerate}
	\item the data are i.i.d. (in particular, $|\cI| = N$ and $|\cO|=0$, i.e. there is no outlier),
	\item $X$ is isotropic and $L$-subgaussian,
	\item for $f^*=\inr{t^*,\cdot}$, $\zeta=Y-f^*(X)\in L^{q_0} $ for some $q_0>2$.
\end{enumerate} 
\end{Assumption}

Unlike many results on these estimators, Assumption~\ref{ass:examples} only requires a ``minimal'' $L^{q_0}$ for $q_0>2$ moment on the noise. It appears that LASSO and  SLOPE still achieve optimal rates of convergence under this weak stochastic assumption but the price to pay is a severely deteriorated probability estimate. 
%Let us now state the main results from \cite{LM_reg1} on both estimators under Assumption~\ref{ass:examples}.  

\begin{Theorem}[Theorem~1.4 in \cite{LM_reg1}] \label{thm:intro-LASSO-est}
Consider the LASSO under Assumption \ref{ass:examples}. Let $s\in[d]$. Assume that $N\geqslant c_1 s \log(ed/s)$ and that there is some $v \in \R^d$ supported on at most $s$ coordinates for which $\norm{t^*-v}_1\leqslant c_2 \|\xi\|_{L^{q_0}} s \sqrt{\log(ed)/N}$. The Lasso estimator $\hat t$ with regularization parameter $\lambda= c_3\|\xi\|_{L^{q_0}} \sqrt{\log(ed)/N}$ is such that with probability at least
\begin{equation}\label{eq:proba_lasso}
1-\frac{c_4 \log^{q_0} N}{N^{q_0/2-1}} - 2 \exp\left(-c_5 s \log(ed/s)\right)
\end{equation} for every $1\leqslant p \leqslant 2$
\begin{equation*}
\norm{\hat t-t^*}_p \leqslant c_6\|\xi\|_{L_q} s^{1/p}\sqrt{\frac{\log(ed)}{N}}.
\end{equation*}The constants $(c_j)_{j=1}^6$ depend only on $L$ and $q_0$.
\end{Theorem}
The error rate in Theorem \ref{thm:intro-LASSO-est} coincides with the standard estimate on the LASSO (cf. \cite{MR2533469}), but in a broader context: $t^*$ does not need to be sparse but should be approximated by a sparse vector; the target $Y$ is arbitrary (there is no need for a statistical model) and the noise $\zeta$ may be heavy tailed and does not need to be independent from $X$. But there is no room for outliers, the design matrix $X$ still needs to be subgaussian and the data are assumed to be i.i.d.. We will see below that the MOM version of the LASSO can go further, achieving minimax optimal error bounds with a much better probability estimate.
% than in \eqref{eq:proba_lasso}.  

Turning to SLOPE, recall the following result for the regularization norm $\Psi(t)=\sum_{j=1}^d \beta_j t_j^\sharp$ when $\beta_j = C \sqrt{\log(ed/j)}$. 

\begin{Theorem}[Theorem~1.6 in \cite{LM_reg1}]\label{thm:intro-SLOPE-est}
Consider the SLOPE under Assumption \ref{ass:examples}. Assume that $N \geqslant  c_1 s \log(ed/s)$ and  that there is $v\in\R^d$ such that $|{\rm supp}(v)|\leqslant s$ and $\Psi(t^*-v)\leqslant c_2 \|\xi\|_{L_q} s\log(ed/s)/\sqrt{N}$. The SLOPE estimator with regularization parameter $\lambda= c_3 \|\xi\|_{L_q}/\sqrt{N}$, satisfies with the same probability as in \eqref{eq:proba_lasso} that
\begin{equation*}
\Psi(\hat t-t^*)\leqslant c_4\|\xi\|_{L_q} \frac{s}{\sqrt{N}} \log\Big(\frac{ed}{s}\Big) \ \ \mbox{ and } \ \ \norm{\hat t-t^*}_2^2 \leqslant c_5 \|\xi\|_{L_q}^2 \frac{s }{N}\log\Big(\frac{ed}{s}\Big)\enspace.
\end{equation*}The constants $(c_j)_{j=1}^5$ depend only on $L$ and $q_0$.
\end{Theorem}

% Finally, let us consider the trace norm regularization.
% \begin{Theorem}[Theorem~1.7 in \cite{LM_reg1}]\label{thm:intro-TRACE-est}
% Assume that Assumption~\ref{ass:examples} holds, that there is $V\in\R^{m\times T}$ such that ${\rm rank}(V)\leqslant s$ and $\norm{A^*-V}_1\leqslant c_1\|\xi\|_{L_q} s \sqrt{\max\{m,T\}/N}$ and that $N \geqslant c_2 s \max\{m,T\}$. The trace norm regularization procedure with regularization parameter $\lambda = c_3\|\xi\|_{L_q} \sqrt{\max\{m,T\}/N}$ satisfies with probability at least
% \begin{equation}\label{eq:proba_S1}
% 1-\frac{c_4 \log^{q_0} N}{N^{q_0/2-1}} - 2 \exp\left(-c_5 s\max\{m,T\} \right)
% \end{equation} for any $1\leqslant p\leqslant2$
% \begin{equation*}
% \norm{\hat A-A^*}_p \leqslant  c_6\|\xi\|_{L_q}  s^{1/p}\sqrt{\frac{\max\{m,T\}}{N}}.
% \end{equation*}
% The constants $(c_j)_{j=1}^6$ depends only on $L$ and $q_0$.
% \end{Theorem}

\subsection{Statistical analysis of MOM LASSO and MOM SLOPE} % (fold)
\label{sub:statistical_performance_of_the_mom_lasso_mom_slope_and_mom_trace_norm}
In this section, Theorem~\ref{thm:LepskiReg} is applied to the set $F$ of linear functionals indexed by $\R^d$ with regularization functions being either the $\ell_1$-norm or the SLOPE norm. The aim is to show that the results from Section~\ref{sub:classical_results_for_lasso_slope_and_trace_norm} are still satisfied (and sometimes even improved) by their MOM version under much weaker assumptions and with a much better probability deviation. 
%In particular, MOM estimators allow for a large part of outliers, informative data do not have to be distributed like $(X, Y)$ and the design $X$ can be heavy-tailed but still one can improve the polynomial probability deviation bound from \eqref{eq:proba_lasso} to a sub-exponential bound as if the design and the noise were subgaussian. 
Start with the new set of assumptions.

\begin{Assumption}\label{ass:examples_MOM}Denote by $(e_j)_{j=1}^d$ the canonical basis of $\R^d$. We assume that
\begin{enumerate}
	\item $|\cI| \geqslant N/2$ and $|\cO|\leqslant c_1 s \log(ed/s)$,
	\item $X$ is isotropic and for every $t\in\R^d$, $p\in [C_0\log(ed)]$ and $j\in [d]$, $\norm{\inr{X, e_j}}_{L^p}\leqslant L \sqrt{p}\norm{\inr{X, e_j}}_{L^2}$,
	\item for $f^*=\inr{t^*,\cdot}$,  $\zeta=Y-f^*(X)\in L^{q_0} $ for some $q_0>2$.
	\item there exists $\theta_0$ such that for all $t\in\R^d$, $\norm{\inr{X, t}}_{L^2}\leqslant \theta_0 \norm{\inr{X, t}}_{L^1}$,
	\item there exists $\theta_m$ such that ${\rm var}(\zeta\inr{X, t})\leqslant \theta_m \norm{\inr{X, t}}_{L^2}$.
\end{enumerate} 
\end{Assumption}

In order to apply Theorem~\ref{thm:LepskiReg} we need to compute the fixed point functions $r_Q(\cdot)$, $r_M(\cdot)$ and solve the sparsity equation in both cases. We start with the fixed point functions. To that end we recall the definition of Gaussian mean widths: for a set $V\subset\R^d$, the Gaussian mean width of $V$ is defined as
\begin{equation}\label{eq:Gauss_mean_width}
\ell^*(V) = \E\cro{ \sup_{(v_j)\in V} \sum_{j=1}^d g_j v_j},\quad \text{where }\quad (g_1,\ldots,g_d)\sim\mathcal N_d(0,I_d)\enspace.
\end{equation}
%
%Now, the functions $r_Q(\cdot)$ and $r_M(\cdot)$ in the $\ell_1^d$-regularization example are computed in \cite[Theorem~1.6]{shahar_gafa_ln}. 
The dual norm of the $\ell_1^d$-norm is the $\ell_\infty^d$-norm which is $1$-unconditional  with respect to the canonical basis of $\R^d$ \cite[Definition~1.4]{shahar_gafa_ln}. Therefore, \cite[Theorem~1.6]{shahar_gafa_ln} applies under the following assumption.

\begin{Assumption}\label{ass:shahar_theo16}
There exist constants $q_0>2$, $C_0$ and $L$ such that $\zeta\in L^{q_0}$, $X$ is isotropic and for every $j\in[d]$ and $1\leqslant p\leqslant C_0 \log d$, $\norm{\inr{X,e_j}}_{L^p}\leqslant L\sqrt{p}\norm{\inr{X,e_j}}_{L^2}$.
\end{Assumption}
%
%The unit ball $B(0,1)$ plays a central role in our analysis. In this example, it is isometric to $B_1^d$. 

\noindent
Under Assumption~\ref{ass:shahar_theo16}, if $\sigma=\norm{\zeta}_{L^{q_0}}$, \cite[Theorem~1.6]{shahar_gafa_ln} shows that, for every $\rho>0$, 
\begin{gather*}
\E \sup_{v\in \rho B_1^d \cap r B_2^d} \left|\sum_{i\in[N]}\eps_i \inr{v, X_i}\right|\leqslant c_2\sqrt{N}\ell^*(\rho B_1^d \cap r B_2^d)\enspace,\\
\E \sup_{v\in \rho B_1^d \cap r B_2^d} \left|\sum_{i\in[N]} \eps_i \zeta_i \inr{v, X_i}\right|\leqslant c_2\sigma\sqrt{N}\ell^*(\rho B_1^d \cap r B_2^d)\enspace.
\end{gather*}
Local Gaussian mean widths $\ell^*(\rho B_1^d \cap r B_2^d)$ are bounded from above in \cite[ Lemma~5.3]{LM_reg_comp} and computations of $r_M(\cdot)$ and $r_Q(\cdot)$ follow 
\begin{align*}
r_M^2(\rho) &\lesssim_{L,q_0, \gamma_M} 
\begin{cases}
\sigma^2\frac{ d}{N} & \mbox{ if } \rho^2 N \geqslant \sigma^2 d^2\\
 \rho\sigma\sqrt{\frac{1}{N}\log\Big(\frac{e\sigma d}{\rho\sqrt{N}}\Big)} & \mbox{ otherwise}
\end{cases}
\enspace,\\
\notag r_Q^2(\rho) &
\begin{cases}
 = 0 & \mbox{ if }  N \gtrsim_{L, \gamma_Q}  d \\
 \lesssim_{L, \gamma_Q} \frac{\rho^2}{N}\log\Big(\frac{c(L, \gamma_Q)d}{N}\Big) & \mbox{ otherwise}
\end{cases}
\enspace.
\end{align*}
Therefore, one can take 
\begin{equation}\label{eq:r_function_LASSO}
r^2(\rho) \sim_{L,q_0, \gamma_Q, \gamma_M} 
\begin{cases}
\max\left(\rho \sigma\sqrt{\frac{1}{N}\log\Big(\frac{e \sigma d}{\rho\sqrt{N}}\Big)}, \frac{\sigma^2 d}{N}\right) & \mbox{ if } N \gtrsim_L d\\
\max\left(\rho \sigma \sqrt{\frac{1}{N}\log\Big(\frac{e \sigma d}{\rho\sqrt{N}}\Big)},\frac{\rho^2}{N}\log\Big(\frac{d}{N}\Big)\right)  & \mbox{ otherwise} 
\end{cases}
\enspace.
\end{equation}

Now we turn to a solution of the sparsity equation for the $\ell_1^d$-norm. This equation has been solved in \cite[Lemma~4.2]{LM_reg_comp}, we recall this result.

\begin{Lemma} \label{lem:sparse_equa_LASSO}
If there exists $v\in\R^d$ such that $v \in t^*+(\rho/20)B_1^d$ and $|{\rm supp}(v)| \leqslant c\rho^2/r^2(\rho)$ then 
\begin{equation*}
\Delta(\rho)=\inf_{h \in \rho S_1^{d-1}\cap r(\rho)B_2^{d}} \sup_{g \in \Gamma_{t^*}(\rho)} \inr{h, g-t^*}\geqslant \frac{4\rho}{5}\enspace.
\end{equation*}
where $S_{1}^{d-1}$ is the unit sphere of the $\ell_1^d$-norm and $B_2^{d}$ is the unit Euclidean ball in $\R^d$.
\end{Lemma}
As a consequence, if $N\gtrsim s \log(ed/s)$ and if there exists a $s$-sparse vector in $t^*+(\rho/20)B_1^d$, Lemma~\ref{lem:sparse_equa_LASSO} and the choice of $r(\cdot)$ in \eqref{eq:r_function_LASSO} imply that for $\sigma = \norm{\zeta}_{L^{q_0}}$,
\begin{equation*}
\rho^* \sim_{L, q_0} \sigma s \sqrt{\frac{1}{N}\log\left(\frac{ed}{s}\right)} \mbox{ and } r^2(\rho^*)\sim \frac{\sigma^2 s}{N}\log\left(\frac{ed}{s}\right)
\end{equation*}then $\rho^*$ satisfies the sparsity equation and $r^2(\rho^*)$ is the rate of convergence of the LASSO for $\lambda \sim r^2(\rho^*)/\rho^*\sim \norm{\zeta}_{L^{q_0}}\sqrt{\log(ed/s)/N}$. But, this choice of $\lambda$ requires to know the sparsity parameter $s$ which is usually not available. That is the reason why we either need to choose a larger value for the $r(\cdot)$ function as in \cite{LM_reg1} -- this results in the suboptimal $\sqrt{\log(ed)/N}$ rates of convergence from Theorem~\ref{thm:intro-LASSO-est} --  or to use an adaptation step as section~\ref{sec:adap_lep_reg} -- this results in the better minimax rate $\sqrt{\log(ed/s)/N}$ achieved by the MOM LASSO. To get the latter one needs a final ingredient which is the computation of the radii $\rho_K$ and $\lambda\sim r^2(\rho_K)/\rho_K$. Let $K\in[N]$ and $\sigma = \norm{\zeta}_{L^{q_0}}$. The equation $K=c r(\rho_K)^2N$ is solved by
\begin{equation}\label{eq:LASSO_choice_rho_K}
 \rho_K\sim_{L,q_0} \frac{K}{\sigma} \sqrt{\frac{1}{N}\log^{-1}\left(\frac{\sigma^2 d}{K}\right)}
 \end{equation} for the $r(\cdot)$ function defined in \eqref{eq:r_function_LASSO}.
 Therefore,  \begin{equation}\label{eq:reg_param_lasso}
 \lambda \sim \frac{r^2(\rho_K)}{\rho_K} \sim_{L,q_0} \sigma \sqrt{\frac{1}{N}\log\left(\frac{e\sigma d}{\rho_K\sqrt{N}}\right)} \sim_{L,q_0} \sigma \sqrt{\frac{1}{N}\log\left(\frac{e\sigma^2 d}{K}\right)}\enspace.
 \end{equation}
The regularization parameter depends on the ``level of noise'' $\sigma$, the $L^{q_0}$-norm of $\zeta$. This parameter is unknown in practice. Nevertheless, it can be estimated and replaced by this estimator in the regularization parameter as in \cite[Sections~5.4 and 5.6.2]{MR3307991}.

The following result follows from Theorem~\ref{thm:LepskiReg} together with the computation of $\rho^*$, $r_Q(\cdot)$, $r_M(\cdot)$ and $r(\cdot)$ from the previous sections. 

\begin{Theorem}\label{theo:mom_lasso_sharp}
Grant Assumption~\ref{ass:examples_MOM}. The MOM-LASSO estimator $\hat{t}$ satisfies, with probability at least $1-c_1 \exp(-c_2 s \log(ed/s))$,  for every $1\leqslant p\leqslant 2$, 
\begin{equation*}
\norm{\hat{t} - t^*}_p\leqslant c_3\norm{\zeta}_{L_{q_0}}s^{1/p}\sqrt{\frac{1}{N}\log\left(\frac{ed}{s}\right)},
\end{equation*}where $(c_j)_{j=1}^3$ depends only on $\theta_0, \theta_m$ and $q_0$.
\end{Theorem}
In particular, Theorem~\ref{theo:mom_lasso_sharp} shows that, for our estimator contrary to the one in \cite{LugosiMendelson2016}, the sparsity parameter $s$ does not have to be known in advance in the LASSO case.
\proof
It follows from Theorem~\ref{thm:LepskiReg}, the computation of $r(\rho_K)$ from \eqref{eq:r_function_LASSO} and $\rho_K$ in \eqref{eq:LASSO_choice_rho_K} that with probability at least $1-c_0\exp(- cr(\rho_K)^2N/\overline{C})$, $\norm{\hat{t} - t^*}_1\leqslant \rho_{K^*}$ and $\norm{\hat t - t^*}_2\lesssim r(\rho_K)$.
The result follows since $\rho_{K^*}\sim\rho^* \sim_{L, q_0} \sigma s \sqrt{\frac{1}{N}\log\left(\frac{ed}{s}\right)}$
and $\norm{v}_p\leqslant \norm{v}_1^{-1+2/p}\norm{v}_2^{2-2/p}$ for all $v\in\R^d$ and $1\leqslant p\leqslant2$.
\endproof

Theoretical properties of MOM LASSO (cf. Theorem~\ref{theo:mom_lasso_sharp}) outperform those of LASSO (cf. Theorem~\ref{thm:intro-LASSO-est}) in several ways:
\begin{itemize}
	\item Estimation rates achieved by MOM-LASSO are the actual minimax rates $s \log(ed/s)/N$,  see \cite{BLT16}, while classical LASSO estimators achieve the rate $s \log(ed)/N$. This improvement is possible thanks to the adaptation step in MOM-LASSO. 
	\item  the probability deviation in  \eqref{eq:proba_lasso} is polynomial -- $1/N^{(q_0/2-1)}$ -- whereas it is exponentially small for MOM LASSO.  Exponential rates for LASSO hold only if $\zeta$ is subgaussian ($\norm{\zeta}_{L_p}\leqslant C \sqrt{p}\norm{\zeta}_{L_2}$ for all $p\geqslant2$).
%	, the probability deviation for LASSO is the one of MOM LASSO $1-c_2 \exp(-c_3 s \log(ed/s))$. But this requires the noise to be sub-gaussian in Theorem~\ref{theo:lasso_classi} whereas such a result is achieved under the only $L_{q_0}$ for $q_0>2$ assumption on $\zeta$ in Theorem~\ref{theo:mom_lasso},
	\item MOM LASSO is insensitive to data corruption by up to $s\log(ed/s)$ outliers while only one outlier can be responsible of a dramatic breakdown of the performance of LASSO (cf. Figure~\ref{fig:robustness}). Moreover, the informative data are only asked to have equivalent $L^2$ moments to the one of $P$ for the MOM LASSO whereas the properties of the LASSO are only known in the i.i.d. setup.
	\item Assumptions on $X$ are weaker for MOM LASSO than for LASSO. In the LASSO case, we assume that $X$ is subgaussian whereas for the MOM LASSO we assumed that the coordinates of $X$ have $C_0 \log(ed)$ moments and that it satisfies a $L^2/L^1$ equivalence assumption.
\end{itemize}

Let us now turn to the SLOPE case. The computation of the fixed point functions $r_Q(\cdot)$ and $r_M(\cdot)$ rely on \cite[Theorem~1.6]{shahar_gafa_ln} and the computation from \cite{LM_reg1}. Again, the SLOPE norm has a dual norm which is $1$-unconditional with respect to the canonical basis of $\R^d$, \cite[Definition~1.4]{shahar_gafa_ln}. Therefore, it follows from  \cite[Theorem~1.6]{shahar_gafa_ln}  that under Assumption~\ref{ass:shahar_theo16}, one has 
\begin{gather*}
\E \sup_{v\in \rho \cB \cap r B_2^d} \left|\sum_{i\in[N]}\eps_i \inr{v, X_i}\right|\leqslant c_2\sqrt{N}\ell^*(\rho \cB \cap r B_2^d)\enspace,\\
\E \sup_{v\in \rho \cB \cap r B_2^d} \left|\sum_{i\in[N]} \eps_i \zeta_i \inr{v, X_i}\right|\leqslant c_2\sigma\sqrt{N}\ell^*(\rho \cB \cap r B_2^d)\enspace,
\end{gather*}where $\cB$ is the unit ball of the SLOPE norm. Local Gaussian mean widths $\ell^*(\rho \cB \cap r B_2^d)$ are bounded from above in \cite[ Lemma~5.3]{LM_reg_comp}: $\ell^*(\rho \cB \cap r B_2^d)\lesssim \min\{C\rho,\sqrt{d}r\}$ when $\beta_j=C \sqrt{\log(ed)/j}$ for all $j\in[d]$  and computations of $r_M(\cdot)$ and $r_Q(\cdot)$ follow:
\begin{equation*}
r_Q^2(\rho) \lesssim_{L} \left\{
\begin{array}{cc}
0 & \mbox{ if }  N \gtrsim_L d
\\
\\
\frac{\rho^2}{N} & \mbox{ otherwise,}
\end{array}
\right.
\mbox{ and } \ \
r_M^2(\rho) \lesssim_{L,q,\delta} \left\{
\begin{array}{cc}
\|\xi\|_{L_q}^2 \frac{ d}{N} & \mbox{ if } \rho^2 N \gtrsim_{L,q,\delta} \|\xi\|_{L_q}^2 d^2
\\
\\
\|\xi\|_{L_q}\frac{\rho}{\sqrt{N}} & \mbox{ otherwise.}
\end{array}
\right.
\end{equation*}

The sparsity equation relative to the SLOPE norm has been solved in Lemma~4.3 from \cite{LM_reg1}.
\begin{Lemma} \label{lemma:delta-for-slope}
Let $1 \leqslant s \leqslant d$ and set ${\cal B}_s = \sum_{j \leqslant s} \beta_j/\sqrt{j}$. If $t^*$ is $\rho/20$ approximated (relative to the SLOPE norm) by an $s$-sparse vector and if $40{\cal B}_s \leqslant \rho/r(\rho)$ then $\Delta(\rho) \geqslant 4\rho/5$.
\end{Lemma}

For  $\beta_j\leqslant C \sqrt{\log(ed/j)}$, one may verify that $\cB_s= \sum_{j\leqslant s}\beta_j/\sqrt{j} \lesssim C \sqrt{s \log(ed/s)}$. Hence, the condition $\cB_s\lesssim \rho/r(\rho)$ holds when $N \gtrsim_{L,q_0} s \log(ed/s)$ and $\rho \gtrsim_{L,q_0} \|\xi\|_{L_q} \frac{s }{\sqrt{N}} \log\Big(\frac{ed}{s}\Big)$. Hence, it follows from Lemma~\ref{lemma:delta-for-slope} that $\Delta(\rho)\geqslant 4\rho/5$ when there is an $s$-sparse vector in $t^*+(\rho/20) B_\Psi$; therefore, one may apply Theorem~\ref{thm:RBRhoEstPen} for the choice of the regularization parameter: $\lambda \sim r^2(\rho)/\rho \sim_{L,q,\delta} \|\xi\|_{L_q}/\sqrt{N}$.

Now, the final ingredient is to compute the $\rho_K$ solution to $K = c r(\rho_K)^2 N$. It is straightforward to check that $\rho_K\sim K/(\sigma \sqrt{N})$ and still $\lambda \sim r^2(\rho_K)/\rho_K \sim_{L,q,\delta} \|\xi\|_{L_q}/\sqrt{N}$.

The following result follows from Theorem~\ref{thm:LepskiReg} together with the computation of $\rho^*, \rho_K$, $r_Q(\cdot)$, $r_M(\cdot)$ and $r(\cdot)$ above. Its proof is similar to the one of Theorem~\ref{theo:mom_lasso_sharp} and is therefore omitted. 

\begin{Theorem}\label{theo:mom_slope_sharp}
Grant Assumption~\ref{ass:examples_MOM}. The MOM-SLOPE estimator $\hat{t}$ satisfies, with probability at least $1-c_1 \exp(-c_2 s \log(ed/s))$, 
\begin{equation*}
\norm{\hat{t} - t^*}_2^2\leqslant c_3\norm{\zeta}_{L_{q_0}}^2 \frac{s}{N}\log\left(\frac{ed}{s}\right),
\end{equation*}where $(c_j)_{j=1}^3$ depends only on $\theta_0, \theta_m$ and $q_0$.
\end{Theorem}
MOM-SLOPE has the same advantages upon SLOPE as MOM-LASSO upon LASSO. Those improvements are listed below Theorem~\ref{theo:mom_lasso_sharp} and will not be repeated. The only difference is that SLOPE, unlike LASSO, already achieves the minimax rate $s \log(ed/s)/N$ .

% subsection statistical_performance_of_the_mom_lasso_mom_slope_and_mom_trace_norm (end)

\section{Learning without regularization} % (fold)
\label{sec:learning_without_regularization_and_minimax_optimality}
All the results from the previous sections also apply in the setup of learning with no regularization which is the framework one should consider when there is no a priori known structure on the oracle.

We consider the learning problem with no regularization. In this setup, we may use both minmaximization or maxminimization estimators
\begin{equation}\label{eq:MOM-ERM}
 \ERM{K} \in \argmin_{f\in F}\sup_{g\in F} T_K(g, f) \mbox{ and }  \widehat{g}_K \in \argmax_{g\in F}\inf_{f\in F} T_K(g, f)
 \end{equation}where $T_K(g, f) = \MOM{K}{\ell_f-\ell_g}$.

 We show below that $\widehat f_K$ and $\widehat g_K$ are efficient procedures even in situations where the dataset is corrupted by outliers. The case $K=1$ corresponds to the classical ERM: $\widehat f_1 = \widehat g_1 \in\argmin_{f\in F}P_N\ell_f$ which can only be trusted when used with a ``clean dataset''. 

 Indeed, the ideal setup for ERM is the subgaussian (and convex) framework: that is for a convex class $F$ of functions,  i.i.d. data $(X_i, Y_i)_{i=1}^N$ having the same distribution as $(X, Y)$ and such that for some $L>0$ and all $f,g\in F$, 
 \begin{equation}\label{eq:subgaussian_assum}
   \norm{Y}_{\psi_2}<\infty \mbox{ and } \norm{g(X) - f(X)}_{\psi_2}\leqslant L\norm{g(X)-f(X)}_{L_2}.
   \end{equation} When $F$ satisfies the right-hand side of \eqref{eq:subgaussian_assum}, we say that $F$ is a $L$-subgaussian class. It is proved in \cite{LM13} that in this setup the ERM is an optimal minimax procedure (cf. Theorem~A$^{\prime}$ from  \cite{LM13} recalled in Theorem~\ref{thm:minimax_LM13} below). 

But first, we need a version of the two theorems \ref{thm:RBRhoEstPen} and \ref{thm:LepskiReg} valid for $\widehat f_K$ and  $\widehat g_K$ (that is for the learning problem with no regularization). Let us first introduce the set of assumptions we use and for the sake of shortness we consider the simplification introduced in  Remark~\ref{rem:equiv_L2_norm}. Then, we will introduce the two fixed points driving the statistical properties of $\widehat f_K$ and $\widehat g_K$. 

\begin{Assumption}\label{ass:equivalent_L2}For all $i\in\cI$ and $f\in F$,
$\norm{f(X_i)-f^*(X_i)}_{L^2} = \norm{f(X)-f^*(X)}_{L_2}$,  $$\norm{Y_i-f(X_i)}_{L^2} = \norm{Y-f(X)}_{L_2}, {\rm var}((Y-f^*(X))(f(X)-f^*(X)))\leqslant \theta_m^2 \norm{f(X)-f^*(X)}_{L^2}^2$$ and $\norm{f(X_i)-f^*(X_i)}_{L^2}\leqslant \theta_0 \norm{f(X_i) - f^*(X_i)}_{L^1}$.
\end{Assumption}

The two fixed points associated to this problem are $r_Q(\rho, \gamma_Q)$ and  $r_M(\rho, \gamma_M)$ as in Definition~\ref{def:the-three-parameters-Reg} for $\rho=\infty$:
\begin{align*}
r_Q(\gamma_Q)& = \inf\left\{r>0 : \forall J\subset \cI, |J|\geqslant \frac{N}2,\;\E \sup_{f \in F: \norm{f-f^*}_{L^2_P}\leqslant r} \left|\sum_{i\in J} \eps_i (f-\bayes)(X_i)\right| \leqslant  \gamma_Q|J| r \right\}\enspace,\\
r_M(\gamma_M)&=  \inf\left\{r>0 : \forall J\subset \cI, |J|\geqslant \frac{N}2,\;\E \sup_{f \in F: \norm{f-f^*}_{L^2_P}\leqslant r} \left|\sum_{i\in J} \eps_{i} \zeta_i (f-\bayes)(X_i)\right| \leqslant \gamma_M |J|r^2 \right\}\enspace,
\end{align*}
and let $r^* = r^*(\gamma_Q,\gamma_M) =  \max\{r_Q(\gamma_Q),r_M(\gamma_M)\}$.

\begin{Theorem}\label{thm:learning_withour_reg_1}
Grant Assumptions \ref{ass:equivalent_L2} and let $r_Q(\gamma_Q)$, $r_M(\gamma_M)$ and $r^*$ be defined as above for $\gamma_Q=(384\theta_0)^{-1}$, $\gamma_M = \eps/192$ and $\eps = 1/(32\theta_0^2)$. Assume that $N\geqslant 384 \theta_0^2$ and $|\cO|\leqslant N/(768\theta_0^2)$. Let $K^*$ denote the smallest integer such that $K^*\geqslant  N\eps^2 (r^*)^2/(384 \theta_m^2)$.  Then, for all $K\in[\max(K^*, 8 |\cO|), N/(96 \theta_0^2)]$,  with probability larger than  $1-2\exp(-7K/9216)$, 
the estimators $\ERM{K}$ and $\widehat{g}_K$ defined in \eqref{eq:MOM-ERM} satisfy
\[
\norm{\widehat{g}_K-\bayes}_{L^2_P}, \norm{\ERM{K}-\bayes}_{L^2_P}\leqslant \frac{\theta_m}{\eps}\sqrt{\frac{384 K}{N}}\quad \text{and} \quad R(\widehat{g}_K), R(\ERM{K})\leqslant R(\bayes)+(1+2\eps)\frac{384\theta_m^2 K}{\eps^2 N}\enspace.
\]
\end{Theorem}

Moreover, one can choose adaptively $K$ via Lepski's method. We will do it only for the maxmin estimators $\widehat{g}_K$. Similar result hold for the minmax estimators $\widehat{f}_K$ from straightforward modifications (the same as in Section~\ref{sec:adap_lep_reg}).  Define the confidence regions: for all $J\in[K]$ and  $g\in F$, 
\begin{equation*}
\hat R_J=\left\{g\in F : \fC_{J}(g)\geqslant \frac{-384 \theta_m^2 J}{\eps N}\right\} \mbox{ where }  
\fC_{J}(g) = \inf_{f\in F} T_{J}(g, f) 
\end{equation*}and $T_J(g, f) = \MOM{J}{\ell_f-\ell_g}$ for all $f,g\in F$. Next, for all $J\in [\max(K^*, 8 |\cO|), N/(96 \theta_0^2)]$ , let 
\begin{equation*}
\hat K=\inf\left\{K\in \left[\max(K^*, 8 |\cO|), \frac{N}{96 \theta_0^2}\right] : \bigcap_{J=K}^{K_2}\hat R_J\ne \emptyset\right\} \mbox{ and } \widehat{g}\in \bigcap_{ J=\hat K}^{K_2}\hat R_J\enspace.
\end{equation*}
The following theorem shows the performance of the resulting estimator.
\begin{Theorem}\label{thm:learning_without_reg_lepsky}
 Grant Assumption~\ref{ass:equivalent_L2}. For $\eps=1/(32\theta_0^2)$ and all $K\in [\max(K^*, 8 |\cO|), N/(96 \theta_0^2)]$, with probability larger than $1-2\exp(-K/2304)$,
 \begin{equation*}
  \norm{\widehat{g}-\bayes}_{L^2_P}\leqslant \frac{\theta_m}{\eps}\sqrt{\frac{384 K}{N}},\qquad R(\widehat{g})\leqslant R(\bayes)+ (1+2\eps)\frac{384\theta_m^2 K}{\eps^2 N}\enspace.
 \end{equation*}
\end{Theorem}

The proofs of Theorem~\ref{thm:learning_withour_reg_1} and \ref{thm:learning_without_reg_lepsky} essentially follow the one of Theorem~\ref{thm:RBRhoEstPen} and \ref{thm:LepskiReg}. We will only sketch the proof for the maxmin estimator $\widehat g_K$ given that we already studied the minmax estimators in the regularized setup in Section~\ref{sec:proofs}.

\textit{Proof of Theorem~\ref{thm:learning_withour_reg_1}.}
It follows from Lemma~\ref{lem:UBQPReg} and Lemma~\ref{lem:proc_multiplicatif} for $\rho=\infty$ that there exists an event $\Omega(K)$ such that $\bP(\Omega(K))\geqslant 1-2\exp\left(- 7K/9216\right)$ and,  on $\Omega(K)$, for all $f\in F$,
\begin{enumerate}
	\item  if $\norm{f-\bayes}_{L^2_P}\geqslant r_Q(\gamma_Q)$ then
\begin{equation}\label{eq:quad_proc_nonReg}
 Q_{1/4,K}((f-\bayes)^2)\geqslant \frac1{(4\theta_0)^2}\norm{f-\bayes}^2_{L^2_P}\enspace,
\end{equation}
\item there exists $3K/4$ block $B_k$ with $k\in \cK$, for which
\begin{equation}\label{eq:multi_proc_nonReg}
|(P_{B_k}-\overline{P}_{B_k})[2\zeta(f-\bayes)]|\leqslant \eps\max\left(r_M^2(\gamma_M), \frac{384\theta_m^2}{\eps^2}\frac{K}{N}, \norm{f-f^*}_{L_p^2}^2\right)\enspace. 
\end{equation}
\end{enumerate}
Moreover, it follows from Assumption~\ref{ass:equivalent_L2} that for all $k\in\cK$, 
$\overline{P}_{B_k}[\zeta(f-\bayes)] = P[\zeta(f-\bayes)]$ and $P[2\zeta(f-\bayes)]\leqslant 0$ because of the  convexity of $F$ and the nearest point theorem. Therefore, on the event $\Omega(K)$, for all $f\in F$, 
\begin{align}
\label{eq:UBPMNC-nonReg} Q_{3/4,K}(2\zeta(f-\bayes))&\leqslant \eps\max\left(r_M^2(\gamma_M), \frac{384 \theta_m^2}{\eps^2}\frac{K}{N}, \norm{f-f^*}_{L_p^2}^2\right)
\end{align}and 
\begin{align}\label{eq:oracle-ineg-nonReg}
\nonumber &P[-2\zeta(f-\bayes)]\leqslant P_{B_k}[-2\zeta(f-\bayes)]+\eps\max\left(r_M^2(\gamma_M), \frac{384 \theta_m^2}{\eps^2}\frac{K}{N}, \norm{f-f^*}_{L_p^2}^2\right)\\ 
\nonumber &\leqslant Q_{1/4,K}[(f-\bayes)^2-2\zeta(f-\bayes)]+\eps\max\left(r_M^2(\gamma_M), \frac{384 \theta_m^2}{\eps^2}\frac{K}{N}, \norm{f-f^*}_{L_p^2}^2\right)\\
&\leqslant T_{K}(\bayes,f)+\eps\max\left(r_M^2(\gamma_M), \frac{384 \theta_m^2}{\eps^2}\frac{K}{N}, \norm{f-f^*}_{L_p^2}^2\right)\enspace.
\end{align}

Let us place ourself on the event $\Omega(K)$ and let $r_K$ be such that  $r_K^2 =  384 \theta_m^2 K / (\eps^2 N)$. Given that $r_K\geqslant r^*$, it follows from \eqref{eq:quad_proc_nonReg} and \eqref{eq:UBPMNC-nonReg} that if $f\in F$ is such that $\norm{f-f^*}_{L_P^2}\geqslant r_K$ then 
\begin{equation}\label{eq:final_nonReg}
T_K(f, f^*)\leqslant Q_{3/4, K}(2 \zeta(f-f^*)) - Q_{1/4}((f-f^*))\leqslant \left(\eps - \frac{1}{16 \theta_0^2}\right)\norm{f-f^*}_{L^2_P}^2\leqslant \left(\frac{-1}{32 \theta_0^2}\right)\norm{f-f^*}_{L^2_P}^2
\end{equation}for $\eps = 1/(32\theta_0^2)$ and if $\norm{f-f^*}_{L_P^2}\leqslant r_K$ then $T_K(f, f^*)\leqslant Q_{3/4, K}(2 \zeta(f-f^*))\leqslant \eps r_K^2$. In particular,
\begin{equation*}
\fC_K(f^*) = \inf_{f\in F} T_K(f^*, f) = -\sup_{f\in F} T_K(f, f^*) \geqslant -\eps r_K^2
\end{equation*} and since $\fC_K(\hat g_K)\geqslant \fC_K(f^*)$ one has $\fC_K(\hat g_K)\geqslant - \eps r_K^2$. On the other hand, we have $\fC_K(\hat g_K) = \inf_{f\in F} T_K(\hat g_K, f)\leqslant T_K(\hat g_K, f^*)$. Therefore, $ T_K(\hat g_K, f^*)\geqslant  -\eps r_K^2$. But, we know from \eqref{eq:final_nonReg} that if $g\in F$ is such that $\norm{g-f^*}_{L^2_P}> \sqrt{32\eps} \theta_0 r_K$ then $T_K(g, f^*)\leqslant  (-1/(32\theta_0^2))\norm{g-f^*}_{L^2_P}^2<-\eps r_K^2$. Therefore, one necessarily have $\norm{\hat g_K-f^*}_{L^2_P}\leqslant \sqrt{32\eps} \theta_0 r_K = r_K$.

The oracle inequality now follows from \eqref{eq:oracle-ineg-nonReg}:
\begin{align*}
R(\hat g_K)-R(f^*) &= \norm{\hat g_K-f^*}_{L^2_P}^2 + P[-2\zeta(\hat g _K - f^*)]
\leqslant r_K^2  + T_K(f^*, \hat g_K) + \eps r_K^2 \leqslant (1+2\eps)r_K^2\enspace.   
\end{align*}
\endproof

\textit{Proof of Theorem~\ref{thm:learning_without_reg_lepsky}.}
Consider the same notations as in the proof of Theorem~\ref{thm:learning_withour_reg_1} and denote $K_2 = N/(96 \theta_0^2)$. It follows from the proof of Theorem~\ref{thm:learning_withour_reg_1}, that with probability larger than $1-2 \sum_{J=K}^{K_2}\exp(-7J/9216)$, for all $J\in[K, K_2]$, $\fC_J(f^*)\geqslant -\eps r_J^2$ therefore, $f^*\in \hat R_J$ and so $\hat K\leqslant K$ . The latter implies that $\widehat g\in \hat R_K$which, by using the same argument as in the end of the proof of Theorem~\ref{thm:learning_withour_reg_1} implies that $\norm{\widehat g-f^*}_{L^2_P}\leqslant r_K$ and then $R(\widehat g)-R(f^*)\leqslant (1+2\eps) r_K$.
\endproof

\textbf{Example: Ordinary least squares.} Let us consider the case where $F=\{\inr{\cdot, t} : t\in\R^d\}$ is the set of all linear functionals indexed by $\R^d$. We assume that for all $i\in\cI$ and $t\in\R^d$,
\begin{enumerate}
	\item  $\E\inr{X_i, t}^2 = \E \inr{X, t}^2$,
	\item $\E(Y_i-\inr{X_i, t})^2 = \E(Y-\inr{X, t})^2$,
	\item $\E(Y-\inr{X, t^*})^2\inr{X, t}^2\leqslant \theta_m^2 \E\inr{X, t}^2$,
	\item $\sqrt{\E \inr{X, t}^2}\leqslant \theta_0 \E |\inr{X, t}|$.
\end{enumerate}
Let us now compute the fixed points $r_Q(\gamma_Q)$ and $r_M(\gamma_M)$. The proof essentially follows from Example~1 in \cite{k:05}. Let $J\subset \cI$ be such that $|J|\geqslant N/2$. Denote by $V\subset\R^d$ the smallest linear span containing almost surely $X$. Let $\varphi_1, \cdots, \varphi_D$ be an orthonormal basis of $V$ with respect to the Hilbert norm $\norm{t}=\E \inr{X, t}^2$. It follows from Cauchy-Schwartz inequality that
\begin{align*}
&\E \sup_{f \in F: \norm{f-f^*}_{L^2_P}\leqslant r} \left|\sum_{i\in J} \eps_i (f-\bayes)(X_i)\right| = \E \sup_{\sum_{j=1}^D \theta_j^2\leqslant r^2} \left|\sum_{j=1}^D \theta_j\sum_{i\in J}\eps_i\inr{X_i, \varphi_j} \right|\leqslant r \E \left(\sum_{j=1}^D\left(\sum_{i\in J}\eps_i \inr{X_i, \varphi_j}\right)^2 \right)^{1/2}\\
&\leqslant r \sqrt{\sum_{j=1}^D \sum_{i\in J}\E \inr{X_i, \varphi_j}^2} = r \sqrt{D |J|}.
\end{align*}As a consequence, $r_Q(\gamma_Q) = 0$ if $\gamma_Q |J|\geqslant \sqrt{D |J|}$, i.e. if $\gamma_Q\geqslant \sqrt{D/|J|}$. Using the same arguments as above, we  have
\begin{align*}
 &\E \sup_{f \in F: \norm{f-f^*}_{L^2_P}\leqslant r} \left|\sum_{i\in J} \eps_i\zeta_i (f-\bayes)(X_i)\right| \leqslant r \sqrt{\sum_{j=1}^D \sum_{i\in J} \E \zeta_i \inr{X_i, \varphi_j}^2}\leqslant r \theta_m \sqrt{D |J|}.
 \end{align*} Therefore, $r_M(\gamma_M)\leqslant (\theta_m/\gamma_M)\sqrt{D/|J|}\leqslant (\theta_m/\gamma_M)\sqrt{2D/N}$ and $K^*=D$. 

 Now, it follows from Theorem~\ref{thm:learning_without_reg_lepsky}, that if $N\geqslant 2(384\theta_0)^2 D$ and $|\cO|\leqslant N/(768\theta_0^2)$ then the MOM OLS with adaptively chosen number of blocks $K$ is such that for all $K\in\left[\max\left(D, 8 |\cO|\right), N/(96 \theta_0^2)\right]$, with probability at least $1-2\exp(-K/2304)$, 
 \begin{equation}\label{eq:results_MOM_OLS}
 \sqrt{\E \inr{\hat t - t^*, X}^2}\leqslant \frac{\theta_m}{\eps} \sqrt{\frac{384 K}{N}}. 
 \end{equation}
  A consequence of \eqref{eq:results_MOM_OLS}, is that if the number of outliers is less than $D/8$ then the MOM OLS recovers the classical $D/N$ rate of convergence for the means square error. This happens with probability at least $1-2 \exp(-D/2304)$, that is with an exponentially large probability. This is a remarkable fact given that we only made assumptions on the $L^2$ moments of the design $X$. Moreover, this result is obtained under the only assumption on the informative data that they have equivalent $L^2$ moments to the one of the distribution of interest $P$. Therefore, only very little information on $P$ needs to be brought to the statistician via the data; moreover those data can be corrupted up to $D/8$ complete outliers. Finally, note that we did not assume isotropicity of the design $X$ to obtain \eqref{eq:results_MOM_OLS}. Therefore,  \eqref{eq:results_MOM_OLS} holds even for very degenerate design $X$ and the price we pay is the true dimension of $X$ that is of the dimension of the smallest linear span containing almost surely $X$ not the one of the whole space $\R^d$.

\section{Minimax optimality of Theorem~\ref{thm:RBRhoEstPen}, \ref{thm:LepskiReg}, \ref{thm:learning_withour_reg_1} and \ref{thm:learning_without_reg_lepsky}} % (fold)
\label{sec:minimax_optimality_of_}
The aim of this section is to show that the rates obtained in Theorems~\ref{thm:RBRhoEstPen}, \ref{thm:LepskiReg}, \ref{thm:learning_withour_reg_1} and \ref{thm:learning_without_reg_lepsky} are optimal in a minimax sense. To that end we recall a minimax lower bound result from \cite{LM13}.

\begin{Theorem}[Theorem~A$^{\prime}$ in \cite{LM13}]\label{thm:minimax_LM13}
    There exists an absolute constant $c_0$ for which the following holds. Let $X$ be a random variable taking values in $\cX$. Let $F$ be a class of functions such that $\E f^2(X)<\infty$. Assume that $F$ is star-shaped around one of its point (i.e. there exists $f_0\in F$ such that for all $f\in F$ the segment $[f_0, f]$ belongs to $F$). Let $\zeta$ be a centered real-valued Gaussian variable with variance $\sigma$ independent of $X$ and for all $f^*\in F$ denote by $Y^{f^*}$ the target variable
    \begin{equation}\label{eq:model_gaussian_noise}
Y^{f^*} = f^*(X)+\zeta.
\end{equation}

Let $0<\delta_N<1$ and $r_N^2>0$. Let $\hat f_N$ be a statistics (i.e. a measurable function from $(\cX\times \R)^N$ to $L^2(P_X)$ where $P_X$ is the probability distribution of $X$). Assume that $\hat f_N$ is such that for all $f^*\in F$, with probability at least $1-\delta_N$, 
\begin{equation*}
\norm{\hat f_N(\cD) - f^*}_{L^2_P}^2 = R(\hat f_N(\cD)) -  R(f^*)\leqslant r_N^2
\end{equation*} where $\cD=\{ (X_i, Y_i) : i\in[N]\}$ is a set of $N$ i.i.d. copies of $(X, Y^{f^*})$. Then, necessarily, one has 
\begin{equation*}
r^2_N\geqslant \min\left(c_0 \sigma^2 \frac{\log(1/\delta_N)}{N}, \frac{1}{4}{\rm diam}(F, L^2(P_X))\right)
\end{equation*}where ${\rm diam}(F, L^2(P_X))$ denotes the $L^2(P_X)$ diameter of $F$.
\end{Theorem} 

Theorem~\ref{thm:minimax_LM13} proves that if the statistical model \eqref{eq:model_gaussian_noise} holds then there is a strong connexion between the deviation parameter $\delta_N$ and the uniform rate of convergence $r_N^2$ over $F$ : the smaller $\delta_N$, the larger $r^2_N$. We now use this result to prove that Theorems~\ref{thm:RBRhoEstPen}, \ref{thm:LepskiReg}, \ref{thm:learning_withour_reg_1} and \ref{thm:learning_without_reg_lepsky} are essentially optimal. 

In Theorems~\ref{thm:learning_withour_reg_1} and \ref{thm:learning_without_reg_lepsky}, the deviation bounds are $1-c_1\exp(-c_2K)$ and the residual terms in the $L^2_P$ (to the square) estimation rates are like $c_3 K/N$. Therefore, setting $\delta_N = c_1\exp(-c_2K)$ then Theorem~~\ref{thm:minimax_LM13} proves that no procedure can do better than 
\begin{equation*}
\min\left(c_0 \sigma^2 \frac{\log(1/\delta_N)}{N}, \frac{1}{4}{\rm diam}(F, L^2(P_X))\right) = \min\left(c_4\sigma^2 \frac{K}{N}, \frac{1}{4}{\rm diam}(F, L^2(P_X))\right).
\end{equation*} Given  that one can obviously bound from above the performance of $\widehat f_K$ and $\widehat{g}_K$ as well as those of $\widehat f$ and $\widehat{g}$ in Theorems~\ref{thm:learning_withour_reg_1} and \ref{thm:learning_without_reg_lepsky} by the $L^2_P$-diameter of $F$ (because $f^*$ and those estimators are in $F$), then the result of Theorem~\ref{thm:learning_withour_reg_1} and \ref{thm:learning_without_reg_lepsky} are optimal even in the very strong Gaussian setup with i.i.d. data satisfying a Gaussian regression model like \eqref{eq:model_gaussian_noise}. The remarkable point is that Theorem~\ref{thm:learning_withour_reg_1} and \ref{thm:learning_without_reg_lepsky} have been obtained under much weaker assumptions than those considered in Theorem~\ref{thm:minimax_LM13} since outliers may corrupt the dataset, the noise and the design do not have to be independent, the informative data are only assumed to have a $L^2$ norm equivalent to the one of $P$ and may therefore be heavy tailed. 

Given the form of the deviation bounds in  Theorems~\ref{thm:RBRhoEstPen} and \ref{thm:LepskiReg} and given that $r(\rho_K)\sim K/N$ and that $r(2\rho_K)\sim K/N$ (if one assumes a weak regularity assumption on the class $F$) then the same conclusions hold for Theorems~\ref{thm:RBRhoEstPen} and \ref{thm:LepskiReg}: there is no procedure doing better than the MOM estimators even in the very good framework of Theorem~\ref{thm:minimax_LM13}.

% section minimax_optimality_of_ (end)

% subsection cyclic_coordinate_descent (end)

% subsection douglas_racheford_admm (end)

% subsection algorithms_and_their_ (end)

% subsection data_generating_process_and_corruption_by_outliers (end)

% section simulation_study (end)
\begin{footnotesize}
\bibliographystyle{plain}
\bibliography{biblio}
\end{footnotesize}

\end{document}

%% file: packages_notes.tex
\usepackage{graphicx}
\usepackage{color}
\usepackage{amsmath}
\usepackage{amssymb}
\usepackage{amscd}
\usepackage{bbm}
\usepackage{tikz}
\usetikzlibrary{patterns}
\usepackage{hyperref}
\hypersetup{
    bookmarks=true,         % show bookmarks bar?
    colorlinks=true,       % false: boxed links; true: colored links
    linkcolor=red,          % color of internal links (change box color with linkbordercolor)
    citecolor=green,        % color of links to bibliography
    filecolor=magenta,      % color of file links
    urlcolor=cyan           % color of external links
}

\usepackage[utf8]{inputenc}
\usepackage{amsthm}
\usepackage{bbm} %indicatrice
\usepackage[linesnumbered,lined,boxed,commentsnumbered]{algorithm2e}

\usepackage{marginnote}

\newtheorem{Theorem}{Theorem}
\newtheorem{Assumption}{Assumption}
\newtheorem{Definition}{Definition}
\newtheorem{Lemma}{Lemma}
\newtheorem{Proposition}{Proposition}

\newtheorem{Remark}{Remark}

%% file: commandes_guillaume.tex
\newcommand{\inr}[1]{\bigl< #1 \bigr>}

\newcommand{\norm}[1]{\left\|#1\right\|}%
\newcommand{\an}{{\mbox{ and }}}

\newcommand{\pa}[1]{\left(#1\right)}
\newcommand{\cro}[1]{\left\{#1\right\}}

\newcommand\eps{\epsilon}

\def \endproof
{{\mbox{}\nolinebreak\hfill\rule{2mm}{2mm}\par\medbreak}}

\DeclareMathOperator*{\argmax}{argmax}
\DeclareMathOperator*{\argmin}{argmin}

\newcommand{\ERM}{\hat f_n^{ERM}}

\def\ds1{\textrm{1\kern-0.25emI}} %{\mathds{1}}

\newcommand \E{\mathbb{E}}
\newcommand \R{\mathbb{R}}

\newcommand \cB{{\cal B}}
\newcommand \cC{{\cal C}}
\newcommand \cD{{\cal D}}

\newcommand \cF{{\cal F}}
\newcommand \cG{{\cal G}}

\newcommand \cI{{\cal I}}

\newcommand \cK{{\cal K}}
\newcommand \cL{{\cal L}}

\newcommand \cN{{\cal N}}
\newcommand \cO{{\cal O}}
\newcommand \cP{{\cal P}}
\newcommand \cQ{{\cal Q}}
\newcommand \cR{{\cal R}}

\newcommand \cX{{\cal X}}

\newcommand \bE{{\mathbb E}}

\newcommand \bN{{\mathbb N}}

\newcommand \bP{{\mathbb P}}

\newcommand \bX{{\mathbb X}}
\newcommand \bY{{\mathbb Y}}
\newcommand \bZ{{\mathbb Z}}
\newcommand \fC{{\mathfrak C}}

\newcommand{\hf}{ {\hat f}}

\newcommand{\Crit}[2]{\cC_{#1}\big(#2\big)}
\newcommand{\bayes}{f^*}
\newcommand{\MOM}[2]{\text{MOM}_{#1}\big(#2\big)}
\renewcommand{\ERM}[1]{\widehat{f}_{#1}}

\newcommand{\param}[1]{\theta_{#1}}
\newcommand{\cabs}[1]{c_{#1}}

%% file: dessin_1.tex
\begin{figure}[h]
\centering
\begin{tikzpicture}[scale=0.3]
\draw (0,0) circle (8cm);
\draw (2,0) node {$f^*$};
\filldraw (0,0) circle (0.2cm);
\draw (-10,0) -- (0,10) -- (10,0) -- (0,-10) -- (-10,0);
\draw[style = dashed] (-13,-4.2) -- (13,4.2);
\draw[style = dashed] (-13,4.2) -- (13,-4.2);
\draw[style = dashed] (-4.2,-13) -- (4.2,13);
\draw[style = dashed] (4.2,-13) -- (-4.2,13);
%%%%%%%%%%%%%%%%%%%%%
\fill[pattern = north east lines, pattern color=blue] (-4.2,13) -- (-2.47,7.64) -- (0, 10) -- (2.47,7.64) -- (4.2,13) -- cycle;
\fill[pattern = horizontal lines, pattern color=green] (0, 10) -- (-2.47,7.64) arc (110:70:7.2) -- cycle;
\fill [ pattern= north east lines, pattern color=blue] (4.2,13) -- (2.47,7.64) -- (7.64, 2.47) -- (13,4.2) -- (13,13) -- cycle;

\fill[pattern = north east lines, pattern color=blue] (4.2,-13) -- (2.47,-7.64) -- (0, -10) -- (-2.47,-7.64) -- (-4.2,-13) -- cycle;
\fill[pattern = horizontal lines, pattern color=green] (0, -10) -- (2.47,-7.64) arc (-70:-110:7.2) -- cycle;
\fill [ pattern= north east lines, pattern color=blue] (-4.2,-13) -- (-2.47,-7.64) -- (-7.64, -2.47) -- (-13,-4.2) -- (-13,-13) -- cycle;

\fill[pattern = north east lines, pattern color=blue] (13, -4.2) -- (7.64, -2.47) -- (10, 0) -- (7.64, 2.47) -- (13, 4.2) -- cycle;
\fill[pattern = horizontal lines, pattern color=green] (10, 0) -- (7.64, -2.47) arc (-20:20:7.2) -- cycle;
\fill [ pattern= north east lines, pattern color=blue] (13, -4.2) -- (7.64, -2.47) -- (2.47, -7.64) -- (4.2, -13) -- (13,-13) -- cycle;

\fill[pattern = north east lines, pattern color=blue] (-13, 4.2) -- (-7.64, 2.47) -- (-10, 0) -- (-7.64, -2.47) -- (-13, -4.2) -- cycle;
\fill[pattern = horizontal lines, pattern color=green] (-10, 0) -- (-7.64, -2.47) arc (200:160:7.2) -- cycle;
\fill [ pattern= north east lines, pattern color=blue] (-13, 4.2) -- (-7.64, 2.47) -- (-2.47, 7.64) -- (-4.2, 13) -- (-13,13) -- cycle;

%%%%%%%%%%%%%%%%%%%%%%
\filldraw[color=red, very thick, fill = red, opacity = 0.2] (4.2,13) -- (2.47,7.64) -- (7.64,2.47) -- (13,4.2);
\filldraw[color=red, very thick, fill = red, opacity = 0.2] (4.2,13) -- (13,13) -- (13,4.2);
\filldraw[color=red, very thick, fill = red, opacity = 0.2] (-4.2,-13) -- (-2.47,-7.64) -- (-7.64,-2.47) -- (-13,-4.2);
\filldraw[color=red, very thick, fill = red, opacity = 0.2] (-4.2,-13) -- (-13,-13) -- (-13,-4.2);
\filldraw[color=red, very thick, fill = red, opacity = 0.2] (-4.2,13) -- (-2.47,7.64) -- (-7.64,2.47) -- (-13,4.2);
\filldraw[color=red, very thick, fill = red, opacity = 0.2] (-4.2,13) -- (-13,13) -- (-13,4.2);
\filldraw[color=red, very thick, fill = red, opacity = 0.2] (4.2,-13) -- (2.47,-7.64) -- (7.64,-2.47) -- (13,-4.2);
\filldraw[color=red, very thick, fill = red, opacity = 0.2] (4.2,-13) -- (13,-13) -- (13,-4.2);
%%%%%%%%%%%%%%%%%%
\filldraw[color=blue, very thick, fill = blue, opacity = 0.1] (-4.2,13) -- (-2.47,7.64) arc (110:70:7.2) -- (4.2,13) -- cycle;
\filldraw[color=blue, very thick, fill = blue, opacity = 0.1] (-4.2,-13) -- (-2.47,-7.64) arc (-110:-70:7.2) -- (4.2,-13);
\filldraw[color=blue, very thick, fill = blue, opacity = 0.1] (-13,4.2) -- (-7.64,2.47) arc (160:200:7.2) -- (-13,-4.2);
\filldraw[color=blue, very thick, fill = blue, opacity = 0.1] (13,4.2) -- (7.64,2.47) arc (20:-20:7.2) -- (13,-4.2);
%%%%%%%%%%%%%
\draw (8,8) node {$R>M$};
\draw (-8,8) node {$R>M$};
\draw (8,-8) node {$R>M$};
\draw (-8,-8) node {$R>M$};
\draw (0,11) node {$Q>M$};
\draw (-11,2) node {$Q>M$};
\draw (11,2) node {$Q>M$};
\draw (0,-11) node {$Q>M$};
%%%%%%%%%%%%
\draw (-3.1,3.3) node {$F_1^{(\kappa)}$};
\draw (0.4,9) node {$F_2^{(\kappa)}$};
\draw (8.2,10) node {$F_3^{(\kappa)}$};
\end{tikzpicture}
\caption{Partition $\{F_1^{(\kappa)}, F_2^{(\kappa)}, F_3^{(\kappa)}\}$ of $F$ and the control of the multiplier MOM process by either the quadratic MOM process (the ``$Q>M$'' part) or the regularization term (the ``$R>M$'' part).}
\label{fig:partition_set_F_positive_excess_loss}
\end{figure}